\documentclass{article}

\RequirePackage[OT1]{fontenc}
\RequirePackage[colorlinks,citecolor=blue,urlcolor=blue]{hyperref}
\usepackage{amssymb}
\usepackage{graphicx}
\usepackage{bbm}
\usepackage{mathrsfs}
\usepackage{amsmath}
\usepackage{amsthm}
\usepackage{enumerate}
\usepackage{verbatim}
\usepackage[hmargin=3cm, vmargin=3.5cm]{geometry}

\theoremstyle{plain}
\newtheorem{thm}{Theorem}
\newtheorem{cor}[thm]{Corollary}
\newtheorem{lem}[thm]{Lemma}
\newtheorem{prop}[thm]{Proposition}

\theoremstyle{definition}
\newtheorem{ex}{Example}

\newtheorem*{rmk}{Remark}

\newcommand{\Q}{\mathbb{Q}}
\renewcommand{\P}{\mathbb{P}}
\newcommand{\E}{\mathbb{E}}
\newcommand{\R}{\mathbb{R}}
\newcommand{\F}{\mathcal{F}}

\newcommand{\G}{\mathcal{G}}
\renewcommand{\H}{\mathcal{H}}
\newcommand{\Gt}{\tilde{\mathcal{G}}}

\newcommand{\ind}{\mathbbm{1}}

\newcommand{\eps}{\varepsilon}
\newcommand{\bp}{\begin{proof}}
\newcommand{\ep}{\end{proof}}
\def\bal#1\eal{\begin{align*}#1\end{align*}}
\newcommand{\Nc}{\mathcal{N}}

\newcommand{\Z}{\mathbb{Z}}
\renewcommand{\d}{{\rm{d}}}

\renewcommand{\S}{\mathcal{S}}

\newcommand{\N}{\mathbb{N}}

\author{Simon C.~Harris\thanks{University of Auckland, Private Bag 92019, Auckland 1142, New Zealand. Email: \href{mailto:simon.harris@auckland.ac.nz}{simon.harris@auckland.ac.nz}}, Samuel G.G.~Johnston\thanks{University College Dublin, Belfield, Dublin 4, Ireland. Email: \href{mailto:sggjohnston@gmail.com}{sggjohnston@gmail.com}}~ and Matthew I.~Roberts\thanks{University of Bath, Claverton Down, Bath BA2 7AY, UK. Email: \href{mailto:mattiroberts@gmail.com}{mattiroberts@gmail.com}}}

\title{The coalescent structure of continuous-time\\ Galton-Watson trees}

\begin{document}
\maketitle

\begin{abstract}
Take a continuous-time Galton-Watson tree. If the system survives until a large time $T$, then choose $k$ particles uniformly from those alive. What does the ancestral tree drawn out by these $k$ particles look like? Some special cases are known but we give a more complete answer. We concentrate on near-critical cases where the mean number of offspring is $1+\mu/T$ for some $\mu\in\R$, and show that a scaling limit exists as $T\to\infty$. Viewed backwards in time, the resulting coalescent process is topologically equivalent to Kingman's coalescent, but the times of coalescence have an interesting and highly non-trivial structure. The randomly fluctuating population size, as opposed to constant size populations where the Kingman coalescent more usually arises, have a pronounced effect on both the results and the method of proof required. We give explicit formulas for the distribution of the coalescent times, as well as a construction of the genealogical tree involving a mixture of independent and identically distributed random variables. In general subcritical and supercritical cases it is not possible to give such explicit formulas, but we highlight the special case of birth-death processes.
\end{abstract}

\section{Introduction}
Let $L$ be a random variable taking values in $\Z_+=\{0,1,2,\ldots\}$. Consider a continuous-time Galton-Watson tree beginning with one initial particle and branching at rate $r$ with offspring distribution $L$. We will give more details of the model shortly.

Fix a large time $T$, and condition on the event that at least $k$ particles are alive at time $T$. Choose $k$ particles uniformly at random (without replacement) from those alive at time $T$. These particles, and their ancestors, draw out a smaller tree. The general question that we attempt to answer is: what does this tree look like? This is a fundamental question about Galton-Watson trees; several authors have given answers via interesting and contrasting methods for various special cases, usually when $k=2$. We aim to give a more complete answer with a unified approach that can be adapted to other situations.

Before explaining our most general results we highlight some illuminating examples. Let $\Nc_t$ be the set of particles that are alive at time $t$, and write $N_t = \#\Nc_t$ for the number of particles that are alive at time $t$. Let $m=\E[L]$ and for each $j\ge0$ let $p_j = \P(L=j)$. We assume throughout the article, without further mention, that $p_0+p_1\neq 1$.

On the event $\{N_T\ge 2\}$, choose a pair of particles $(U_T,V_T)\in\Nc_T$ uniformly at random (without replacement). Then let $\S(T)$ be the last time at which these uniformly chosen particles shared a common ancestor. If $N_T \le 1$ then set $\S(T)=0$.

If $p_0\in[0,1)$ and $p_2=1-p_0$, then the model is known as a birth-death process. In this case we are able to calculate explicitly the distribution of $\S(T)$ conditional on $\{N_T\ge 2\}$. In particular,
\begin{itemize}
\item in the supercritical case when $p_2>p_0$, the law of $\S(T)$ conditional on $\{N_T\ge 2\}$ converges as $T\to\infty$ to a non-trivial distribution with tail satisfying
\[\lim_{T\to\infty}\P(\S(T) \ge t\,|\,N_T\ge 2)\sim 2r(m-1)t e^{-r(m-1)t} \,\,\hbox{ as }\,\, t\to\infty;\]
\item in the subcritical case $p_0>p_2$, the law of $T-\S(T)$ conditional on $\{N_T\ge 2\}$ converges as $T\to\infty$ to a non-trivial distribution with tail satisfying
\[\lim_{T\to\infty}\P(T-\S(T) \ge t\,|\,N_T\ge 2)\sim \Big(1-\frac{2p_2}{3p_0}\Big)e^{r(m-1)t} \,\,\hbox{ as }\,\, t\to\infty. \]
\end{itemize}

\noindent
In the critical case we can work more generally.
\begin{itemize}
\item If $L$ has any distribution satisfying $m=\E[L]=1$ and $\E[L^2]<\infty$, then the law of $\S(T)/T$ conditional on $\{N_T\ge 2\}$ converges as $T\to\infty$ to a non-trivial distribution on $[0,1]$ satisfying
\[\lim_{T\to\infty}\P\Big(\frac{\S(T)}{T} \ge t\,\Big|\,N_T\ge 2\Big) = \frac{2(1-t)}{t^2}\Big(\log\Big(\frac{1}{1-t}\Big) - t\Big).\]
\end{itemize}
This last result (the critical case) is known: Durrett \cite{durrett:genealogy} gave a power series expansion, and Athreya \cite{athreya:coalescence} gave a representation in terms of a geometric number of exponential random variables, both of which we will show agree with our explicit formula. Lambert \cite{lambert:coalescence_GW} gave a similar formula for a certain critical continuous state branching process. Athreya also mentioned that his expression could alternatively be obtained by using the excursion representation of continuum random trees. This method was also used by Popovic \cite{popovic:asymptotic_genealogy_critical_bp}, Aldous and Popovic \cite{aldous_popovic:critical_bp_biodiversity}, Lambert \cite{lambert:contour_splitting_trees}, and Lambert and Popovic \cite{lambert:coalescent_branching_trees} to investigate related questions. We give more details of this link in Section \ref{crt_sec}.

Beyond the critical case, we can find a distributional scaling limit when $L$ is ``\emph{near-critical}''. We let the distribution of $L$ depend on $T$, and write $\P_T$ to signify that the Galton-Watson process now depends on $T$ as a result.
\begin{itemize}
\item Suppose that $L$ satisfies $\E_T[L] = 1+\mu/T+o(1/T)$, $\E_T[L(L-1)]=\beta+o(1)$, and that $L^2$ is uniformly integrable under $\P_T$. Then the law of $\S(T)/T$ conditional on $\{N_T\ge 2\}$ converges as $T\to\infty$ to a non-trivial distribution on $[0,1]$ satisfying
\end{itemize}
\[\lim_{T\to\infty} \hspace{-1mm}\P_T\hspace{-0.5mm}\Big(\frac{\S(T)}{T} \ge s \Big| N_T\ge 2\Big) = 2\Big(\frac{e^{r\mu(1-s)}-1}{e^{r\mu(1-s)}-e^{r\mu}}\Big) + 2 \frac{(e^{r\mu}-1)(e^{r\mu(1-s)}-1)}{(e^{r\mu(1-s)}-e^{r\mu})^2} \hspace{-0.5mm}\log\hspace{-0.5mm}\Big(\frac{e^{r\mu}-1}{e^{r\mu(1-s)}-1}\Big).\]
O'Connell \cite[Theorem 2.3]{oconnell:genealogy_mrca} gave this result by using a diffusion approximation, relating the near-critical process to a time-changed Yule tree, and then adapting the method of Durrett \cite{durrett:genealogy} from the critical case. Again, these authors only considered choosing two particles at time $T$.

All of the above special cases---although they are already interesting in their own right---are just a taster of our general results. The effectiveness and adaptability of our method is demonstrated by the fact that it recovers, in these cases, the results of several separate investigations using different techniques \cite{athreya:coalescence, durrett:genealogy, lambert:coalescence_GW, oconnell:genealogy_mrca}. In our main result (see Theorem \ref{nearcritthm}), we will give a complete description for the genealogical tree of a uniform sample of $k\geq2$ individuals in near-critical Galton-Watson processes in the large time limit.

We now attempt to describe our general results in a little more detail. For any $k\ge 2$, under a second moment condition on $L$, we sample $k$ particles without replacement at time $T$ and trace back the tree induced by them and their ancestors. It turns out that if we view this tree backwards in time, then the coalescent process thus obtained is topologically the same as Kingman's coalescent, but has different coalescent rates. We give an explicit joint distribution function for the limiting $k-1$ coalescent times, which are also asymptotically independent of the Kingman tree topology; it turns out that they can be constructed by choosing $k$ independent random variables with a certain distribution and renormalising by the maximum. Equivalently, the coalescent times can also be interpreted as being a mixture of independent identically distributed random variables. The correlation introduced by this mixture is linked to the random variations of the population size. On the other hand, Kingman's coalescent usually arises from populations where the total number of individuals is kept constant: see, for example, \cite{schweinsberg2003coalescent}. One of the biggest hurdles in our proof was to overcome the effect of fluctuations in the population size; we did this using a very natural change of measure $\Q^{k,T}$ under which the coalescent times decorrelate, making calculations easier.

After this article was released, using knowledge of the precise form of our answers,  Lambert  \cite{lambert:genealogy_binary} was able to construct a remarkable method to obtain some of our formulas for \emph{coalescent point processes}. However, \cite{lambert:genealogy_binary} assumes binary branching, so whilst it can apply to birth-death processes, it does not cover our main results concerning general near-critical Galton-Watson processes. We discuss this approach further in Section \ref{crt_sec}.

Ren, Song and Sun \cite{ren_song_sun:2spinesuperprocess, ren_song_sun:2spineGW} have also subsequently used a 2-spine approach (involving analogues of our $\Q^{2,T}$) to give elegant probabilistic proofs of Yaglom theorems about the size of the population conditional on survival, both for the discrete time critical Galton-Watson processes \cite{ren_song_sun:2spineGW} and critical superprocesses \cite{ren_song_sun:2spinesuperprocess}.

In Section \ref{sec:results}, we state full details our main results, we present a more intuitive probabilistic construction of the near-critical scaling limit, and we  then provide a heuristic explanation and intuitive probabilistic derivation for it. 
We follow that with discussion of some of the properties of the scaling limit and comparisons to related results in Section \ref{chatsec}. In Section \ref{CoMsec}, we introduce the tools required to prove our results, including a change of measure and a version of Campbell's formula. We then prove our main result for birth-death processes in Section \ref{BDsec}, and our main result for near-critical processes in Section \ref{nearcritsec}.

\section{Results}\label{sec:results}

We first describe, in more detail than previously, our basic continuous-time Galton-Watson tree. Under a probability measure $\P$, we begin with one particle, the root, which we give the label $\emptyset$. This particle waits an exponential amount of time $\tau_\emptyset$ with parameter $r$, and then instantaneously dies and gives birth to some offspring with labels $1,2,\ldots,L_\emptyset$, where $L_\emptyset$ is an independent copy of the random variable $L$. To be precise, at the time $\tau_\emptyset$ the particle $\emptyset$ is no longer alive and its offspring are. These offspring then repeat, independently, this behaviour: each particle $u$ waits an independent exponential amount of time with parameter $r$ before dying and giving birth to offspring $u1, u2, \ldots, uL_u$ where $L_u$ is an independent copy of $L$, and so on. We let $p_j=\P(L = j)$ and $m=\sum_{j=1}^\infty j p_j$. Since we will be using more than one probability measure, we will write $\P[\cdot]$ instead of $\E[\cdot]$ for the expectation operator corresponding to $\P$.

Denote by $\Nc_T$ the set of all particles alive at time $T$. For a particle $u\in \Nc_T$ we let $\tau_u$ be the time of its death, and define $\tau_u(T) = \tau_u \wedge T$. If $u$ is an ancestor of $v$, we write $u\le v$, and if $u$ is a \emph{strict} ancestor of $v$ (i.e.~$u\le v$ and $u\neq v$) then we write $u<v$. For technical reasons we introduce a graveyard $\Delta$ which is not alive (it is not an element of $\Nc_T$).

For a particle $u\in\Nc_t$ and $s\le t$, let $u(s)$ be the unique ancestor of $u$ that was alive at time $s$. For two particles $u,v\in \Nc_T$, let $\sigma(u,v)$ be the last time at which they shared a common ancestor,
\[\sigma(u,v) = \sup\{ t\ge 0 : u(t) = v(t)\}.\]

Now fix $k\in\N$, and at time $T$, on the event $N_T\ge k$, pick $k$ particles $U^1_T,\ldots, U^k_T$ uniformly at random without replacement from $\Nc_T$. We let $\mathcal P^k_t(T)$ be the partition of $\{1,\ldots,k\}$ induced by letting $i$ and $j$ be in the same block if particles $U^i_T$ and $U^j_T$ shared a common ancestor at time $t$, i.e.~if $\sigma(U^i_T,U^j_T)>t$. We order the elements of $\mathcal P^k_t(T)$ by their smallest element.

There are two aspects to the information contained in $\mathcal P^k_t(T)$. The first is the topological information; given a collection of blocks, which block will split first, and when it does, what will the new blocks created look like? The second is the times at which the splits occur. We will find that in the models we look at, the topological information is (asymptotically) universal and rather simple to describe, whereas the split times are much more delicate and depend on the parameters of the model. In order to separate out these two aspects, we require some more notation.

Let $\nu^k_t(T)$ be the number of blocks in $\mathcal P^k_t(T)$, or equivalently the number of distinct ancestors of $U^1_T\,\ldots,U^k_T$ that are alive at time $t$; that is, $\nu^k_t(T) = \#\{u\in\mathcal N_t : u<U^i_T \hbox{ for some } i\le k\}$.

For $i=1,\ldots, k-1$ let
\[\S^k_i(T) = \inf\{t\ge 0 : \nu^k_t > i\}.\]
We call $S^k_1(T)\le \ldots\le S^k_{k-1}(T)$ the \emph{split times}. For technical reasons it is often easier to consider the unordered split times; we let $(\tilde\S^k_1(T),\ldots,\tilde\S^k_{k-1}(T))$ be a uniformly random permutation of $(\S^k_1(T),\ldots,\S^k_{k-1}(T))$.

For $i=0,\ldots,k-1$ let $P^k_i(T) = \mathcal P^k_{S^k_i}(T)$, and let $\mathcal H = \sigma(P^k_0(T),\ldots,P^k_{k-1}(T))$, so that $\mathcal H$ contains all the topological information about the tree generated by $U^1_T,\ldots U^k_T$, but almost no information about the split times.

\subsection{Birth-death processes}\label{birthdeathresults}

Fix $\alpha\ge 0$ and $\beta>0$. Suppose that $r = \alpha+\beta$, $p_0 = \alpha/(\alpha+\beta)$ and $p_2 = \beta/(\alpha+\beta)$, with $p_j=0$ for $j\neq 0,2$. This is known as a birth-death process with birth rate $\beta$ and death rate $\alpha$. Note that since there are only binary splits, if there are at least $k$ particles alive at time $T$ then when we pick $k$ uniformly at random as above there are always exactly $k-1$ distinct split times. Our first theorem gives an explicit distribution for these split times, in the non-critical case and conditional on $\{N_T\ge k\}$.

\begin{thm}\label{noncritBDthm}
Suppose that $\alpha\neq\beta$. For any $s_1,\ldots,s_{k-1}\in(0,T]$, the unordered split times are independent of $\mathcal H$ and satisfy
\begin{align*}
&\P(\tilde \S^k_1(T) \ge s_1, \ldots, \tilde\S^k_{k-1}(T) \ge s_{k-1} | N_T\ge k)\\
&= \frac{k(E_0 -\alpha/\beta)^k}{(E_0-1)^{k-1}}\Bigg[ \frac{1}{(E_0-\alpha/\beta)}\prod_{i=1}^{k-1}\frac{E_i - 1}{E_i-E_0}+ \sum_{j=1}^{k-1}\frac{(E_j-1)}{(E_j-E_0)^2}\bigg(\prod_{\substack{i=1\\ i\neq j}}^{k-1} \frac{E_i-1}{E_i-E_j}\bigg)\log\hspace{-0.5mm}\Big(\frac{\beta E_0-\alpha}{\beta E_j-\alpha}\Big)\Bigg]
\end{align*}
where $E_j = e^{(\beta-\alpha)(T-s_j)}$ for each $j=1,\ldots,k$ and $s_0=0$. Furthermore, the partition process $P^k_0(T),P^k_1(T),\ldots,P^k_{k-1}(T)$ has the following description:
\begin{itemize}
\item if $P^k_i(T)$ contains blocks of sizes $a_1,\ldots,a_{i+1}$, the probability that the next block to split will be block $j$ is $\frac{a_j-1}{k-i-1}$;
\item if a block of size $a$ splits, it creates two blocks whose (ordered) sizes are $l$ and $a-l$ with probability $1/(a-1)$ for each $l=1,\ldots,a-1$.
\end{itemize}
\end{thm}

The case of the Yule tree, in which $\beta=1$ and $\alpha=0$, gives simpler formulas for the split times.

\begin{ex}[Yule tree]
Suppose that $\alpha=0$ and $\beta=1$. Then for any $s\in(0,T]$,
\[\P(\tilde\S^2_1(T) \ge s \,|\, N_T\ge 2) = \frac{2(e^{-s}-e^{-T})(e^{-s}-1+s)}{(1-e^{-T})(1-e^{-s})^2}\]
and for any $s_1,s_2\in(0,T]$,
\begin{align*}
&\P(\tilde\S^3_1(T) \ge s_1,\, \tilde\S^3_2(T) \ge s_2 \,|\, N_T\ge 3)\\
&= 3(e^{-s_1}\hspace{-0.5mm}-\hspace{-0.5mm}e^{-T})(e^{-s_2}\hspace{-0.5mm}-\hspace{-0.5mm}e^{-T})\frac{\big(s_1(1\hspace{-0.8mm}-\hspace{-0.8mm}e^{-s_2})^2 - s_2(1\hspace{-0.8mm}-\hspace{-0.8mm}e^{-s_1})^2 + (1\hspace{-0.8mm}-\hspace{-0.8mm}e^{-s_1})(1\hspace{-0.8mm}-\hspace{-0.8mm}e^{-s_2})(e^{-s_2}\hspace{-0.8mm}-\hspace{-0.8mm}e^{-s_1})\big)}{(1-e^{-T})^2(1-e^{-s_1})^2(1-e^{-s_2})^2(e^{-s_2}-e^{-s_1})}.
\end{align*}
\end{ex}

Returning to general $\alpha\neq\beta$, the case $k=2$, mentioned in the introduction, is of particular interest. Note that when $k=2$, there is only one split time, so the choice of ordered or unordered is irrelevant. To be consistent with the description in the introduction we write $\S(T) = \S^2_1(T)$. Taking a limit as $T\to\infty$ simplifies the formula significantly, although we have to consider the supercritical and subcritical cases separately.

\begin{ex}[Supercritical birth-death, $T\to\infty$]
Suppose that $\beta>\alpha$. Then for any $s> 0$,
\[\lim_{T \to \infty} \P( \S(T) \ge s \,|\, N_T \geq 2 ) = \frac{ 2e^{-(\beta-\alpha)s} }{ (1-e^{-(\beta-\alpha)s})^2 }\big( (\beta-\alpha)s - 1 + e^{-(\beta-\alpha)s} \big).\]
\end{ex}

\begin{ex}[Subcritical birth-death, $T\to\infty$]
Suppose that $\alpha>\beta$. Then for any $s> 0$,
\[\lim_{T \to \infty}  \mathbb{P} ( \S(T) \ge T - s \,|\, N_T \geq 2 )  = \frac{2\alpha^2}{\beta^2} (e^{(\alpha-\beta)s} - 1) \Big( e^{(\alpha-\beta)s} \log \Big( 1 + \frac{\beta}{\alpha e^{(\alpha-\beta)s} - \beta } \Big) - \frac{\beta}{\alpha} \Big).\]
\end{ex}

To our knowledge all of these results are new. We note (as Durrett also mentioned in \cite{durrett:genealogy}) that in the supercritical case, the time $\S(T)$ is likely to be near $0$, whereas in the subcritical case, $\S(T)$ is likely to be near $T$. This much is to be expected, but the detailed behaviour is perhaps more surprising: as mentioned in the introduction, some elementary calculations using the formulas above show that in the supercritical case,
\[\lim_{T\to\infty}\P(\S(T) \ge s\,|\,N_T\ge 2)\sim 2(\beta-\alpha)s e^{-(\beta-\alpha)s} \,\,\hbox{ as }\,\, s\to\infty,\]
whereas in the subcritical case,
\[\lim_{T\to\infty}\P(T-\S(T) \ge s\,|\,N_T\ge 2)\sim \Big(1-\frac{2\beta}{3\alpha}\Big)e^{-(\alpha-\beta)s} \,\,\hbox{ as }\,\, s\to\infty.\]

We can also give analogous results in the critical case $\alpha=\beta$.

\begin{thm}\label{critBDthm}
Suppose that $\alpha=\beta$. For any $s_1,\ldots,s_{k-1}\in(0,T]$ with $s_i\neq s_j$ for $i\neq j$, the unordered split times are independent of $\mathcal H$ and satisfy
\begin{multline*}
\P(\tilde \S^k_1(T)/T \ge s_1, \ldots, \tilde\S^k_{k-1}(T)/T \ge s_{k-1} \,|\, N_T\ge k)\\
\hspace{6mm}= k \Big(1+\frac{1}{\beta T}\Big)^k \Bigg[ \frac{1}{1+1/T}\prod_{i=1}^{k-1}\Big(1-\frac{1}{s_i}\Big)+ \sum_{j=1}^{k-1}\frac{1-s_j}{s_j^2}\bigg(\prod_{\substack{i=1\\ i\neq j}}^{k-1} \frac{1-s_i}{s_j-s_i}\bigg)\log\Big(\frac{1+1/T}{1-s_j+1/T}\Big)\Bigg].
\end{multline*}
Furthermore, the partition process $P^k_0(T),P^k_1(T),\ldots,P^k_{k-1}(T)$ has the following description:
\begin{itemize}
\item if $P^k_i(T)$ contains blocks of sizes $a_1,\ldots,a_{i+1}$, the probability that the next block to split will be block $j$ is $\frac{a_j-1}{k-i-1}$;
\item if a block of size $a$ splits, it creates two blocks whose (ordered) sizes are $l$ and $a-l$ with probability $1/(a-1)$ for each $l=1,\ldots,a-1$.
\end{itemize}
\end{thm}

\begin{ex}
Suppose that $\alpha=\beta$. Then for any $s>0$
\[\P(\tilde\S^2_1(T)/T \ge s \,|\, N_T\ge 2) = 2\Big(1+\frac{1}{\beta T}\Big)^2 \Big(\frac{1-s}{s^2}\Big)\Big(\log\Big(\frac{1+1/T}{1-s+1/T}\Big) - \frac{s}{1+1/T}\Big)\]
and for any $s_1,s_2>0$,
\begin{multline*}
\P(\tilde\S^3_1(T)/T \ge s_1,\, \tilde\S^3_2(T)/T \ge s_2 \,|\, N_T\ge 3)\\
= \frac{3(1+\frac{1}{\beta T})^3(1-s_1)(1-s_2)}{s_1^2 s_2^2 (s_2-s_1)} \bigg[ s_2^2\log\Big(\frac{1-s_1+\frac{1}{T}}{1+\frac{1}{T}}\Big) - s_1^2 \log\Big(\frac{1-s_2+\frac{1}{T}}{1+\frac{1}{T}}\Big) + \frac{s_1 s_2 (s_2-s_1)}{1+\frac{1}{T}}\bigg].
\end{multline*}
\end{ex}

We can easily let $T\to\infty$ in these formulas, but in the critical case---and even in near-critical cases---if we are willing to take a scaling limit as $T\to\infty$ then we can work much more generally.

\subsection{Near-critical processes: a scaling limit}\label{nearcritresultssec}

We no longer restrict to birth-death processes; the birth distribution $L$ may take any (non-negative integer) value. In order to consider a scaling limit, we take Galton-Watson processes that are \emph{near-critical}, in that the mean number of offspring is approximately $1+\mu/T$ for some $\mu\in\R$. We also insist that the variance converges. Conditional on survival to time $T$, we sample $k$ particles uniformly without replacement, and ask for the structure of the genealogical tree generated by these $k$ particles. In other branching models when the population is kept constant, it has been shown that the resulting coalescent process converges as $T\to\infty$ to Kingman's coalescent \cite{schweinsberg2003coalescent}. We see something slightly different.

To state our result precisely, we need some more notation. Fix $\mu \in \mathbb{R}$ and $\sigma>0$. Suppose that for each $T>0$, the offspring distribution $L$ satisfies
\begin{itemize}
\item $\P_T[L] = 1 + \mu/T + o(1/T)$
\item $\P_T[L (L - 1)] = \sigma^2 + o(1)$
\item $L^2$ is uniformly integrable under $\P_T$: that is, for any $\eps>0$, there exists $K$ such that
\[\P_T[L^2\ind_{\{L>K\}}]<\epsilon \text{ for all } T.\]
\end{itemize}

\begin{thm}[Near-critical scaling limit]\label{nearcritthm}
Suppose that the conditions above hold. Then the split times are asymptotically independent of $\mathcal H$, and if $\mu\neq 0$, then for any $s_1,\ldots,s_{k-1}\in(0,1)$ with $s_i\neq s_j$ for any $i\neq j$,
\begin{multline*}
\lim_{T\to\infty} \P_T(\tilde \S^k_1(T)/T \ge s_1, \ldots, \tilde\S^k_{k-1}(T)/T \ge s_{k-1} \,|\, N_T\ge k)\\
= k \prod_{i=1}^{k-1} \frac{E_i}{E_i-E_0} + k\sum_{j=1}^{k-1} \frac{E_0 E_j}{(E_j-E_0)^2} \bigg( \prod_{\substack{i=1\\ i\neq j}}^{k-1} \frac{E_i}{E_i-E_j}\bigg)\log\frac{E_0}{E_j}
\end{multline*}
where $E_j = e^{r\mu(1-s_j)}-1$ for each $j=0,\ldots,k-1$ and $s_0=0$. If $\mu=0$, then instead
\begin{multline*}
\lim_{T\to\infty} \P(\tilde \S^k_1(T)/T \ge s_1, \ldots, \tilde\S^k_{k-1}(T)/T \ge s_{k-1} \,|\, N_T\ge k)\\
\hspace{10mm}= k \prod_{i=1}^{k-1} \frac{s_i-1}{s_i} - k\sum_{j=1}^{k-1} \frac{1-s_j}{s_j^2} \bigg( \prod_{\substack{i=1\\ i\neq j}}^{k-1} \frac{1-s_i}{s_j-s_i}\bigg)\log(1-s_j).
\end{multline*}
Furthermore, the partition process $P^k_0(T),P^k_1(T),\ldots,P^k_{k-1}(T)$ has the following description:
\begin{itemize}
\item if $P^k_i(T)$ contains blocks of sizes $a_1,\ldots,a_{i+1}$, the probability that the next block to split will be block $j$ converges as $T\to\infty$ to $\frac{a_j-1}{k-i-1}$;
\item if a block of size $a$ splits, with probability tending to $1$ it creates two blocks whose (ordered) sizes are $l$ and $a-l$ with probability converging to $\frac{1}{a-1}$ for each $l=1,\ldots,a-1$.
\end{itemize}
\end{thm}

In Theorems \ref{noncritBDthm} and \ref{critBDthm} we saw that the split times were independent of $\mathcal H$. This cannot be the case in Theorem \ref{nearcritthm}, since two or more split times may be equal with positive probability, an event which is captured by both the split times and the topological information $\mathcal H$. However we do see that the split times are \emph{asymptotically} independent, in that $\P_T(A\cap B) \to \P_T(A)\P_T(B)$ for any $A\in\sigma(\S^k_1(T),\ldots,\S^k_{k-1}(T))$ and $B\in \mathcal H$, which is the best that we can hope for.

We note here that the topology of the (limiting) tree described forwards in time in Theorem \ref{nearcritthm} is the same as that described backwards in time by Kingman's coalescent; but the times of splits (or times of mergers, in the coalescent picture) are drastically different.

In the case that the process is actually critical we recover the following simple formula for the split times.

\begin{ex}[Critical processes]
Suppose that $\P[L]=1$ and $\P[L^2]<\infty$. Then for any $s\in(0,1)$,
\begin{equation}\label{critk2}
\lim_{T\to\infty} \P(\S(T)/T \ge s \,|\, N_T\ge 2) = \frac{2(s-1)}{s^2}\big(\log(1-s)+s\big).
\end{equation}
\end{ex}

\begin{ex}[Near-critical scaling limit, $k=2$]\label{nearcritk2}
Suppose that the conditions of Theorem \ref{nearcritthm} hold with $\mu\neq 0$. Then for any $s\in(0,1)$,
\begin{multline*}
\lim_{T\to\infty} \P_T(\S(T)/T \ge s \,|\, N_T\ge 2)\\
 = 2\Big(\frac{e^{r\mu(1-s)}-1}{e^{r\mu(1-s)}-e^{r\mu}}\Big) + 2 \frac{(e^{r\mu}-1)(e^{r\mu(1-s)}-1)}{(e^{r\mu(1-s)}-e^{r\mu})^2} \log\Big(\frac{e^{r\mu}-1}{e^{r\mu(1-s)}-1}\Big).
\end{multline*}
\end{ex}

Both of these examples are known, but to our knowledge the general formula is not. We give more details in Section \ref{comparisonsec}.

\subsection{Construction of the near-critical scaling limit}\label{limitconstructionsec}

In this section we investigate further the scaling limit observed in Theorem \ref{nearcritthm}. Our aim is to give a more intuitive probabilistic understanding of the scaling limit, rather than the explicit formulas seen in Theorems \ref{noncritBDthm} to \ref{nearcritthm}.

We work under the conditions of Section \ref{nearcritresultssec}: we fix $\mu \in \mathbb{R}$ and $\sigma>0$, and suppose that for each $T>0$ the offspring distribution $L$ satisfies
\begin{itemize}
\item $\P_T[L] = 1 + \mu/T + o(1/T)$
\item $\P_T[L (L - 1)] = \sigma^2 + o(1)$
\item $L^2$ is uniformly integrable under $\P_T$.
\end{itemize}
Theorem \ref{nearcritthm} says that the rescaled unordered split times, conditional on at least $k$ particles being alive at time $T$, converge jointly in distribution to an explicit limit,
\[\Big(\frac{\tilde \S^k_1(T)}{T},\ldots,\frac{\tilde \S^k_{k-1}(T)}{T}\Big) \xrightarrow{(d)} (\tilde \S^k_1,\ldots,\tilde\S^k_{k-1}).\]
We aim to shed some more light on this limit. First we note that, although the split times (for fixed $T$) do not usually have a joint density---with positive probability one split time may equal another---their scaling limit \emph{does} have a density. Indeed, from the proof of Theorem \ref{nearcritthm} (or by checking directly) we see that this density satisfies (with $s_0=0$)
\[f_k(s_1,\ldots,s_{k-1}) = \begin{cases} \displaystyle k(r\mu)^{k-1}(1-e^{-r\mu}) \int_0^\infty \theta^{k-1} \prod_{i=0}^{k-1} \frac{e^{r\mu(1-s_i)}}{(1+\theta(e^{r\mu(1-s_i)}-1))^2}\,\d\theta & \hbox{ if } \mu > 0\\
                                          \displaystyle k \int_0^\infty \theta^{k-1} \prod_{i=0}^{k-1} \frac{1}{(1+\theta(1-s_i))^2}\,\d\theta & \hbox{ if } \mu = 0\\
                                          \displaystyle k(-1)^k(r\mu)^{k-1}(1-e^{-r\mu})\hspace{-1.5mm} \int_0^\infty \hspace{-2mm}\theta^{k-1}\hspace{-1mm} \prod_{i=0}^{k-1} \frac{e^{r\mu(1-s_i)}}{(1-\theta(e^{r\mu(1-s_i)}-1))^2}\d\theta & \hbox{ if } \mu < 0.
\end{cases}\]

The following proposition gives a construction of the scaling limit of the tree in the critical case $\mu=0$, in the spirit of Aldous' construction of Kingman's coalescent \cite[Section 4.2]{aldous:coalescence_review}. In particular it gives a method for consistently constructing the times $(\tilde \S^k_1,\ldots,\tilde\S^k_{k-1})$.

\begin{thm}[A construction for critical genealogies]\label{maxconcrit}
Suppose that $\mu=0$. Let $X_1,X_2,\ldots$ be a sequence of independent and identically distributed random variables on $(0,\infty)$ with density $(1+x)^{-2}$. Let $M_k = \max_{i\le k} X_i$, and choose $I$ such that $X_I=M_k$. For $i\le k$ define $T_i = 1-X_i/M_k$. Then $(T_1,\ldots,T_{I-1},T_{I+1},\ldots,T_k)$ is equal in distribution to $(\tilde \S^k_1,\ldots,\tilde\S^k_{k-1})$ in the critical case $(\mu=0$).

Moreover, the ancestral tree drawn out by the $k$ uniformly chosen particles has the following description: let $U_1,U_2,\ldots$ be independent uniform random variables on $[0,1]$. Within the unit square, for each $1\le i\le k$, draw a vertical line from $(U_i,0)$ to $(U_i,1-T_i)$. These lines represent the branches of our tree. Now, for each $1\le i\le k-1$, draw a horizontal line starting from $(U_i,T_i)$ towards $(U_I,T_i)$ but stopping as soon as it hits another (vertical) line (see Figure \ref{aldPfig} below).
\end{thm}

\begin{figure}[h!]
  \centering
   \includegraphics[width=5cm]{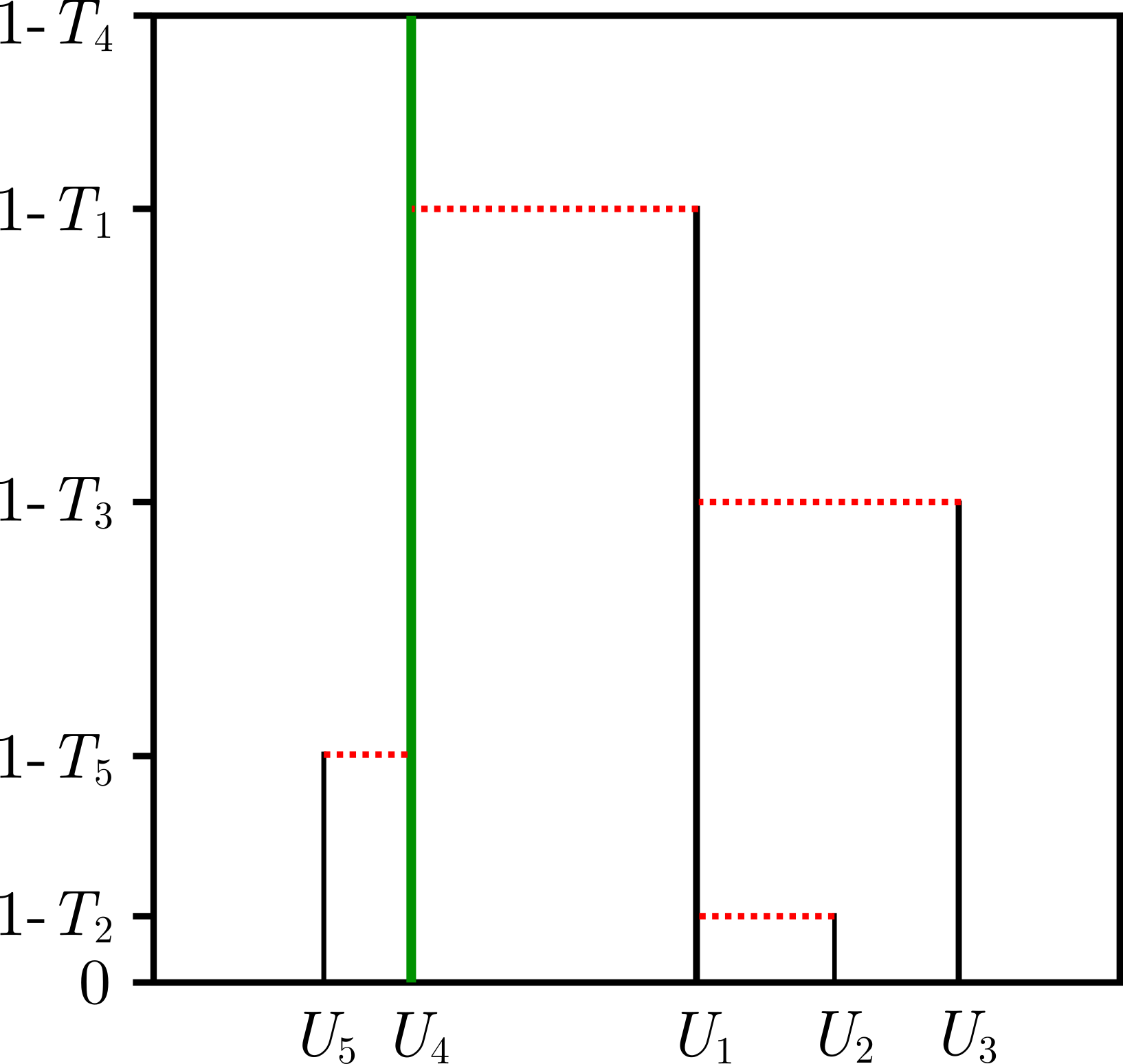}
   \vspace{0mm}
  \caption{\small{A representation of the rescaled tree drawn out by $5$ particles chosen uniformly at random from those alive at a large time. Here $I=4$.}}
\label{aldPfig}
  \end{figure}

This result, in particular, clarifies the consistency of the split times. Of course, if we choose $k+1$ particles uniformly without replacement at time $T$, and then forget one of them, the result should be consistent with choosing $k$ particles originally. This is not immediately obvious from the distribution function given in Theorem \ref{nearcritthm}, but it follows easily from the construction in Theorem \ref{maxconcrit}.

\begin{rmk}
In the above construction the scale of the horizontal axis has no meaning; any permutation of the vertical lines could replace the random variables $U_0,\dots,U_{k-1}$ and give the same Kingman tree topology. Indeed, the tallest (green) line could just as well be fixed, say as the leftmost, and the remaining vertical lines randomly permuted without changing the tree topology. Nevertheless, in Section \ref{heuristics}, we will describe a construction under $\Q^k$ where the gaps on the horizontal axis between the vertical lines can be interpreted as the population size: see Figure \ref{Qpopconst} and the discussion immediately beforehand.
\end{rmk}

We can do something similar when $\mu\neq 0$.

\begin{thm}[A construction for near-critical genealogies]\label{maxconnoncrit}
Suppose that $\mu\neq0$. Let $X_1,X_2,\ldots$ be a sequence of independent and identically distributed random variables on $(0,\infty)$ with density $(1+x)^{-2}$. Let $M_k = \max_{i\le k} X_i$, and choose $I$ such that $X_I=M_k$. For $i\le k$ define
\[T_i = 1-\frac{1}{r\mu}\log\Big(1+(e^{r\mu}-1)\frac{X_i}{M_k}\Big).\]
Then $(T_1,\ldots,T_{I-1},T_{I+1},\ldots,T_k)$ is equal in distribution to $(\tilde \S^k_1,\ldots,\tilde\S^k_{k-1})$.

Moreover, the ancestral tree drawn out by the $k$ uniformly chosen particles has the same construction as in Theorem \ref{maxconcrit}.
\end{thm}

\subsection{Heuristic explanation of our results}\label{heuristics}

In this section, we aim to give a quick intuitive probabilistic derivation of Theorem \ref{maxconcrit}. For this we will need to construct a certain very natural probability measure, $\Q^{k,T}$. Whilst $\Q^{k,T}$ will not be precisely defined until Section \ref{CoMsec} (see \eqref{Qmeasure}), and it is fundamental to the entire success of our approach, for now it will be sufficient to know only a few of its basic properties. The probability measure $\Q^{k,T}$ will describe the behaviour of $k$ distinguished \emph{spine} particles along which standard Galton-Watson processes are immigrated. Under $\Q^{k,T}$, these $k$ spines will have the property of looking like a uniform choice without replacement from all those $N_T$ particles alive at time $T$. For this heuristic we will use this measure $\Q^{k,T}$, together with the classical theorems of Kolmogorov \cite{kolmogorov:solution_biological_problem} about the asymptotics of the survival probability, and Yaglom \cite{yaglom:certain_limit_thms} about the distribution of the scaled population size conditioned to survive (see for example \cite[Theorem 12.7]{lyons_peres:probability_on_trees} for a modern treatment of both these results).

Let $E_k$ be any event concerning the tree drawn out by the $k$ uniformly sampled particles (we will only consider these conditionally on $N_T\ge k$ so that they always exist). It will be easy to show, using the definition of our change of measure $\Q^{k,T}$, that
\begin{equation}\label{heuristiceq1}
\P( E_k | N_T \ge k) = \Q^{k,T}\Big[ \frac{\ind_{E_k^\xi}}{N_T (N_T-1)\ldots (N_T-k+1)} \Big] \frac{\P[N_T (N_T-1)\ldots (N_T-k+1)]}{\P(N_T\ge k)}
\end{equation}
where $E_k^\xi$ is the event corresponding to $E_k$, but for the $k$ spines under $\Q^{k,T}$, rather than the $k$ uniformly chosen particles under $\P$.

Now, the second factor above can easily be approximated using Yaglom's theorem: when $T$ is large,
\begin{multline}\label{heuristiceq2}
\frac{\P[N_T (N_T-1)\ldots (N_T-k+1)]}{\P(N_T\ge k)} = \P[N_T (N_T-1)\ldots (N_T-k+1)|N_T\ge k]\\
\sim T^k\P[(N_T/T)^k | N_T>0] \sim T^k \P[\mathcal E^k]
\end{multline}
where $\mathcal E$ is an exponential random variable with parameter $2/\sigma^2$. Therefore, in order to describe the distribution of the tree drawn out by the $k$ uniformly sampled particles under $\P$ when $T$ is large, it suffices to understand the joint distribution of the tree drawn out by the $k$ spines together with $N_T$ under $\Q^{k,T}$ when $T$ is large.

Write $\tau_i = \tilde {\mathcal S}^k_i(T)/T$ for the scaled split times of the $k$ uniformly sampled particles, and $\tau_i^\xi$ for the scaled split times of the $k$ spine (unordered, in the sense that they are a random permutation of the ordered split times). We show (this is Lemma \ref{Qtopology} and the case $\mu=0$ of Proposition \ref{almostdensity}; see also the discussion in Section \ref{QTlargeTsec}) that in the limit as $T\to\infty$, under $\Q^{k,T}$ the times $(\tau^\xi_1,\ldots,\tau^\xi_{k-1})$ are uniform random variables on $[0,1]$, and the topology of the underlying tree has a certain topology, which is equivalent to the topology of Kingman's coalescent restricted to $k$ blocks. Here is a way of constructing such a tree, again in the same spirit as Aldous \cite[Section 4.2]{aldous:coalescence_review}: let $U_0,\ldots,U_{k-1}$ and $V_1,\ldots,V_{k-1}$ be independent uniform random variables on $[0,1]$. Also let $V_0=1$. Within the unit square, for each $0\le i\le k-1$, draw a line from $(U_i,0)$ to $(U_i,V_i)$. These lines represent the branches of our tree. Now, for each $1\le i\le k-1$, draw a horizontal line starting from $(U_i,V_i)$ towards $(U_0,V_i)$ but stopping as soon as it hits another (vertical) line. This is our description of the tree drawn out by the spines under $\Q^{k,T}$ as $T\to\infty$. (Note, as previously, that the particular choice of the $U_i$ is merely a convenient way to give a random permutation of the vertical lines; the scale on the horizontal axis has no meaning in this construction.)

\begin{figure}[h!]
  \centering
   \includegraphics[width=5cm]{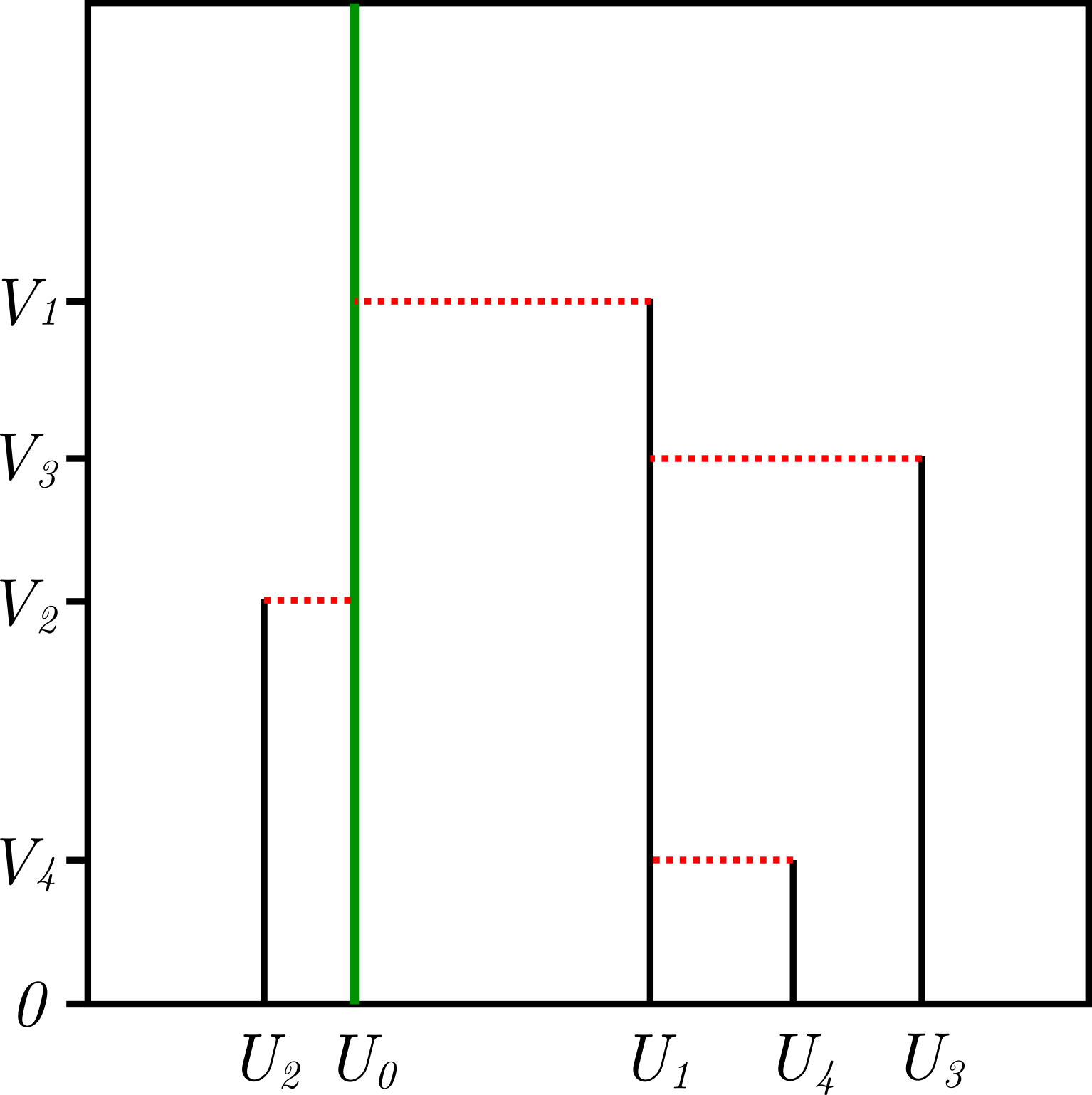}
   \vspace{0mm}
  \caption{\small{A probabilistic representation of the rescaled tree under $\Q^{5,T}$ for large $T$.}}
  \end{figure}

Now we explain how to observe the joint distribution of this tree and the total population size, given the description above. Under $\Q$, each spine---that is, each vertical line in our picture---behaves in the same way, giving birth to ordinary particles at a constant rate (independent of the number of marks following the spine); this can be seen from Lemma \ref{Qbots}. Thus the contribution to the total population of a vertical line of length $v$ in our picture is simply the contribution to the total population of a single spine that lived for time $vT$. It is immediate from the definition of $\Q^{1,vT}$ that a single spine results in a size-biasing of the total population size; by Yaglom's theorem, under $\P$, the total population size after time $vT$ is approximately $vT$ times an independent exponential random variable of parameter $2/\sigma^2$, and therefore under $\Q^{1,vT}$ the total population size is approximately $vT$ times an independent Gamma random variable of parameters $(2,2/\sigma^2)$.

Thus, the total population size $N_T$ under $\Q^{k,T}$ satisfies
\[\frac{N_T}{T} \to^{(d)} \sum_{i=0}^{k-1} V_i \Gamma_i\]
where the branch lengths $V_{1},\dots,V_{k-1}$ are independent $U[0,1]$ random variables, $V_0=1$, and $\Gamma_0,\dots,\Gamma_{k-1}$ are independent identically distributed $\Gamma(2,2/\sigma^2)$ random variables that are also independent of $V_0,\ldots,V_{k-1}$.

\begin{rmk}
Before we apply the description above to obtain an explanation of our results, let us make a further observation. Recall that for each $1\le i \le k-1$, $V_i$ is uniformly distributed on $[0,1]$. A uniform random variable multiplied by an independent $\Gamma(2,2/\sigma^2)$ random variable is exponentially distributed with parameter $2/\sigma^2$; that is, $\mathcal E_i:=V_i\Gamma_i\sim\mathrm{Exp}(2/\sigma^2)$ for $i=1,\dots,k-1$.
Finally, $V_0=1$, and therefore of course $V_0 \Gamma_0$ is distributed as the sum of two independent exponential random variables, say $\mathcal E_0$ and $\mathcal E_0^\prime$, each with parameter $2/\sigma^2$. Thus the total population size under $\Q^{k,T}$ is approximately $T$ times a sum of $k+1$ independent exponential random variables of parameter $2/\sigma^2$, or in other words, $T$ times a $\Gamma(k+1,2/\sigma^2)$ random variable. 
\end{rmk}

\begin{figure}[h!]
  \centering
      \includegraphics[width=5cm]{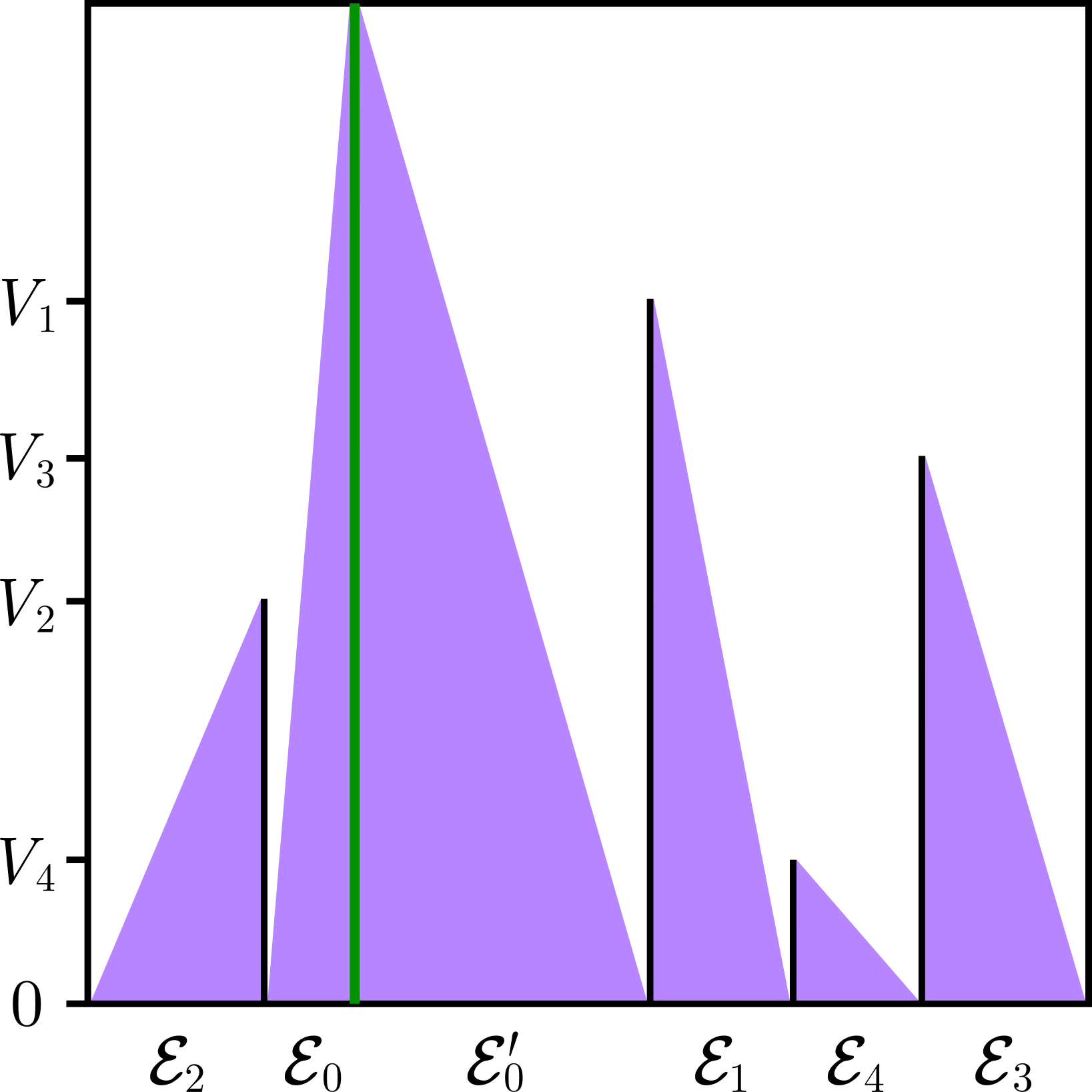}
   \vspace{0mm}
  \caption{\small{Each triangle represents the contribution towards the total population from particles that branched off the adjacent spine. The scale on the horizontal axis can now be interpreted as population size.}}
  \label{Qpopconst}
  \end{figure}

\begin{rmk}
It is also worth noting that size-biased exponential distributions give Gamma distributions. In fact, the exponential distribution can be characterised by relationships with its size-biased versions and uniform random variables; this was key in Ren, Song and Sun's proof of Yaglom's theorem using two spines in \cite{ren_song_sun:2spineGW}, and also appeared with a single spine in Lyons, Pemantle and Peres \cite{lyons_et_al:conceptual_llogl_mean_behaviour_bps}.
\end{rmk}

To complete the explanation of our results, continuing from \eqref{heuristiceq1} and \eqref{heuristiceq2}, we  now see that 
\begin{align*}
\P(\tau_1\in dt_1, \ldots, \tau_{k-1}\in dt_{k-1} | N_T \ge k) &\sim \Q^{k,T}\bigg[ \frac{\ind_{\{\tau_1^{\xi} \in dt_1, \ldots, \tau_{k-1}^\xi\in dt_{k-1}\}}}{N_T (N_T-1)\ldots (N_T-k+1)} \bigg] T^k \P[\mathcal E^k]\\
&\sim \P\bigg[ \frac{\ind_{\{1-V_1\in dt_1, \ldots, 1-V_{k-1}\in dt_{k-1}\}}}{T^k (\sum_{i=0}^{k-1} V_i \Gamma_i)^k} \bigg] T^k \P[\mathcal E^k]\\
&= \P\bigg[ \frac{1}{(\sum_{i=0}^{k-1} (1-t_i) \Gamma_i)^k}\bigg] \P[\mathcal E^k] \,dt_1\,\ldots\,dt_{k-1}.
\end{align*}
We now observe that for any $\alpha>0$,
\[\frac{1}{\alpha^k} = \frac{1}{(k-1)!} \int_0^\infty z^{k-1} e^{-\alpha z} dz.\]
Applying this fact, we get that 
\begin{align*}
\P(\tau_1\in dt_1, \ldots, \tau_{k-1}\in &\, dt_{k-1} \,|\, N_T \ge k)\\
&\sim \P\Big[\frac{1}{(k-1)!} \int_0^\infty z^{k-1}e^{-z \sum_{i=0}^{k-1} (1-t_i)\Gamma_i} dz \Big]\P[\mathcal E^k] \,dt_1\,\ldots\,dt_{k-1}\\
&= \frac{1}{(k-1)!} \int_0^\infty z^{k-1}\prod_{i=0}^{k-1} \frac{1}{(1+\frac{\sigma^2}{2}(1-t_i)z)^2} dz\,k!\Big(\frac{\sigma^2}{2}\Big)^k \,dt_1\,\ldots\,dt_{k-1}\\
&= k \int_0^\infty z^{k-1}\prod_{i=0}^{k-1} \frac{1}{(1+(1-t_i)z)^2} dz\,dt_1\,\ldots\,dt_{k-1}.
\end{align*}
Indeed, this is the joint density of the coalescent times in the critical case as given in Section \ref{limitconstructionsec}, and consistent with the construction in Theorem \ref{maxconcrit}. Further, integrating gives the joint distribution function in Theorem \ref{critBDthm}.

Note that in near-critical cases a similar picture will hold, although the distribution of the rescaled spine split times will not be uniform and will have a density that is proportional to $e^{r\mu(1-s)}$ for $s\in[0,1]$. See Section \ref{nearcritsec} for more details.

\section{Further discussion of the results}\label{chatsec}

In this section we seek to give further understanding of our scaling limit, compare it to known results, and to explore other ways of obtaining similar representations; in order to keep the calculations to a reasonable length, at times we will not worry too much about the technical details. We will return to full rigour in Sections \ref{CoMsec}, \ref{BDsec} and \ref{nearcritsec}, in order to prove our main results.

\subsection{Comparison to known formulas}\label{comparisonsec}

As mentioned in the introduction, the critical case $\mu=0$ has been investigated by other authors. Athreya \cite{athreya:coalescence} gave an implicit description of the distributional limit of $\S(T)/T$. (In fact he worked with discrete-time Galton-Watson processes, but this makes no difference in the limit, and we will continue to use our continuous-time terminology and notation for ease of comparison.) By considering the numbers of descendants at time $T$ of particles alive at an earlier time $sT$, Athreya showed that
\[\lim_{T\to\infty} \P(\S(T)/T < s \,|\, N_T\ge 2) = 1 - E[\phi(G_s)]\]
where $G_s$ satisfies $P(G_s = j) = (1-s)s^{j-1}$ for $j\ge 1$, and
\[\phi(j) = E\bigg[\frac{\sum_{i=1}^j \eta_i^2}{(\sum_{i=1}^j \eta_i)^2}\bigg]\]
where $\eta_1,\eta_2,\ldots$ are independent exponential random variables of parameter $1$.

We check that this description of the scaling limit agrees with our own formula \eqref{critk2}.

\begin{lem}
With $\phi$ and $G_s$ as described above,
\[E[\phi(G_s)] = \frac{2(s-1)}{s^2}\big(\log(1-s)+s\big).\]
\end{lem}

\begin{proof}
Suppose first that we are given $\eta_1,\ldots,\eta_j$. Let $\gamma_j = \sum_{i=1}^j \eta_i$, and let $U_1$ and $U_2$ be independent uniform random variables on $(0,\gamma_j)$. Then for each $l$, $(\eta_l/\sum_{i=1}^j \eta_i)^2$ is the probability that both $U_1$ and $U_2$ fall within the interval $(\gamma_{l-1},\gamma_l)$. Therefore  $(\sum_{i=1}^j \eta_i^2)/(\sum_{i=1}^j \eta_i)^2$ is the probability that for some $l\le j$, both $U_1$ and $U_2$ fall within the interval $(\gamma_{l-1},\gamma_l)$.

Suppose now that we are given only the value of $\gamma_j$, and let $\tilde\gamma_1,\ldots,\tilde\gamma_{j-1}$ be a uniform permutation of $\gamma_1,\ldots,\gamma_{j-1}$. Since $\gamma_1,\gamma_2,\ldots$ can be viewed as the arrival times of a Poisson process of parameter $1$, we know that given $\gamma_j$, the random variables $\tilde\gamma_1,\ldots,\tilde\gamma_{j-1},U_1,U_2$ are independent uniform random variables on $(0,\gamma_j)$. Therefore the probability that $U_1$ and $U_2$ both fall within the interval $(\tilde\gamma_{l-1},\tilde\gamma_l)$ for some $l$ is exactly $2/(j+1)$. Since this does not depend on the value of $\gamma_j$, we get immediately that $\phi(j) = 2/(j+1)$.

Summing over the possible values of $G_s$, we get
\begin{multline*}
E[\phi(G_s)] = \sum_{j=1}^\infty \frac{2}{j+1}(1-s)s^{j-1} = 2\frac{(1-s)}{s^2}\sum_{j=1}^\infty \frac{s^{j+1}}{j+1} = 2\frac{(1-s)}{s^2} \int_0^s \frac{u}{1-u}\,\d u\\
= 2\frac{(1-s)}{s^2}\Big(\log\Big(\frac{1}{1-s}\Big)-s\Big) = 2\frac{(s-1)}{s^2}\big(\log(1-s)+s\big).\qedhere
\end{multline*}
\end{proof}

Durrett \cite{durrett:genealogy} also gave a description of the limit $\S(T)/T$ in the critical case, showing that
\[\lim_{T\to\infty} \P(\S(T)/T > s \,|\, N_T\ge 2) = (1-s)\bigg(1+2\sum_{j=1}^\infty \frac{s^j}{j+2}\bigg).\]
It is easy to expand our formula \eqref{critk2} as a power series and check that it agrees with the above. Durrett, in fact, went on to give power series expressions for the distributions of $\S^3_1$ and $\S^3_2$. He further stated that it was ``theoretically'' possible to calculate distributions of split times for $k>3$, and also mentioned that he could derive a joint distribution for $\S^3_1$ and $\S^3_2$, again in power series form, but that ``we would probably not obtain a useful formula''. This makes clear the advantage of our method, which gives explicit formulas for the joint distribution for each $k$ without going through an interative procedure.

O'Connell \cite{oconnell:genealogy_mrca} gave exactly the formula in our Example \ref{nearcritk2}, the near-critical scaling limit in the case $k=2$. He also provided a very interesting application to a biologically motivated problem: how long ago did the most recent common ancestor of all humans live?

In subcritical and supercritical cases, it is impossible to give such explicit results in generality as the genealogical structure of the tree depends on the detail of the offspring distribution. However one can characterize the distribution of the split times using integral formulas involving the generating function of the offspring distribution. Lambert \cite{lambert:coalescence_GW} (in discrete time) and Le \cite{le:coalescence_GW} (in continuous time) did this in the case $k=2$ for quite general Galton-Watson processes. They also investigated the case $k\ge 3$, but gave only an implicit representation for the joint distribution of the split times. More recently Grosjean and Huillet \cite{grosjean_huillet:genealogy_coalescence} and Johnston \cite{johnston:coalescence_subcrit_supercrit} gave detailed answers for general $k$.

Donnelly and Kurtz \cite[Theorem 5.1]{donnelly_kurtz:particle_measure_popn_models} showed that the genealogy of the Feller diffusion is a time-change of Kingman's coalescent, in which the rate at which two lineages merge is inversely proportional to the population size. The Feller diffusion started from $x$ is itself the scaling limit of a critical Galton-Watson process started with a population of size $\lfloor Nx\rfloor$, so taking a limit as $x\downarrow 0$ one might expect to be able to recover our results. However, finding the marginal distribution of the coalescent times---that is, \emph{not} conditional on the population size---is highly non-trivial, as the two quantities are so closely connected; this can be seen in \eqref{heuristiceq1}, for example. We manage to overcome this serious difficulty by decoupling the dependence between the population size and the split times via the measure $\Q^{k,T}$, which adjusts for the varying population size whilst simultaneously ensuring the $k$ spines form a uniform sample without replacement from population at time $T$.

Besides being more difficult, the question of understanding the distribution of the coalescent tree drawn out by a sample from a large population, without knowing the population size, appears to be more natural from the point of view of biological applications.

Indeed, whilst the formulae for the genealogies in near-critical Galton-Watson processes look complicated, they are nevertheless explicit, they have simple constructions, and the underlying natural branching model allows the population to vary randomly with time. In this latter respect, the structure obtained is significantly different from under fixed sized population assumptions. It is hoped that our results may eventually prove useful in applications, for example using computational methods to fit these genealogical models to real data.

\subsection{Contour processes and the continuum random tree}\label{crt_sec}

Athreya \cite{athreya:coalescence} mentioned that his result could alternatively be obtained by representing the limiting random trees with Brownian excursions. We give a non-rigorous discussion of this approach.

It is known that a critical Galton-Watson tree conditioned to survive until time $T$ converges, as $T\to\infty$ (in a suitable topology), to a \emph{continuum random tree}. There is a vast literature, beginning with Aldous \cite{aldous:CRTI}, on continuum random trees as the scaling limit of various discrete structures. For our rough discussion we can think of drawing our tree, conditioned to survive to time $T$ and renormalised by $T$, and tracing a contour around it starting from the root and proceeding in a depth-first manner from left to right. The height of that contour process converges as $T\to\infty$ to a Brownian excursion $(B_t)_{t\in[0,\nu]}$ conditioned to reach height $1$. It is easy to see that two points $u,v\in[0,\nu]$ correspond to the same ``vertex'' in the limiting tree if they are at the same height and the excursion between $u$ and $v$ is always above $B_u$; i.e. $B_u=B_v=\inf_{t\in(u,v)} B_t$. The total population of the tree at time $sT$ corresponds to the local time of the Brownian excursion at level $s$. Choosing two particles at time $T$, then, means picking two points on the excursion at height $1$ according to the local time measure; and the two particles have a common ancestor at time $t$ if the two points chosen are in the same sub-excursion above height $t$.

   \begin{figure}[h!]
   \centering
    \includegraphics[width=11cm]{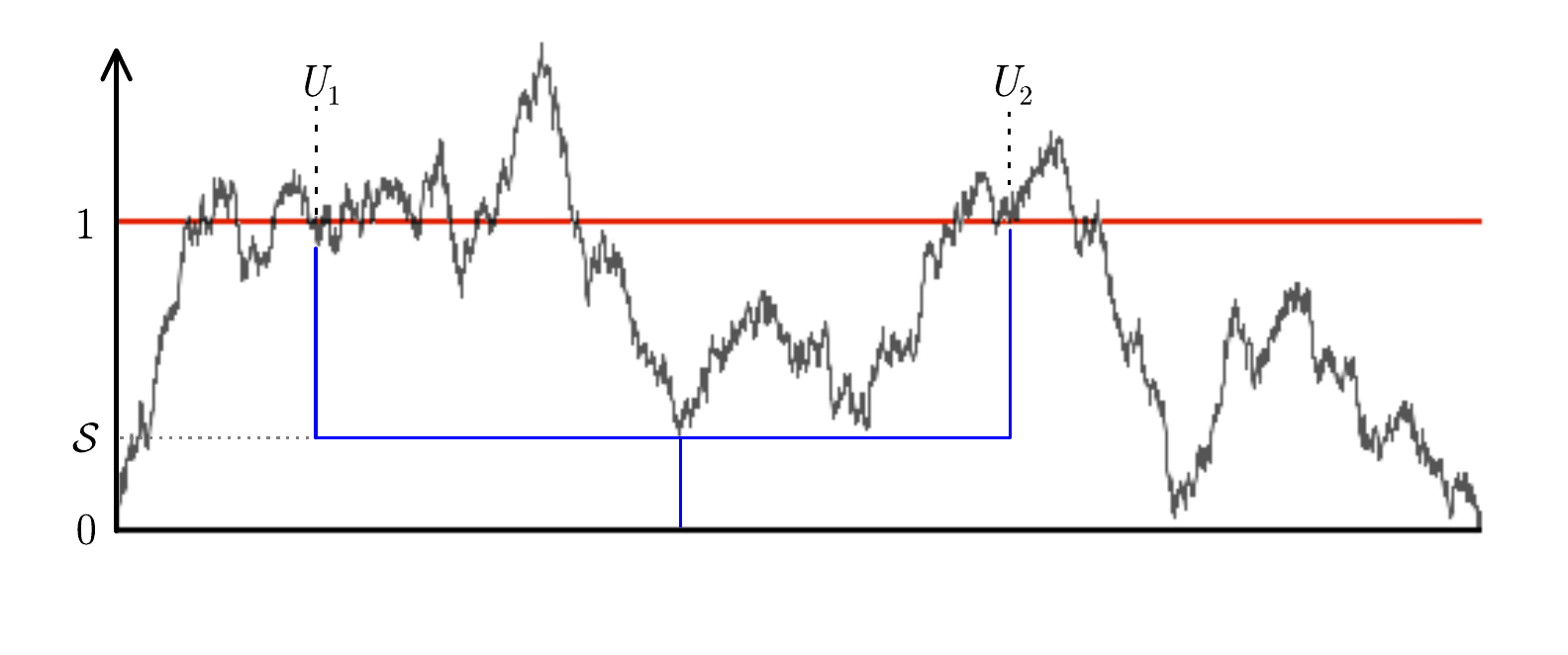}
    \vspace{-8mm}
   \caption{\small{A Brownian excursion conditioned to reach height $1$. Two points $U_1$ and $U_2$ are chosen uniformly according to local time at height $1$, and the induced tree is drawn below the excursion. The split time of the two particles is denoted by $\S$.}}
   \end{figure}

In order to calculate the probability of this last event, we (obviously) need to know a little about Brownian excursions. Excursions, indexed by local time, occur according to a Poisson point process with intensity Lebesgue $\times\,\, n$ for some excursion measure $n$. This measure $n$ satisfies $n(\sup_t f(t)>a ) = \frac{1}{2a}$; and the local time at $0$ when the Brownian motion first hits $-\delta$ is exponentially distributed with parameter $\frac{1}{2\delta}$. See for example \cite{rogers:guided_tour_excursions}.

Take a Brownian excursion conditioned to reach height $1$, and choose two points $U_1$ and $U_2$ at height $1$ uniformly according to local time measure. Let $L_1$ be the total local time at level $1$, and $L_U$ be the total local time between $U_1$ and $U_2$. The event that $U_1$ and $U_2$ are in the same sub-excursion above height $s$ is exactly the event that there is no excursion from level $1$ between $U_1$ and $U_2$ that goes below level $s$ (and stays above level $0$); by the facts about Brownian excursions above, given $L_U$, the number of such excursions is a Poisson random variable with parameter $L_U(\frac{1}{2(1-s)} - \frac{1}{2})$. Thus the probability that $U_1$ and $U_2$ are in the same sub-excursion above height $s$ is
\[\int_0^\infty \P(L_1 \in \d x) \int_0^x \P( L_U \in \d y \,|\, L_1 = x) e^{-y(\frac{1}{2(1-s)}-\frac{1}{2})}.\]
The local time $L_1$ is exponential of parameter $1/2$, and it is easy to check that the density of the distance between two uniform random variables on $(0,x)$ is $2(x-y)/x^2$. Thus the above equals
\[\int_0^\infty \frac{1}{2} e^{-x/2} \int_0^x \frac{2(x-y)}{x^2} e^{-y(\frac{1}{2(1-s)}-\frac{1}{2})}\, \d y\, \d x.\]
Making the substitution $z=y/x$ and changing the order of integration, we get
\[\int_0^1 (1-z) \int_0^\infty e^{-\frac12 x (1+\frac{z}{1-s}-z)}\,\d x\,\d z,\]
and it is then easy to integrate directly to get that the limiting split time $\S$ satisfies
\[\P(\S\ge s) = 2\Big(\frac{s-1}{s^2}\Big)(\log(1-s)+s)\]
which agrees with \eqref{critk2}.

Applying this sophisticated machinery works well (at least if we do not worry too much about the technical details) in this simple case. However it becomes much more difficult to generalise these techniques to obtain the joint distribution of the split times for three particles, rather than just two; let alone the general formula for $k$ particles that appeared in Theorem \ref{nearcritthm}.

Popovic \cite{popovic:asymptotic_genealogy_critical_bp} used the following observation. Condition on the event that there are exactly $k$ particles alive at time $T_k$, so that the $k$ particles we choose comprise the whole population, then rescale by $T_k$ and let $k\to\infty$. If $T_k/k\to t$, then the contour process converges to a Brownian excursion conditioned to have local time $1$ at level $t$; and the split times are then governed by the entire collection of excursions below level $t$. These excursions form a Poisson point process with an explicit intensity measure. This allowed Popovic to give some very interesting results about critical processes, and similar techniques were built upon in various ways by her and other authors \cite{aldous_popovic:critical_bp_biodiversity, gernhard:conditioned_reconstructed, lambert:contour_splitting_trees, lambert:coalescent_branching_trees}. Although these are certainly related to our investigation, they often look at the entire population alive at time $T$, rather than sampling a fixed number of individuals, which results in a different scaling regime. Biological motivation for why we might like to sample a fixed number of individuals from a growing population---that is, our regime---can be found in \cite{oconnell:genealogy_mrca}.

After this article was released, Lambert \cite{lambert:genealogy_binary} constructed a remarkable method for obtaining some of our formulas from contour processes. Given a branching process whose population at time $T$ is geometrically distributed (for example a birth-death process), the work in \cite{lambert_stadler:birth_death_cpp} allows one to sample each particle at time $T$ independently with some fixed probability $y\in(0,1)$ and reconstruct the genealogical tree of the sampled particles. By taking $y$ to be a realisation of a carefully chosen improper random variable $Y$, and conditioning the resulting number of particles sampled to be exactly $k$, in \cite{lambert:genealogy_binary} Lambert  produces our Proposition \ref{BDdensity}. However, constructing the correct (improper) distribution for $Y$ would have been extremely difficult without prior knowledge of the answers provided by our results.

Lambert's results in \cite{lambert:genealogy_binary} are for a large class of processes known as \emph{coalescent point processes}. However, coalescent point processes necessarily have geometrically distributed population sizes. As Lambert says in \cite{lambert:genealogy_binary}, ``we consider here possibly non-Markovian and time-inhomogeneous branching processes, but always binary." For Galton-Watson processes, this means only our birth-death process results are in common with Lambert's coalescent point process results in \cite{lambert:genealogy_binary}. In a more recent private communication, Lambert has told us that he can carry out his construction even in non-binary cases, and that his results hold beyond geometrically distributed population sizes.

Another advantage of our approach is that it does not require a Markovian contour process, and has the potential to be generalised, for example, to Galton-Watson processes with infinite variance, or spatial branching processes. We plan to carry out some of these generalisations in future work.

\subsection{Purple trees}

For a moment forget about the scaling limit, and consider a birth-death process (that is, fix $\alpha\ge 0$ and $\beta>0$, and suppose that $r = \alpha+\beta$, $p_0 = \alpha/(\alpha+\beta)$ and $p_2 = \beta/(\alpha+\beta)$, with $p_j=0$ for $j\neq 0,2$). Wait until time $T$, and then colour any particle that has a descendant alive at time $T$ purple, and any particle whose descendants all die before time $T$ red.

The purple tree, often called the \emph{reduced tree} in the literature, was first introduced by Fleischmann and Siegmund-Schultze \cite{fleischmann_ss:reduced_GW}. The reduced tree has been used in several of the references given in Section \ref{comparisonsec}, in particular O'Connell \cite{oconnell:genealogy_mrca}. On a related note, Harris, Hesse and Kyprianou \cite{harris_hesse_kyprianou:bbm_strip} considered a supercritical branching process and coloured any particle whose descendants survived forever blue, and anyone whose descendants all died out red. Of course red particles in our picture are also red in the Harris-Hesse-Kyprianou picture, whereas each of our purple particles may be either red or blue in their colouring.

Now suppose that, rather than running the birth-death process until time $T$ and then colouring all the particles, we want to construct the coloured picture dynamically as the process evolves. If we start with one particle and condition on the process surviving until time $T$, then the first particle is certainly purple, since at least one of its descendants must survive.

Let $p_t = \P(N_t=0)$. Using generating functions one can show that
\[p_t = \frac{\alpha e^{(\beta-\alpha)t}-\alpha}{\beta e^{(\beta-\alpha)t}-\alpha}, \hspace{8mm} 1-p_t = \frac{(\beta-\alpha)e^{(\beta-\alpha)t}}{\beta e^{(\beta-\alpha)t}-\alpha};\]
see Section \ref{BDgenfuncts} for details.

If a purple particle branches at time $s$, then its two children could be either both purple, or one red and one purple. The probability that they are both purple must be 
\[\frac{(1-p_{T-s})^2}{1-p_{T-s}^2},\]
corresponding to the probability that \emph{both} descendancies survive given that at least one does. The probability that one is purple and one is red must similarly be
\[\frac{2p_{T-s}(1-p_{T-s})}{1-p_{T-s}^2}.\]
One can check from \cite{harris_hesse_kyprianou:bbm_strip} that purple particles branch at rate $\beta(1+p_{T-s})$ at time $s$, and red particles branch at rate $\beta p_{T-s}$ at time $s$. In particular purple particles give birth to new purple particles at rate
\[\beta(1+p_{T-s})\cdot \frac{(1-p_{T-s})^2}{1-p_{T-s}^2} = \beta(1-p_{T-s}).\]
Similar calculations can be done generally, rather than just for birth-death processes. However it is easy to see that in any near-critical cases the probability that a purple particle has more than two purple children at any branching event will tend to zero, so in a sense the important information is captured by the simpler birth-death calculations. Indeed we saw in Theorem \ref{nearcritthm} that in our scaling limit, only the mean of the branching process really matters; and we will see again in Lemma \ref{splitsdistinct} that only binary splits appear in the limit. For this intuitive discussion we therefore carry out our calculations only in the birth-death case only. 

Of course, to understand the coalescent structure of the tree drawn out by particles chosen at time $T$, we can ignore the red particles; only the purple tree matters. Let us now return to a near-critical scaling limit by assuming that $\beta = \alpha + \gamma/T$ for some $\gamma\neq 0$. At time $sT$, the purple tree branches at rate
\[\beta(1-p_{T-sT}) = \frac{\beta\gamma e^{\gamma (1-s)}/T}{\beta e^{\gamma(1-s)} - (\beta - \gamma/T)} = \frac{\gamma e^{\gamma(1-s)}}{T(e^{\gamma(1-s)}-1)} \cdot \Big(1-\frac{\gamma}{\beta T(e^{\gamma(1-s)}-1+\frac{\gamma}{\beta T})}\Big).\]
Scaling time $[0,T]$ onto $[0,1]$ and considering the large $T$ limit, at time $s\in(0,1)$ the purple tree undergoes binary branching at rate
\begin{equation}
\lim_{T\rightarrow\infty} T\,\beta(1-p_{T(1-s)})=\frac{\gamma e^{\gamma(1-s)}}{e^{\gamma(1-s)}-1}.
\label{purplerate}
\end{equation}
Thus, since
\[\int_0^t \frac{\gamma e^{\gamma(1-s)}}{e^{\gamma(1-s)}-1}\, \d s = \int_{e^{\gamma(1-t)}}^{e^{\gamma}} \frac{1}{u-1}\,\d u = \log\Big(\frac{e^\gamma-1}{e^{\gamma(1-t)}-1}\Big),\]
we see that the purple tree in the near-critical scaling limit is the same as a Yule tree (binary branching at rate $1$) observed under the time change
\[t \mapsto \log\Big(\frac{e^\gamma-1}{e^{\gamma(1-t)}-1}\Big).\]

Following the same route in the purely critical case $\alpha=\beta$ gives that the rescaled purple tree binary branches at rate $1/(1-s)$, which corresponds to a Yule tree under the time change $t\mapsto -\log(1-t)$.

These rough calculations help to explain the similarities between our formulas in the near-critical scaling limit (Theorem \ref{nearcritthm}) and in the birth-death process (Theorem \ref{noncritBDthm}). 
In particular, for the coalescence behaviour, only the purple tree matters. In the large time $T$ limit, only binary branching occurs in the purple tree, since the chance of any purple particle having more than one other purple offspring at a time (or in close proximity) becomes negligible. Further, the purple branching rate is given by the limit of the original branching rate weighted by the probability of survival, that is $\lim_{T\rightarrow\infty} T \beta(1-p_{T(1-s)})$, as calculated above, and this rate corresponds to a simple deterministic time change of a Yule tree in all near-critical cases.

An anonymous referee pointed out to us that Theorem 2.2 of \cite{oconnell:genealogy_mrca} gives an apparently incorrect formula in place of our \eqref{purplerate}, although the main Theorem 2.3 of \cite{oconnell:genealogy_mrca} is nevertheless correct.

\section{Spines and changes of measure}\label{CoMsec}

In this section we lay down many of the technical tools that we will need to prove the results in the previous sections. Our two most important signposts will be Proposition \ref{firstprop}, which translates questions about uniformly chosen particles under $\P$ into calculations under a new measure $\Q$; and Proposition \ref{Qmgf_gen_prop}, which is a version of Campbell's formula under $\Q$ which will be central to our analysis.

First, of course, we must introduce $\Q$, and we begin by describing the idea of \emph{spines}, which introduce extra information into our tree by allocating \emph{marks} to certain special particles. Spine methods are now well known and a thorough treatment can be found for example in \cite{harris_roberts:many_to_few}. We give only a brief introduction.

\subsection{The $k$-spine measure $\P^k$}\label{pk_description}

We define a new measure $\P^k$ under which there are $k$ distinguished lines of descent, which we call spines. Briefly, $\P^k$ is simply an extension of $\P$ in that all particles behave as in the original branching process; the only difference is that some particles carry marks showing that they are part of a spine.

Under $\P^k$ particles behave as follows:
\begin{itemize}
\item We begin with one particle which carries $k$ marks $1,2,\ldots,k$.
\item We think of each of the marks $1,\ldots,k$ as distinguishing a particular line of descent or ``spine'', and define $\xi^i_t$ to be the label of whichever particle carries mark $i$ at time $t$.
\item A particle carrying $j$ marks $b_1 < b_2 < \ldots < b_j$ at time $t$ branches at rate $r$, dying and being replaced by a random number of particles according to the law of $L$, independently of the rest of the system, just as under $\P$.
\item Given that $a$ particles $v_1,\ldots,v_a$ are born at a branching event as above, the $j$ marks each choose a particle to follow independently and uniformly at random from amongst the $a$ available. Thus for each $1\leq l\leq a$ and $1\leq i \leq j$ the probability that $v_l$ carries mark $b_i$ just after the branching event is $1/a$, independently of all other marks.
\item If a particle carrying $j>0$ marks $b_1 < b_2 < \ldots < b_j$ dies and is replaced by 0 particles, then its marks are transferred to the graveyard $\Delta$.
\end{itemize}

Again we emphasise that under $\P^k$, the system behaves exactly as under $\P$ except that some particles carry extra marks showing the lines of descent of $k$ spines. We write $\xi_t = (\xi^1_t,\ldots,\xi^k_t)$. Obviously $\xi_t$ depends on $k$ too, but we omit this from the notation.

We let $n_t$ be the number of distinct spines (i.e.~the number of particles carrying marks) at time $t$, and for $i\ge 1$
\[\psi_i = \inf\{t\ge 0 : n_t\not\in \{1,\ldots, i\}\}\]
with $\psi_0=0$. We view $\psi_i$ as the $i$th spine split time (although, for example, the first and second spine split times may be equal---corresponding to marks following three different particles at the first branching event). We also let $\rho^i_t$ be the number of marks following spine $i$.

The set of distinct spine particles at any time $t$, and the marks that are following those spine particles, induce a partition $\mathcal Z^k_t$ of $\{1,\ldots,k\}$. That is, $i$ and $j$ are in the same block of $\mathcal Z^k_t$ if $\xi^i_t = \xi^j_t$. If we then let
\[Z^k_i = \mathcal Z^k_{\psi_i}\]
for $i=0,\ldots,k-1$, we have created a discrete collection of partitions $Z_0,Z_1,\ldots,Z_{k-1}$ which describe the topological information about the spines without the information about the spine split times. It will occasionally be useful to use the $\sigma$-algebra $\mathcal H' = \sigma(Z_0,Z_1,\ldots)$.

For any particle $u\in\Nc_t$, there exists a last time at which $u$ was a spine (which may be $t$). If this time equals $\psi_i$ for some $i$, then we say that $u$ is a \emph{residue} particle; if it does not equal $\psi_i$ for any $i$, and $u$ is not a spine, then we say that $u$ is \emph{ordinary}. Each particle is exactly one of residue, ordinary, or a spine.

Of course $\P^k$ is not defined on the same $\sigma$-algebra as $\P$. We let $\F^k_t$ be the filtration containing all information about the system, including the $k$ spines, up to time $t$; then $\P^k$ is defined on $\F^k_\infty$. For more details see \cite[Section 5]{harris_roberts:many_to_few}. Let $\F^0_t$ be the filtration containing only the information about the Galton-Watson tree. Let $\Gt^k_t$ be the filtration containing all the information about the $k$ spines (including the birth events along the $k$ spines) up to time $t$, but none of the information about the rest of the tree. Finally let $\G^k_t$ be the filtration containing information only about spine splitting events (including which marks follow which spines); $\G^k_t$ does not know when births of ordinary particles from the spines occur.

  \begin{figure}[h!]
  \centering
   \includegraphics[width=11cm]{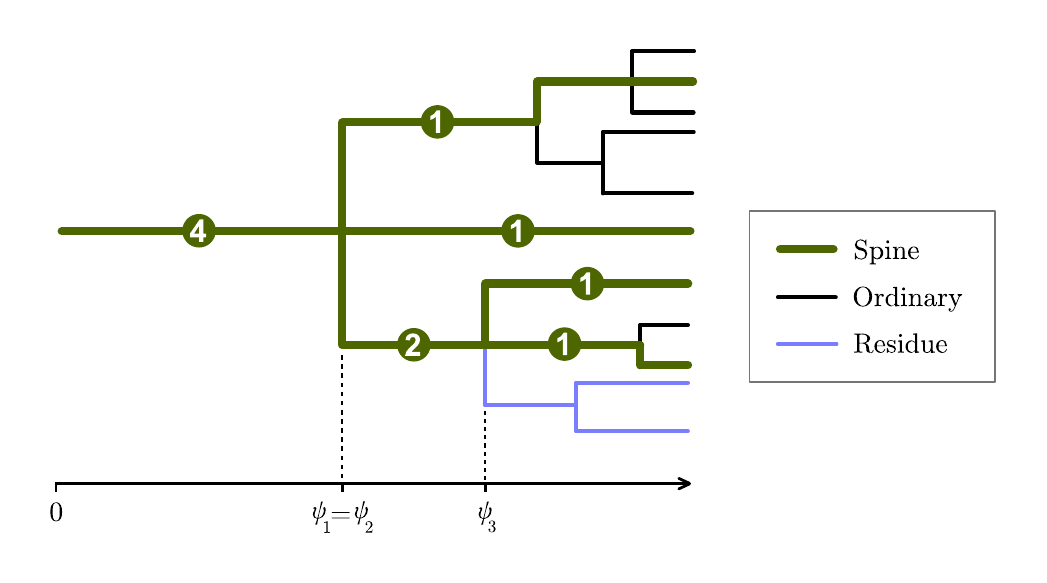}
   \vspace{-5mm}
  \caption{\small{Spines, ordinary particles and residue particles. The horizontal axis represents time. The numbers show how many marks are carried by each spine.} \label{skelfig}}
  \end{figure}

\subsection{A change of measure}

We will now introduce a new measure. Under this measure, the $k$ spines will be uniformly chosen (without replacement) at time $T$, which will allow us to represent uniformly chosen particles under $\P$ as calculations using the spines under our new measure. This very natural new measure has some remarkable properties, including the fact that it can be fully described forwards in time. Without this new measure we found calculating with uniformly chosen particles to be intractable.

Throughout the rest of this section we fix $k\ge 1$ and assume that $\P[L^k]<\infty$. This condition will be relaxed later, but it is required even to define our changed measure.

For any set $S$ and $k\ge 1$, let $S^{(k)}$ be the set of distinct $k$-tuples from $S$, and for $n\ge0$, write
\[n^{(k)} = \begin{cases} n(n-1)(n-2)\ldots (n-k+1) & \hbox{ if } n\ge k\\ 0 &\hbox{ otherwise.}\end{cases}\]
Note that $|S^{(k)}| = |S|^{(k)}$.
For $t\ge 0$, define
\[g_{k,t} := \ind_{\{\xi^i_t \neq \xi^j_t \,\forall i\neq j\}} \prod_{i=1}^k \prod_{v<\xi^i_t} L_v\]
and
\[\zeta_{k,t} := \frac{g_{k,t}}{\P[N^{(k)}_t]}.\]

\begin{lem}\label{gproj}
For any $t\ge 0$,
\[\P^k[g_{k,t}|\F^0_t] = N^{(k)}_t.\]
In particular, $\P^k[\zeta_{k,t}] = 1$.
\end{lem}

\begin{proof}
\begin{align*}
\P^k[g_{k,t}|\F^0_t] &= \P^k\bigg[ \sum_{u\in\Nc_t^{(k)}} \ind_{\{\xi_t = u\}} \prod_{i=1}^k \prod_{v<u_i} L_v \,\bigg|\, \F^0_t\bigg]\\
&= \sum_{u\in\Nc_t^{(k)}} \Big(\prod_{i=1}^k \prod_{v<u_i} L_v\Big) \P^k(\xi_t = u \, |\, \F^0_t).
\end{align*}
Recall that the marks act independently, and at each branching event choose uniformly amongst the available children. Therefore
\begin{equation}\label{Pspineprob}
\P^k(\xi_t = u \, |\, \F^0_t) = \prod_{i=1}^k \P^k(\xi^i_t = u_i\,|\,\F^0_t) = \prod_{i=1}^k \prod_{v<u_i} \frac{1}{L_v}.
\end{equation}
Thus
\[\P^k[g_{k,t}|\F^0_t] = \sum_{u\in\Nc_t^{(k)}} 1 = |\Nc_t^{(k)}| = N_t^{(k)}.\]
This gives the first part of the result, and taking expectations gives the second.
\end{proof}

We now fix $T>0$ and define a new probability measure $\Q^{k,T}$ by setting
\begin{equation}
\left.\frac{\d \Q^{k,T}}{\d\P^k}\right|_{\F^k_T} :=  \frac{\ind_{\{\xi^i_T \neq \xi^j_T \,\forall i\neq j\}} \prod_{i=1}^k \prod_{v<\xi^i_T} L_v}{\P[N_T(N_T-1)\dots(N_T-k+1)]}
=\zeta_{k,T}
\label{Qmeasure}
\end{equation}
Often, when the choice of $T$ and $k$ is clear, we write $\P$ instead of $\P^k$ (since $\P^k$ is an extension of $\P$ this should not cause any problems) and $\Q$ instead of $\Q^{k,T}$. 
Then, by Lemma \ref{gproj},
\begin{equation}\label{RadNik0}
\left.\frac{\d \Q^{k,T}}{\d\P^k}\right|_{\F^0_T} = 
\frac{N_T(N_T-1)\dots(N_T-k+1)}{\P[N_T(N_T-1)\dots(N_T-k+1)]}
=
\frac{N_T^{(k)}}{\P[N_T^{(k)}]}=:Z_{k,T}.
\end{equation}
To see why the measure $\Q^{k,T}$ will be useful to us, we show how to translate questions about particles sampled uniformly without replacement under $\P$ into questions about the spines under $\Q$.
\begin{prop}\label{firstprop}
Suppose that $f$ is a measurable functional of $k$-tuples of particles at time $T$. Then
\[\P\Big[\frac{1}{N_T^{(k)}} \sum_{u\in \Nc_T^{(k)}} f(u) \, \Big| \, N_T \ge k\Big] = \frac{\P[N_T^{(k)}]}{\P(N_T\ge k)(k-1)!} \int_0^\infty (e^z-1)^{k-1} \Q^{k,T}\Big[e^{-zN_T} f(\xi_T)\Big]\,\d z.\]
\end{prop}
We defer the proof of this result to section \ref{Qfirstpropsec}.

\subsection{Description of $\Q^{k,T}$}\label{initialQsec}

In this section, we give a full description of the measure $\Q^{k,T}$. We defer the proofs to section \ref{Qproofsec}.

Our first lemma states that $\Q^{k,T}$ satisfies a time-dependent Markov branching property, in that the descendants of any particle behave independently of the rest of the tree.

\begin{lem}[Symmetry lemma]\label{symm}
Suppose that $v\in\Nc_t$ is carrying $j$ marks at time $t$. Then, under $\Q^{k,T}$, the subtree generated by $v$ after time $t$ is independent of the rest of the system and behaves as if under $\Q^{j,T-t}$.
\end{lem}

We already know from \eqref{QgivenG} and the discussion following it that particles that are not spines behave exactly as under $\P^k$: they branch at rate $r$ and have offspring distribution $L$. The behaviour of the spine particles is more complicated.

Recall that $\tau_\emptyset$ is the first branching event, and $\psi_1$ is the time of the first spine splitting event, i.e.
\[\psi_1 = \inf\{t\ge 0 : \exists i, j \hbox{ with } \xi^i_t \neq \xi^j_t\}.\]
(Note that if the spines die without giving birth to any children, this counts as a splitting event.)
By the symmetry lemma, in order to understand the split times under $\Q$, it suffices to understand the distributions of $\tau_\emptyset$ and $\psi_1$.

\begin{lem}\label{Qbots}
For any $t\in[0,T]$ and $k\ge 0$, we have
\[\Q^{k,T}(\tau_\phi > t) = \frac{\P^k[N^{(k)}_{T-t}]}{\P^k[N^{(k)}_T]} e^{-rt},\]
\[\Q^{k,T}(\psi_1 > t) = \frac{\P^k[N^{(k)}_{T-t}]}{\P^k[N^{(k)}_T]} e^{(m-1)rt},\]
and
\[\Q^{k,T}(\tau_\phi > t | \psi_1 > t) = e^{-mrt}.\]
\end{lem}

The third part of Lemma \ref{Qbots} combined with the symmetry lemma (Lemma \ref{symm}) tells us the following: given $\G^k_T$ (the information only about spine splitting events), under $\Q^{k,T}$ each spine gives birth to non-spine particles according to a Poisson process of rate $mr$, independently of everything else. In particular when there are $n$ distinct spines alive, there are $n$ independent Poisson point processes and the total rate at which non-spine particles are immigrated along the spines is $nmr$.

We call birth events that occur along the spines, but which do not occur at spine splitting events, \emph{births off the spine}. The following lemma tells us the distribution of the number of children born at such events.

\begin{lem}\label{sizebiased}
For any $j\ge0$, $k\ge 1$ and $0\le t<T$,
\[\Q^{k,T}(L_\emptyset = j | \tau_\emptyset = t, \, \psi_1>t) = \frac{jp_j}{m}.\]
\end{lem}

A random variable that takes the value $j$ with probability $j p_j/m$ for each $j$ is said to be \emph{size-biased} (relative to $L$). Lemma \ref{sizebiased} then tells us (in conjunction with the symmetry lemma) that births off any spine are always size-biased, no matter how many marks are following that particular spine. (The number of marks therefore only affects spine splitting events.)

To have a complete description of the behaviour of the process under $\Q^{k,T}$, it remains to understand how the marks distribute themselves amongst the available children at a spine splitting event. To do this, we write $\mathcal P_t^\xi$ for the partition of $\{1,\ldots,k\}$ induced by letting $i$ and $j$ be in the same block if the $i$th and $j$th spines are following the same particle at time $t$. By the symmetry lemma, again it suffices to consider the first spine splitting event.

\begin{lem}\label{Qdescription}
Conditional on $\{ \psi_1 > t \}$, the $\Q^{k,T}$-conditional probability that during the time interval $[t,t+h)$, the spine particle dies and gives birth to $l$ offspring, and at this time the marks are partitioned according to a partition $P$ with blocks of sizes $a_1,\ldots,a_n$, is given by
\[\Q^{k,T} \left( \psi_1 < t+h, \, \mathcal P^\xi_{\psi_1}=P,\, L_{\xi^1_t} = l  ~\Big|~ \psi_1 > t  \right) =  p_l l^{(n)} \frac{\prod_{i=1}^n \P^k [ N_{T-t}^{(a_i)} ]}{ \P^k[N_{T-t}^{(k)} ] } (rh + o(h)).\]
\end{lem}

For a collection of positive integers $a_1,\ldots,a_n$ whose sum is $k$, write
\[n_j = \# \{ i : a_i = j \}, ~~ j \geq 1.\]
(Note that $\sum_{j=1}^k n_j = n$ and $\sum_{j=1}^k j n_j = k$.) Then the number of partitions of $\{1,\ldots, k\}$ into blocks of sizes $a_1,\ldots,a_n$ is
\[\frac{k!}{\prod_{i=1}^n a_i!}\frac{1}{\prod_{j=1}^k n_j!}\]
Combining this observation with Lemmas \ref{Qbots} and \ref{Qdescription} gives us the following corollary.

\begin{cor}\label{Qdesccor}
\begin{multline*}
\Q^{k,T} \left( \psi_1 \in [t, t+dt), \text{ spines split into groups of sizes } a_1,\ldots,a_n,\, L_{\xi^1_t} = l \right)\\
=  \frac{l^{(n)}p_l}{\P[L^{(n)}]}\, \frac{k!}{\prod_{i=1}^n a_i! \prod_{j=1}^{k-1}n_j!} \P[L^{(n)}]re^{(m-1)rt}\frac{\prod_{i=1}^n \P^k [ N_{T-t}^{(a_i)} ]}{ \P^k[N_{T}^{(k)} ] } dt.
\end{multline*}
\end{cor}

\subsection{Understanding the measure $\Q^{k,T}$ as $T\to\infty$}\label{QTlargeTsec}
To help the reader to understand our results from the previous section, particularly Corollary \ref{Qdesccor}, we let $T\to\infty$ and ask what happens to the tree drawn out by the spines. For brevity we will concentrate on the critical case $m=1$, although similar calculations could be done in near-critical cases when $m=1+\mu/T + o(1/T)$. Take $m=1$, $n=2$ and $t=sT$ in Corollary \ref{Qdesccor}; if $a_1\neq a_2$ then we get
\begin{multline*}
\Q^{k,T} \left( \psi_1 \in [sT, sT+Tds), \text{ spines split into two groups of sizes } a_1,a_2,\, L_{\xi^1_{sT}} = l \right)\\
=  \frac{l(l-1)p_l}{\P[L(L-1)]}\, \frac{k!}{a_1!a_2!} \P[L(L-1)]r\frac{\P^k [ N_{T(1-s)}^{(a_1)} ]\P^k [ N_{T(1-s)}^{(a_2)} ]}{ \P^k[N_{T}^{(k)} ] } T ds.
\end{multline*}
We now let $T\to\infty$ and use Kolmogorov's theorem that $T\P(N_{uT}>0)\to 2/(\sigma^2 r u)$, as well as Yaglom's theorem which says that conditionally on survival, $N_{uT}/T$ converges in distribution to an exponential random variable of parameter $2/(\sigma^2 r u)$. Letting $\mathcal E_1$ be exponentially distributed with parameter $2/(\sigma^2 r (1-s))$ and $\mathcal E_2$ be exponentially distributed with parameter $2/(\sigma^2 r)$, this gives
\begin{align*}
&\lim_{T\to\infty}\Q^{k,T} \left( \psi_1 \in [sT, sT+Tds), \text{ spines split into two groups of sizes } a_1,a_2,\, L_{\xi^1_{sT}} = l \right)\\
&= \frac{l(l-1)p_l}{\P[L(L-1)]}\, \frac{k!}{a_1!a_2!} \P[L(L-1)]r\frac{\frac{2}{T\sigma^2 r(1-s)}T^{a_1}\P^k [ \mathcal E_1^{a_1} ]\frac{2}{T\sigma^2 r(1-s)}T^{a_2}\P^k [ \mathcal E_1^{a_2} ]}{ \frac{2}{T\sigma^2 r}T^{k}\P^k [ \mathcal E_1^{k} ] } ds\\
&= l(l-1)p_l\, r \frac{(\sigma^2 r(1-s)/2)^{a_1-1}(\sigma^2 r(1-s)/2)^{a_2-1}}{ (\sigma^2 r/2)^{k-1} } ds\\
&= l(l-1)p_l \frac{2}{\sigma^2}(1-s)^{k-2}ds.
\end{align*}
If $a_1=a_2$ then there is an extra factor of $1/2$ as the two blocks can be rearranged indistinguishably.

As there are $k-1$ possible (ordered) ways of splitting $k$ into two groups of non-zero size, and from the above each of these ways is equally likely,
\begin{multline*}
\lim_{T\to\infty}\Q^{k,T} \left( \psi_1 \in [sT, sT+Tds), \text{ spines split into two groups},\, L_{\xi^1_{sT}} = l \right)\\
= \frac{l(l-1)}{\sigma^2}p_l (k-1)(1-s)^{k-2}ds.
\end{multline*}
We note that if we sum the above quantity over $l$ and integrate over $s\in[0,1]$ we obtain $1$. This means that, in the limit as $T\to\infty$, at the first spine splitting event $\psi_1$, the $k$ spines always split into exactly two groups. We also see that the number of spines in each of the groups is uniform on $\{1,\ldots,k-1\}$, and the total number of offspring at this time is doubly-size-biased. Finally, the first splitting time, when rescaled by $1/T$, converges in distribution to the minimum of $k$ independent uniform random variables on $[0,1]$.

The symmetry lemma, Lemma \ref{symm}, tells us that we can extend our understanding of the \emph{first} spine splitting event to all spine splitting events. When a collection of spines decides to split, they always (in the limit as $T\to\infty$) split uniformly into two groups; this property is shared by the tree drawn out by the Kingman coalescent. Furthermore the $k-1$ spine split times, when rescaled by $1/T$, are independent and uniformly distributed on $[0,1]$.

We stress again that this is true only in the critical case; if instead we are in the near-critical case when $m=1+\mu/T+o(1/T)$ (see Section \ref{nearcritresultssec}) then the uniform density for the independent split times is replaced by $\frac{r\mu e^{r\mu s}}{e^{r\mu}-1}ds$. In particular, the near-critical case is simply a deterministic time-change of the critical picture.

\subsection{Proofs of properties of $\Q^{k,T}$}\label{Qproofsec}

We start this section with the proof of the symmetry lemma.

\begin{proof}[Proof of Lemma \ref{symm}]
Fix $t,T$ and $v$. Let $\H$ be the $\sigma$-algebra generated by all the information except in the subtree generated by $v$ after time $t$. Then it suffices to show that for $s\in(t,T]$ and $i\ge 0$,
\[\Q^{k,T}(\tau_v > s,\, L_v = i | \H) = \Q^{j,T-t}(\tau_\emptyset > s-t, \, L_\emptyset = i)\]
almost surely.

Recall that
\[g_{k,T} = \ind_{\{\xi^i_T \neq \xi^j_T \,\forall i\neq j\}} \prod_{i=1}^k \prod_{v<\xi^i_T} L_v\]
and
\[\zeta_{k,T} = \frac{g_{k,T}}{\P[N^{(k)}_T]}.\]
Let $I$ be the set of marks carried by $v$ at time $t$, and let
\[\tilde g = \ind_{\{\xi^i_T \neq \xi^j_T \, \forall i\neq j,\, i,j\in I^c\}} \prod_{i\in I}\prod_{\xi^i_t\le v<\xi^i_T} L_v\]
and
\[h = \ind_{\{\xi^i_T \neq \xi^j_T \, \forall i\neq j,\, i,j\in I^c\}} \Big(\prod_{i\not\in I}\prod_{v<\xi^i_T} L_v\Big) \prod_{i\in I}\prod_{v< \xi^i_t} L_v.\]
Note that $h$ is $\H$-measurable and $g_{k,T}=\tilde g h$.

By Lemma \ref{condform}, $\Q^{k,T}$-almost surely,
\[\Q^{k,T}(\tau_v > s, \, L_v = i | \H) = \frac{1}{\P^k[\zeta_{k,T}|\H]} \P^k[\zeta_{k,T}\ind_{\{\tau_v > s, \, L_v = i\}} | \H].\]
Cancelling factors of $\P^k[N^{(k)}_T]$ and using the fact that $g_{k,T}=\tilde g h$ where $h$ is $\H$-measurable, we get
\[\Q^{k,T}(\tau_v > s, \, L_v = i | \H) = \frac{1}{h\P^k[\tilde g|\H]} h\P^k[\tilde g \ind_{\{\tau_v > s, \, L_v = i\}} | \H] = \frac{\P^k[\tilde g \ind_{\{\tau_v > s, \, L_v = i\}} | \H]}{\P^k[\tilde g|\H]}.\]
By the Markov branching property under $\P^k$, the behaviour of the subtree generated by $v$ after time $t$ is independent of the rest of the system and---on the event that $v$ is carrying $j$ marks at time $t$---behaves as if under $\P^j$. Thus
\[\Q^{k,T}(\tau_v > s, \, L_v = i | \H) = \frac{\P^j[g_{j,T-t} \ind_{\{\tau_\emptyset > s - t,\, L_\emptyset = i\}}]}{\P^j[g_{j,T-t}]}.\]
almost surely. Applying Lemma \ref{gproj} establishes the result.
\end{proof}

We now move on to the proof of Lemma \ref{Qbots}, which gives the distribution of the split times under $\Q^{k,T}$.

\begin{proof}[Proof of Lemma \ref{Qbots}]
For the first statement,
\[\Q(\tau_\phi > t) = \P[\zeta_{k,T}\ind_{\{\tau_\emptyset > t\}}] = \frac{1}{\P[N^{(k)}_T]}\P[g_{k,T}\ind_{\{\tau_\emptyset > t\}}].\]
By the Markov property and Lemma \ref{gproj},
\[\P[g_{k,T}\ind_{\{\tau_\emptyset > t\}}] = \P(\tau_\emptyset>t)\P[g_{k,T-t}] = e^{-rt}\P[N^{(k)}_{T-t}]\]
as required.

For the second statement,
\[\Q(\psi_1 > t) = \P[\zeta_{k,T}\ind_{\{\psi_1 > t\}}] = \frac{1}{\P[N^{(k)}_T]}\P[g_{k,T}\ind_{\{\psi_1 > t\}}],\]
and by the Markov property and Lemma \ref{gproj},
\[\P[g_{k,T}\ind_{\{\psi_1 > t\}}] = \P\Big[\Big(\prod_{v<\xi^1_t} L_v^k\Big)\ind_{\{\psi_1>t\}}\Big]\P[g_{k,T-t}] = \P\Big[\Big(\prod_{v<\xi^1_t} L_v^k\Big)\ind_{\{\psi_1>t\}}\Big]\P[N^{(k)}_{T-t}].\]
Putting these two lines together we get
\begin{equation}\label{Qpsiinterm}
\Q(\psi_1 > t) = \frac{\P^k[N^{(k)}_{T-t}]}{\P^k[N^{(k)}_T]}\P\Big[\Big(\prod_{v<\xi^1_t} L_v^k\Big)\ind_{\{\psi_1>t\}}\Big].
\end{equation}
Note that $\psi>t$ if and only if all $k$ marks are following the same particle at time $t$ (which must also be alive); thus
\[\P\Big[\Big(\prod_{v<\xi^1_t} L_v^k\Big)\ind_{\{\psi_1>t\}}\Big] = \P\Big[\sum_{u\in\Nc_t} \Big(\prod_{v<u} L_v^k\Big) \ind_{\{\xi^1_t = \ldots = \xi^k_t = u\}}\Big] = \P\Big[\sum_{u\in\Nc_t} 1\Big] = \P[N_t] = e^{(m-1)t}.\]
Substituting back into \eqref{Qpsiinterm} gives the desired result.

The third statement follows easily from the first two.
\end{proof}

We next prove Lemma \ref{sizebiased}, which says that births off the spine are size-biased.

\begin{proof}[Proof of Lemma \ref{sizebiased}]
From the definition of $\Q$,
\begin{align*}
\Q(L_\emptyset = j | \tau_\emptyset = t, \, \psi_1>t) &= \frac{\P[\zeta_{k,T}\ind_{\{L_\emptyset=j\}}|\tau_\emptyset=t,\,\psi_1>t]}{\P[\zeta_{k,T}|\tau_\emptyset=t,\,\psi_1>t]}\\
&= \frac{\P[g_{k,T}\ind_{\{L_\emptyset=j\}}|\tau_\emptyset=t,\,\psi_1>t]}{\P[g_{k,T}|\tau_\emptyset=t,\,\psi_1>t]}\\
&= \frac{\P[g_{k,T}\ind_{\{L_\emptyset=j,\, \psi_1>t\}}|\tau_\emptyset=t]}{\P[g_{k,T}\ind_{\{\psi_1>t\}}|\tau_\emptyset=t]}.
\end{align*}
If the first particle has $i$ offspring, then the product appearing in the definition of $g_{k,T}$ sees a factor of $i^k$; and the probability that all $k$ spines follow the same one of these offspring is $1/i^{k-1}$. Thus, by the Markov property, for any $i$,
\[\P[g_{k,T}\ind_{\{L_\emptyset=i,\, \psi_1>t\}}|\tau_\emptyset=t] = p_i i^k \frac{1}{i^{k-1}}\P[g_{k,T-t}] = i p_i \P[g_{k,T-t}].\]
Thus
\[\Q(L_\emptyset = j | \tau_\emptyset = t, \, \psi_1>t) = \frac{j p_j \P[g_{k,T-t}]}{\sum_i i p_i \P[g_{k,T-t}]} = \frac{j p_j}{m}.\qedhere\]
\end{proof}

The final proof in this section is of Lemma \ref{Qdescription}, which completed the description of $\Q^{k,T}$.

\bp[Proof of Lemma \ref{Qdescription}]
By the symmetry lemma, for any $h\in(0,T-t]$,
\[ \Q^{k,T} \left( \psi_1 < t+h, \, \mathcal P^\xi_{\psi_1}=P,\, L_{\xi^1_t} = l  ~\Big|~ \psi_1 > t  \right) = \Q^{k,T-t} \left( \psi_1 < h,\, \mathcal P^\xi_{\psi_1}=P,\, L_{\emptyset}=l \right).  \]
By the definition of $\Q^{k,T-t}$, this is equal to
\begin{multline}\label{condongamma}
\frac{1}{ \P^k[N_{T-t}^{(k)} ] } \P^k  \left[   g_{k,T-t}\,  \psi_1 < h,\, \mathcal P^\xi_{\psi_1}=P,\, L_{\emptyset}=l  \right]\\
= \frac{1}{ \P^k[N_{T-t}^{(k)} ] } \P^k  \big(  \psi_1 < h,\, \mathcal P^\xi_{\psi_1}=P,\, L_{\emptyset}=l \big)  \P^k \big[ g_{k,T-t} \big| \psi_1 < h,\, \mathcal P^\xi_{\psi_1}=P,\, L_{\emptyset}=l\big]
\end{multline}
First we consider
\begin{align*}
 \P^k  \left(    \psi_1 < h,\, \mathcal P^\xi_{\psi_1}=P,\, L_{\emptyset}=l  \right)  &= \P^k ( \psi_1 < h , L_\emptyset = l )~ \P^k ( \mathcal P^\xi_{\psi_1}=P | \psi_1 < h , L_\emptyset = l )\\
\end{align*}
since $l^{(n)}/l^k$ is the probability that $k$ balls put uniformly and independently into $l$ bins give rise to the partition $P$.

Next we consider
\[\P^k \big[ g_{k,T-t} \big| \psi_1 < h,\, \mathcal P^\xi_{\psi_1}=P,\, L_{\emptyset}=l\big].\]
Note that on the event $\{\mathcal P^\xi_{\psi_1}=P,\, L_{\emptyset}=l\}$, we have
\begin{align*}
g_{k,T-t} &= \ind_{\{\xi^i_{T-t} \neq \xi^j_{T-t} \,\forall i\neq j\}} \prod_{i=1}^k \prod_{v<\xi^i_{T-t}} L_v\\
&= l^k \prod_{p\in P} \ind_{\{\xi^i_{T-t} \neq \xi^j_{T-t} \,\forall i\neq j \in p\}} \prod_{i\in p} \prod_{\xi^i_{\psi_1}\le v<\xi^i_{T-t}} L_v.
\end{align*}
Lemma \ref{gproj} tells us that for each $p\in P$, on the event $\{\mathcal P^\xi_{\psi_1}=P,\, L_{\emptyset}=l\}$,
\[\P^k\Big[\ind_{\{\xi^i_{T-t} \neq \xi^j_{T-t} \,\forall i\neq j \in p\}} \prod_{i\in p} \prod_{\xi^i_{\psi_1}\le v<\xi^i_{T-t}} L_v \,\Big|\,\F^k_{\psi_1}  \Big] = \P^k\big[N^{(|p|)}_{T-t-u}\big]\big|_{u=\psi_1}\]
On the event $\psi_1<h$, we have
\[\P^k\big[N^{(|p|)}_{T-t-u}\big]\big|_{u=\psi_1} = \P^k\big[N^{(|p|)}_{T-t}\big] + o(h)\]
and therefore
\begin{align*}
\P^k \big[ g_{k,T-t} \big| \psi_1 < h,\, \mathcal P^\xi_{\psi_1}=P,\, L_{\emptyset}=l\big]  = l^k \prod_{i=1}^n \P^k [ N_{T-t}^{(a_i)} ] + o(h).
\end{align*}
Putting these calculations back into \eqref{condongamma}, we have shown that
\begin{align*}
\Q^{k,T} \Big( \psi_1 < t+h, \, \mathcal P^\xi_{\psi_1}=P,\, L_{\xi^1_t} = l & ~\Big|~ \psi_1 > t  \Big)\\
&= \frac{1}{ \P^k[N_{T-t}^{(k)} ] } (rh + o(h))p_l \frac{l^{(n)}}{l^k} l^k \Big(\prod_{i=1}^n \P^k [ N_{T-t}^{(a_i)} ] + o(h)\Big)\\
&=  p_l l^{(n)} \frac{\prod_{i=1}^n \P^k [ N_{T-t}^{(a_i)} ]}{ \P^k[N_{T-t}^{(k)} ] } (rh + o(h))
\end{align*}
which completes the proof.
\ep

\subsection{Proof of Proposition \ref{firstprop}}\label{Qfirstpropsec}

Before we prove Proposition \ref{firstprop}, we develop several partial results along the way. The following simple general lemma will be useful.

\begin{lem}\label{condform}
Suppose that $\mu$ and $\nu$ are probability measures on the $\sigma$-algebra $\F$, and that $\G$ is a sub-$\sigma$-algebra of $\F$. If
\[\left.\frac{\d\mu}{\d\nu}\right|_\F = Y \hspace{4mm} \hbox{ and } \hspace{4mm} \left.\frac{\d\mu}{\d\nu}\right|_\G = Z,\]
then for any non-negative random variable $X$,
\[Z\mu[X|\G] = \nu[XY|\G] \hspace{4mm} \nu\hbox{-almost surely}.\]
\end{lem}

\begin{proof}
For any $A\in\G$,
\[\nu[XY\ind_A] = \mu[X\ind_A] = \mu[\mu[X|\G]\ind_A] = \nu[Z\mu[X|\G] \ind_A].\]
Since $Z\mu[X|\G]$ is $\G$-measurable, it therefore satisfies the definition of conditional expectation of $XY$ with respect to $\G$ under $\nu$.
\end{proof}

Applying this to our situation, we get that for any non-negative $\F^k_T$-measurable random variable $X$, on the event $Z_{k,T}>0$,
\begin{equation}\label{QgivenF}
\Q^{k,T}[X|\F^0_T] = \frac{1}{Z_{k,T}}\P^k[X\zeta_{k,T}|\F^0_T],
\end{equation}
and on the event $\zeta_{k,T}>0$, since $\zeta_{k,T}$ is $\Gt^k_T$-measurable,
\begin{equation}\label{QgivenG}
\Q^{k,T}[X|\Gt^k_T] = \frac{1}{\zeta_{k,T}}\P^k[X\zeta_{k,T}|\Gt^k_T] = \P^k[X|\Gt^k_T].
\end{equation}
This last equation \eqref{QgivenG} tells us in particular that any event that is independent of $\Gt^k_T$ has the same probability under $\Q$ as under $\P$. In other words, non-spine particles behave under $\Q$ exactly as they do under $\P$: they branch at rate $r$ and have offspring distribution $L$.

Also note that under $\Q^{k,T}$, the $k$ spine particles are almost surely distinct at time $T$, since directly from the definition of $\zeta_{k,T}$,
\[\Q^{k,T}(\exists i\neq j : \xi^i_T=\xi^j_T) = \P[\zeta_{k,T}\ind_{\{\exists i\neq j : \xi^i_T=\xi^j_T\}}] = 0.\]
In fact, the next lemma tells us that under $\Q^{k,T}$, the spines are chosen uniformly without replacement from those alive at time $T$.

\begin{lem}\label{Quniform}
For any $u\in \Nc_T^{(k)}$, on the event $N_T\ge k$,
\[\Q^{k,T}(\xi_T = u | \F^0_T) = \frac{1}{N_T^{(k)}}.\]
\end{lem}

\begin{proof}
Note that if $N_T\ge k$ then $Z_{k,T}>0$. Then by \eqref{QgivenF}, for any $u\in \Nc_T^{(k)}$,
\[\Q(\xi_T = u | \F^0_T)  = \frac{1}{Z_{k,T}} \P[\zeta_{k,T}\ind_{\{\xi_T = u\}} | \F^0_T ] = \frac{\P[N_T^{(k)}]}{N_T^{(k)}} \frac{1}{\P[N_T^{(k)}]}\Big(\prod_{i=1}^k \prod_{v<u_i} L_v\Big) \P(\xi_t = u | \F^0_T).\]
The result now follows by applying \eqref{Pspineprob}.
\end{proof}

As part of proving Proposition \ref{firstprop} we will need to calculate quantities like $\Q^{k,T}[1/N_T^{(k)}|\G^k_T]$. The next lemma allows us to work with moment generating functions, which are somewhat easier to deal with and will lead to an important product structure from the  independent contributions to $N_T$ along different branches of the $k$ spines' genealogical tree under $\Q^{k,T}$.

\begin{lem}\label{recip}
For any $k\in\N$ and positive integer valued random variable $N$ under an expectation operator $E$, we have
\[
E \left[\frac{1}{N(N-1)\dots(N-k+1)}\right]
= \frac{1}{(k-1)!} \int_0^\infty (e^z-1)^{k-1} E\left[e^{-zN}\right] \d z.
\]

In particular, for any $k\in\N$ and $T\ge 0$,
\[\Q^{k,T}\Big[\frac{1}{N_T^{(k)}}\Big|\G^k_T\Big] = \frac{1}{(k-1)!} \int_0^\infty (e^z-1)^{k-1}\Q^{k,T}[e^{-zN_T}|\G^k_T]\, \d z.\]
\end{lem}

\begin{proof}
We show, by induction on $j$, that for all $j=1,\ldots,k$,
\[E\left[\frac{1}{N^{(j)}}\right] = \frac{1}{(j-1)!} \int_0^\infty (e^z-1)^{j-1}E[e^{-zN}]\, \d z.\]
For $j=1$, by Fubini's theorem,
\[E\left[\frac{1}{N}\right] = E\left[\int_0^\infty e^{-zN} \d z \right] = \int_0^\infty E\left[e^{-zN}\right]\, \d z.\]
For the general step, observe that for $j\le k-1$,
\begin{multline*}
\int_0^\infty (e^z-1)^j E[e^{-zN}]\,\d z 
= \int_0^\infty (e^z-1)^{j-1}E[e^{-z(N-1)}]\d z - \int_0^\infty (e^z-1)^{j-1}E[e^{-zN}]\,\d z
\end{multline*}
and by the induction hypothesis, this equals
\begin{align*}
(j-1)!\,E\left[\frac{1}{(N-1)^{(j)}}\right] - (j-1)!\,E\left[\frac{1}{N^{(j)}}\right]
&= (j-1)! \,E\left[ \frac{N}{N^{(j+1)}} - \frac{N-j}{N^{(j+1)}} \right]\\
&= j!\,E\left[\frac{1}{N^{(j+1)}}\right].
\end{align*}
This gives the result.
\end{proof}

We can now prove Proposition \ref{firstprop}.

\begin{proof}[Proof of Proposition \ref{firstprop}]
First note that
\[\Q[f(\xi_T)|\F^0_T]\ind_{\{N_T\ge k\}} = \Q\bigg[ \sum_{u\in \Nc_T^{(k)}} \ind_{\{\xi_T = u\}} f(u)\, \bigg|\,\F^0_T\bigg] = \sum_{u\in \Nc_T^{(k)}} f(u) \Q(\xi_T = u | \F^0_T)\]
almost surely. Applying Lemma \ref{Quniform}, we get
\[\Q[f(\xi_T)|\F^0_T]\ind_{\{N_T\ge k\}} = \frac{1}{N_T^{(k)}}\sum_{u\in \Nc_T^{(k)}} f(u)\]
almost surely (where we take the right-hand side to be zero if $N_T < k$). Taking $\P$-expectations,
\[\P\Big[\frac{1}{N_T^{(k)}}\sum_{u\in \Nc_T^{(k)}} f(u)\Big] = \P\big[\Q[f(\xi_T)|\F^0_T]\ind_{\{N_T\ge k\}}\big].\]
Applying \eqref{RadNik0} and recalling that under $\Q$ there are at least $k$ particles alive at time $T$ almost surely,
\begin{equation}\label{subfirstprop}
\P\Big[\frac{1}{N_T^{(k)}}\sum_{u\in \Nc_T^{(k)}} f(u)\Big] = \Q\Big[\frac{1}{Z_{k,T}}\Q[f(\xi_T)|\F^0_T]\Big] = \Q\Big[\frac{1}{Z_{k,T}}f(\xi_T)\Big] = \P[N_T^{(k)}]\Q\Big[\frac{1}{N_T^{(k)}}f(\xi_T)\Big].
\end{equation}
Dividing through by $\P(N_T\ge k)$ and using the Tower property of conditional expectation to apply Lemma \ref{recip}  gives the result.
\end{proof}

\subsection{Campbell's formula}

One of the key elements that we need to carry out our calculations will be a version of Campbell's formula. Let $\tilde N_t$ be the number of \emph{ordinary} particles alive at time $t$---that is, they are not spines, and did not split from spines at spine splitting events. Recall that we also defined $n_t$ to be the number of distinct spines alive at time $t$.

We write $F(\theta,t) = \P[\theta^{N_t}]$ and $u(\theta) = \P[\theta^L]-\theta$. These functions satisfy the Kolmogorov forwards and backwards equations
\begin{equation}\label{kolfor}
\frac{\partial}{\partial t} F(\theta,t) = ru(\theta)\frac{\partial}{\partial\theta} F(\theta,t)
\end{equation}
and
\begin{equation}\label{kolback}
\frac{\partial}{\partial t} F(\theta,t) = ru(F(\theta,t));
\end{equation}
see \cite[Chapter III, Section 3]{athreya_ney:branching_processes}. Our main aim is to show the following.

\begin{prop}\label{Qmgf_gen_prop}
For any $z\ge 0$,
\[\Q^{k,T}[e^{-z\tilde N_T}|\G^k_T] = \prod_{i=0}^{k-1}\Big(e^{-r(m-1)(T-\psi_i)}\frac{u(F(e^{-z},T-\psi_i))}{u(e^{-z})}\Big)\]
$\Q^{k,T}$-almost surely.
\end{prop}

Notice in particular that the right-hand side depends only on the values of the split times $\psi_1,\dots,\psi_{k-1}$ of the spines, not any of the other information in $\G^k_T$ (for example the topological information about the tree). This---used in conjunction with Proposition \ref{firstprop}---is a large part of the reason that the split times of our $k$ uniformly chosen particles are (asymptotically) independent of the topological information in the induced tree.

The main step in proving Proposition \ref{Qmgf_gen_prop} comes from the next lemma.

\begin{lem}\label{Qmgf_gen}
For any $z\ge 0$,
\[\Q[e^{-z\tilde N_T}|\G^k_T] = \prod_{i=0}^{k-1} \exp\Big( -r(m-1)(T-\psi_i) + r \int_0^{T-\psi_i} u'(\P[e^{-zN_{s}}])\, \d s\Big).\]
$\Q^{k,T}$-almost surely.
\end{lem}

\begin{proof}
Let $\Lambda_T$ be the total number of birth events off the spines (i.e.~births along spines that are not spine splitting events) before time $T$. Recall (from Lemma \ref{Qbots} and the symmetry lemma) that under $\Q^{k,T}$ each spine gives birth to non-spine particles according to a Poisson process of rate $rm$, independently of everything else. Thus at any time $s\in[0,T]$, the total rate at which spine particles give birth to non-spine particles is $r m n_s$. Besides, such births are size biased (by Lemma \ref{sizebiased} and the symmetry lemma). Finally, once a particle is born off the spines, it generates a tree that behaves exactly as under $\P$ (see \eqref{QgivenG} and the discussion that follows).

Thus, letting $\lambda_T = \int_0^T n_s \d s$,
\[\Q[e^{-z\tilde N_T}|\G^k_T] = \sum_{j=0}^\infty \Q(\Lambda_T=j|\G^k_T)\Big(\int_0^T \sum_{i=1}^{\infty} \frac{i p_i}{m}\P[e^{-zN_{T-s}}]^{i-1} \frac{n_s}{\lambda_T} \d s\Big)^j.\]
Since $\Q(\Lambda_T=j|\G^k_T) = e^{-r m\lambda_T} (rm\lambda_T)^j/j!$, we get
\[\Q[e^{-z\tilde N_T}|\G^k_T] = e^{- r m\lambda_T} \sum_{j=0}^\infty \frac{1}{j!}\Big(r \int_0^T \sum_{i=1}^\infty i p_i\P[e^{-zN_{T-s}}]^{i-1} n_s \d s\Big)^j.\]
Note that
\[\sum_{i=1}^\infty i p_i \theta^{i-1} = \frac{\d}{\d\theta} \sum_{i=1}^\infty p_i\theta^i = u'(\theta)+1.\]
Therefore
\[\Q[e^{-z\tilde N_T}|\G^k_T] = \exp\Big(-r (m-1) \lambda_T + r\int_0^T u'(\P[e^{-zN_{T-s}}]) n_s \d s\Big).\]
Now, we know that between times $\psi_{i-1}$ and $\psi_i$ we have exactly $i$ distinct spine particles. Thus
\[\Q[e^{-z\tilde N_T}|\G^k_T] = \prod_{i=0}^{k-1} \exp\Big( -r(m-1)(T-\psi_i) + r \int_{\psi_i}^T u'(\P[e^{-zN_{T-s}}])\, \d s\Big)\]
and the result follows.
\end{proof}

\begin{proof}[Proof of Proposition \ref{Qmgf_gen_prop}]
Recalling \eqref{kolback} that $F(\theta,s)$ satisfies the backwards equation
\[\frac{\partial}{\partial s} F(\theta,s) = r u(F(\theta,s)),\]
by making the substitution $t=F(\theta,s)$ we see that
\[r\int_a^b u'(F(\theta,s))\d s = r\int_{F(\theta,a)}^{F(\theta,b)} \frac{u'(t)}{ru(t)} \d t = \log\Big(\frac{u(F(\theta,b))}{u(F(\theta,a))}\Big).\]
Applying this to Lemma \ref{Qmgf_gen}, we have
\[\Q[e^{-z\tilde N_T}|\G^k_T] = \prod_{i=0}^{k-1}\Big(e^{-r(m-1)(T-\psi_i)}\frac{u(F(e^{-z},T-\psi_i))}{u(F(e^{-z},0))}\Big).\]
Noting that $F(e^{-z},0) = e^{-z}$ gives the result.
\end{proof}

\section{Birth-death processes}\label{BDsec}

In this section we aim to prove the results from Section \ref{birthdeathresults}. Recall the setup: fix $a\ge 0$ and $b>0$, and suppose that $r = \alpha+\beta$, $p_0 = \alpha/(\alpha+\beta)$ and $p_2 = \beta/(\alpha+\beta)$, with $p_j=0$ for $j\neq 0,2$. This is known as a birth-death process with birth rate $\beta$ and death rate $\alpha$. Since all particles have either $0$ or $2$ children, and under $\Q$ the spines cannot have $0$ children, they must always have $2$ children. This simplifies the picture considerably.

\subsection{Elementary calculations with generating functions}\label{BDgenfuncts}

Suppose first that we are in the non-critical case $\alpha\neq\beta$. It is easy to calculate the moment generating function under $\P$ for a birth-death process (see \cite[Chapter III, Section 5]{athreya_ney:branching_processes}): for $\alpha\neq\beta$ and $\theta\in(0,1)$,
\[F(\theta,t) := \P[\theta^{N_t}] = \frac{\alpha(1-\theta)e^{(\beta-\alpha)t} + \beta \theta-\alpha}{\beta(1-\theta)e^{(\beta-\alpha)t} + \beta \theta-\alpha}.\]
We then see that
\[\P(N_t = 0) = \lim_{\theta\downarrow 0} F(0,t) = \frac{\alpha e^{(\beta-\alpha)t}-\alpha}{\beta e^{(\beta-\alpha)t}-\alpha}.\]
Writing
\[p_t = \P(N_t=0) = \frac{\alpha e^{(\beta-\alpha)t}-\alpha}{\beta e^{(\beta-\alpha)t}-\alpha}, \hspace{5mm} 1-p_t = \frac{(\beta-\alpha)e^{(\beta-\alpha)t}}{\beta e^{(\beta-\alpha)t}-\alpha}\]
and
\[q_t = \frac{\beta e^{(\beta-\alpha)t}-\beta}{\beta e^{(\beta-\alpha)t}-\alpha}, \hspace{5mm} 1-q_t = \frac{\beta-\alpha}{\beta e^{(\beta-\alpha)t}-\alpha},\]
we get
\[F(\theta,t) = p_t + (1-p_t)\frac{(1-q_t)\theta}{1-q_t\theta} = p_t + \frac{(1-p_t)(1-q_t)}{q_t}\Big(\frac{1}{1-q_t\theta}-1\Big).\]
From this we see that
\[F(\theta,t) = p_t + (1-p_t)(1-q_t)\sum_{j=1}^\infty \theta^j q_t^{j-1}\]
and
\[\frac{\partial^k F(\theta,t)}{\partial \theta^k} = \frac{(1-p_t)(1-q_t)}{q_t}\frac{q_t^k k!}{(1-q_t\theta)^{k+1}}.\]
Therefore
\[\P(N_t = j) = (1-p_t)(1-q_t)q_t^{j-1} \hspace{3mm} \hbox{ for } j\ge 1,\]
so
\[\P(N_t \ge k) = (1-p_t)(1-q_t)\sum_{j=k}^\infty q_t^{j-1} = (1-p_t)q_t^{k-1} = \frac{(\beta-\alpha)e^{(\beta-\alpha)t}\beta^{k-1}(e^{(\beta-\alpha)t}-1)^{k-1}}{(\beta e^{(\beta-\alpha)t} - \alpha)^k}.\]
Also, since $\P[N_t^{(k)}] = \lim_{\theta\uparrow 1}\frac{\partial^k F(\theta,t)}{\partial \theta^k}$,
\begin{equation}\label{Pmom}
\P[N_t^{(k)}] = \frac{(1-p_t)(1-q_t)}{q_t} \frac{q_t^k k!}{(1-q_t)^{k+1}} = k! \Big(\frac{\beta}{\beta-\alpha}\Big)^{k-1} e^{(\beta-\alpha)t}(e^{(\beta-\alpha)t}-1)^{k-1}.
\end{equation}
Thus
\begin{equation}\label{Pcond}
\frac{\P[N_t^{(k)}]}{\P(N_t\ge k)} = \frac{k!}{(\beta-\alpha)^k} (\beta e^{(\beta-\alpha)t}-\alpha)^k
\end{equation}
and
\begin{equation}\label{Pratio}
\frac{\P\big[N_{T-t}^{(k)}\big]}{\P\big[N_T^{(k)}\big]} = e^{-(\beta-\alpha)t}\Big(\frac{e^{(\beta-\alpha)(T-t)}-1}{e^{(\beta-\alpha)T}-1}\Big)^{k-1}.
\end{equation}
Finally, writing
\[F(\theta,t) = \frac{\alpha}{\beta} + \frac{(\beta-\alpha)\theta - \alpha(\beta-\alpha)/\beta}{\beta(1-\theta)e^{(\beta-\alpha)t} + \beta \theta-\alpha},\]
we see that
\begin{equation}\label{Fpartial}
\frac{\partial F(\theta,t)}{\partial t} = \frac{(\beta-\alpha)^2 (\beta \theta - \alpha)(1-\theta)e^{(\beta-\alpha)t}}{(\beta(1-\theta)e^{(\beta-\alpha)t} + \beta \theta - \alpha)^2}.
\end{equation}

In the critical case $\alpha=\beta$, similar calculations give
\begin{equation}\label{critBDpgf}
F(\theta,t) := \P[\theta^{N_t}] = \frac{(1-\theta)\beta t + \theta}{(1-\theta)\beta t + 1},
\end{equation}
\begin{equation}\label{critBDratio}
\P[N_t^{(k)}] = \lim_{\theta\uparrow 1} \frac{\partial^k F(\theta,t)}{\partial \theta^k} = k! (\beta t)^{k-1},
\end{equation}
\begin{equation}\label{critBDcond}
\frac{\P[N_t^{(k)}]}{\P(N_t\ge k)} = k! (\beta t+1)^k
\end{equation}
and
\begin{equation}\label{critBDpartialt}
\frac{\partial F(\theta,t)}{\partial t} = \frac{\partial}{\partial t} \Big(1+\frac{\theta-1}{(1-\theta)\beta t+1}\Big) = \frac{(1-\theta)^2\beta}{((1-\theta)\beta t+1)^2}.
\end{equation}

\subsection{Split time densities}

Recall that $\mathcal H'$ is the $\sigma$-algebra that contains information about which marks follow which spines, but does not know anything about the spine split times.

\begin{lem}\label{Qsplitdist}
Under $\Q^{k,T}$, the spine split times $\psi_1,\ldots,\psi_{k-1}$ are independent of $\mathcal H'$ and have a joint probability density function
\[f^\Q_k(s_1,\ldots,s_{k-1}) = \begin{cases} (k-1)! \Big(\frac{\beta-\alpha}{e^{(\beta-\alpha)T}-1}\Big)^{k-1} \prod_{i=1}^{k-1} e^{(\beta-\alpha)(T-s_i)} & \hbox{ if } \alpha\neq \beta\\
											 (k-1)!/T^{k-1} & \hbox{ if } \alpha = \beta\end{cases}.\]
\end{lem}

\begin{proof}
We do the calculation in the non-critical case $\alpha\neq \beta$. The proof in the critical case is identical.

Recall from Lemma \ref{Qbots} that
\[\Q^{k,T}(\psi_1 > s_1) = \frac{\P[N_{T-s_1}^{(k)}]}{\P[N_{T}^{(k)}]}e^{(m-1)rs_1} = \frac{\P[N_{T-s_1}^{(k)}]}{\P[N_{T}^{(k)}]}e^{(\beta-\alpha)s_1}.\]
Then \eqref{Pratio} gives
\[\Q^{k,T}(\psi_1 > s_1) = e^{-(\beta-\alpha)s_1}\Big(\frac{e^{(\beta-\alpha)(T-s_1)}-1}{e^{(\beta-\alpha)T}-1}\Big)^{k-1}e^{(\beta-\alpha)s_1} = \Big(\frac{e^{(\beta-\alpha)(T-s_1)}-1}{e^{(\beta-\alpha)T}-1}\Big)^{k-1},\]
so $\psi_1$ has density
\[(k-1)(\beta-\alpha)e^{(\beta-\alpha)(T-s_1)}\frac{(e^{(\beta-\alpha)(T-s_1)}-1)^{k-2}}{(e^{(\beta-\alpha)T}-1)^{k-1}}.\]
For $i=2,\ldots, k-1$, between times $\psi_{i-1}$ and $\psi_i$ we have exactly $i$ particles carrying marks. Let $A_i$ be the event that the first of these is carrying $a_1$ marks, the second $a_2$, and so on. Let $\psi^{(j)}_i$ be the time at which the marks following the $j$th of these particles split. By the symmetry lemma, given $\psi_{i-1} = s_{i-1}$ (where we take $s_0=0$), these times are independent with
\[\Q^{k,T}(\psi^{(j)}_i > s_i | \psi_{i-1} = s_{i-1},\, A_i) = \Q^{a_j,T-s_{i-1}}(\psi_1 > s_i-s_{i-1}) = \Big(\frac{e^{(\beta-\alpha)(T-s_i)}-1}{e^{(\beta-\alpha)(T-s_{i-1})}-1}\Big)^{a_j-1}.\]
Then, since the event $\{\psi_i > s_i\} = \bigcap_{j}\{\psi_i^{(j)} > s_i\}$,
\[\Q^{k,T}(\psi_i > s_i | \psi_{i-1} = s_{i-1},\, A_i) = \prod_{j=1}^i \Big(\frac{e^{(\beta-\alpha)(T-s_i)}-1}{e^{(\beta-\alpha)(T-s_{i-1})}-1}\Big)^{a_j-1}.\]
Since $\sum_{j=1}^i (a_j-1) = k-i$, we get
\[\Q^{k,T}(\psi_i > s_i | \psi_{i-1} = s_{i-1},\, A_i) = \Big(\frac{e^{(\beta-\alpha)(T-s_i)}-1}{e^{(\beta-\alpha)(T-s_{i-1})}-1}\Big)^{k-i}.\]
This does not depend on $a_1,\ldots, a_i$, so $\psi_i$ is independent of $\mathcal H'$, and summing over the possible values we obtain
\[\Q^{k,T}(\psi_i > s_i \, |\, \psi_{i-1} = s_{i-1}) = \Big(\frac{e^{(\beta-\alpha)(T-s_i)}-1}{e^{(\beta-\alpha)(T-s_{i-1})}-1}\Big)^{k-i}.\]
Differentiating gives
\[f^\Q_k(s_1,\ldots,s_{k-1}) = (k-1)! (\beta-\alpha)^{k-1} \prod_{i=1}^{k-1} e^{(\beta-\alpha)(T-s_i)}\frac{(e^{(\beta-\alpha)(T-s_i)}-1)^{k-i-1}}{(e^{(\beta-\alpha)(T-s_{i-1})}-1)^{k-i}}.\]
The product telescopes to give the answer.
\end{proof}

\begin{prop}\label{BDdensity}
Let $s_0=0$. The vector $(\S^k_1(T),\ldots,\S^k_{k-1}(T))$ of ordered split times under $\P$ is independent of $\mathcal H$ and has a joint density $f_k^T(s_1,\ldots, s_{k-1})$ equalling
\[\frac{k!(\beta e^{(\beta-\alpha)T}-\alpha)^k (\beta-\alpha)^{2k-1}}{(e^{(\beta-\alpha)T}-1)^{k-1} e^{(\beta-\alpha)T}} \int_0^1 (1-y)^{k-1} \prod_{j=0}^{k-1} \frac{e^{(\beta-\alpha)(T-s_j)}}{(\beta(1-y)e^{(\beta-\alpha)(T-s_j)} + \beta y - \alpha)^2} \d y\]
if $\alpha\neq\beta$, and
\[\frac{k!(\beta T+1)^k}{T^{k-1}} \int_0^1 (1-y)^{k-1} \prod_{j=0}^{k-1} \frac{1}{(\beta(1-y)(T-s_j) + 1)^2} \d y\]
if $\alpha=\beta$.
\end{prop}

\begin{proof}
Again we give the proof in the non-critical case $\alpha\neq\beta$. The critical case is identical. We start with Proposition \ref{firstprop}, which tells us that for any measurable functional $F$,
\begin{equation}\label{firstpropre}
\P\Big[\frac{1}{N_T^{(k)}} \sum_{u\in \Nc_T^{(k)}} F(u) \, \Big| \, N_T \ge k\Big] = \frac{\P[N_T^{(k)}]}{\P(N_T\ge k)(k-1)!} \int_0^\infty (e^z-1)^{k-1} \Q^{k,T}\Big[e^{-zN_T} F(\xi_T)\Big]\,\d z.
\end{equation}
The independence of the spine split times and $\mathcal H'$ under $\Q^{k,T}$ (established in Lemma \ref{Qsplitdist}), together with \eqref{firstpropre} and Proposition \ref{Qmgf_gen_prop}, imply that the split times under $\P$ are independent of $\mathcal H$.

Returning to \eqref{firstpropre} again, we get that in particular
\begin{multline*}
f_k^T(s_1,\ldots, s_{k-1})\\
= \frac{\P[N_T^{(k)}]}{\P(N_T\ge k)(k-1)!} \hspace{-0.5mm}\int_0^\infty (e^z-1)^{k-1} f^\Q_k(s_1,\ldots,s_{k-1})\Q[e^{-zN_T}|\psi_1 = s_1,\ldots,\psi_{k-1}=s_{k-1}] \,\d z.
\end{multline*}
However we also know from Proposition \ref{Qmgf_gen_prop} that
\[\Q[e^{-z\tilde N_T}|\psi_1 = s_1,\ldots,\psi_{k-1}=s_{k-1}] = \prod_{i=0}^{k-1}\Big(e^{-r(m-1)(T-s_i)}\frac{u(F(e^{-z},T-s_i))}{u(e^{-z})}\Big)\]
where $s_0=0$, $F(\theta,t) = \P[\theta^{N_t}]$ and $u(\theta) = \P[\theta^L]-\theta$.
Of course since all births are binary, all particles are either spines or ordinary; so since there are $k$ spines at time $T$ almost surely under $\Q$, $N_T = \tilde N_T + k$. Thus, by \eqref{kolback} and \eqref{Fpartial},
\[\Q[e^{-zN_T}|\psi_1 = s_1,\ldots,\psi_{k-1}=s_{k-1}] = e^{-zk} \prod_{i=0}^{k-1} \Big( \frac{\beta-\alpha}{\beta(1-e^{-z})e^{(\beta-\alpha)(T-s_i)} + \beta e^{-z}-\alpha}\Big)^2.\]
Plugging this into our formula for $f_k^T(s_1,\ldots, s_{k-1})$ above gives
\begin{multline*}
f_k^T(s_1,\ldots, s_{k-1}) = \frac{\P[N_T^{(k)}]}{\P(N_T\ge k)(k-1)!} \int_0^\infty e^{-z}(1-e^{-z})^{k-1} f^\Q_k(s_1,\ldots,s_{k-1})\\
\cdot\prod_{i=0}^{k-1} \frac{(\beta-\alpha)^2}{(\beta(1-e^{-z})e^{(\beta-\alpha)(T-s_i)}+\beta e^{-z} - \alpha)^2}\,\d z.
\end{multline*}
By \eqref{Pcond} and Lemma \ref{Qsplitdist}, this becomes
\[\frac{k!(\beta e^{(\beta-\alpha)T}-\alpha)^k(\beta-\alpha)^{2k-1}}{e^{(\beta-\alpha)T}(e^{(\beta-\alpha)T}-1)^{k-1}}\hspace{-1.5mm} \int_0^\infty \hspace{-1.8mm} e^{-z}(1-e^{-z})^{k-1} \hspace{-0.5mm}\prod_{i=0}^{k-1}\hspace{-0.5mm}  \frac{e^{(\beta-\alpha)(T-s_i)}}{(\beta(1-e^{-z})e^{(\beta-\alpha)(T-s_i)}+\beta e^{-z} - \alpha)^2}\d y.\]
Making the substitution $y=e^{-z}$ gives the result.
\end{proof}

\subsection{Describing the partition process}

We recall now the partition $Z_0,Z_1,\ldots$ which contained the information about the marks following each of the distinct spine particles, without the information about the split times.

\begin{lem}\label{QtopologyBD}
The partition $Z_0,Z_1,\ldots$ has the following distribution under $\Q^{k,T}_T$:
\begin{itemize}
\item If $Z_i$ consists of $i+1$ blocks of sizes $a_1,\ldots,a_{i+1}$, then the $j$th block will split next with probability $\frac{a_j-1}{k-i-1}$ for each $j=1,\ldots,i+1$.
\item When a block of size $a$ splits, it splits into two new blocks, and the probability that these blocks have sizes $l$ and $a-l$ is $\frac{1}{a-1}$ for each $l=1,\ldots,a-1$.
\end{itemize}
\end{lem}

\begin{proof}
Suppose that we are given $\psi_{i}=s$. For the first part, by the symmetry lemma, the probability that the $j$th block splits next is
\[\int_0^{T-s} \Q^{a_j,T-s}(\psi_1\in\d t) \prod_{l\neq j} \Q^{a_l,T-s}(\psi_1 > t)\]
which by Lemma \ref{Qbots} equals
\[\int_0^{T-s}\bigg(-\frac{\d}{\d t} \Big(\frac{\P[N^{(a_j)}_{T-s-t}]}{\P[N^{(a_j)}_{T-s}]} e^{(m-1)rt} \Big)\bigg) \prod_{l\neq j} \frac{\P[N^{(a_l)}_{T-s-t}]}{\P[N^{(a_l)}_{T-s}]} e^{(m-1)rt}\,\d t.\]
If $\alpha\neq\beta$, then applying \eqref{Pratio}, the above becomes
\begin{align*}
&\int_0^{T-s}\bigg(-\frac{\d}{\d t} \Big(\frac{e^{(\beta-\alpha)(T-s-t)}-1}{e^{(\beta-\alpha)(T-s)}-1}\Big)^{a_j-1}\bigg) \prod_{l\neq j} \Big(\frac{e^{(\beta-\alpha)(T-s-t)}-1}{e^{(\beta-\alpha)(T-s)}-1}\Big)^{a_l-1} \,\d t\\
&= (a_j-1)(\beta-\alpha)\int_0^{T-s} e^{(\beta-\alpha)(T-s-t)}\frac{(e^{(\beta-\alpha)(T-s-t)}-1)^{a_j-2}}{(e^{(\beta-\alpha)(T-s)}-1)^{a_j-1}}\prod_{l\neq j} \Big(\frac{e^{(\beta-\alpha)(T-s-t)}-1}{e^{(\beta-\alpha)(T-s)}-1}\Big)^{a_l-1} \,\d t\\
&= (a_j-1)(\beta-\alpha)\int_0^{T-s} \frac{e^{(\beta-\alpha)(T-s-t)}}{e^{(\beta-\alpha)(T-s-t)}-1} \Big(\frac{e^{(\beta-\alpha)(T-s-t)}-1}{e^{(\beta-\alpha)(T-s)}-1}\Big)^{k-i-1} \,\d t.
\end{align*}
Since the integrand does not depend on $a_j$, and we know the sum of the above quantity over $j=1,\ldots,i+1$ must equal $1$ (since one of the blocks must split first), we get
\[(\beta-\alpha)\int_0^{T-s} \frac{e^{(\beta-\alpha)(T-s-t)}}{e^{(\beta-\alpha)(T-s-t)}-1} \Big(\frac{e^{(\beta-\alpha)(T-s-t)}-1}{e^{(\beta-\alpha)(T-s)}-1}\Big)^{k-i-1} \,\d t = \frac{1}{k-i-1}\]
and therefore the probability that the $j$th block splits next equals $\frac{a_j-1}{k-i-1}$ as claimed. If $\alpha=\beta$ then applying \eqref{critBDratio} in place of \eqref{Pratio} gives the same result.

For the second part, let $\rho^1_t$ be the number of marks following the first spine particle at time $t$. From the definition of $\Q^{k,T}$,
\[\Q^{k,T}(\rho^1_t = i\,|\,\tau_\emptyset=t) = \frac{\P[g_{k,T}\ind_{\{\rho^1_t = i\}}\,|\,\tau_\emptyset=t]}{\P[g_{k,T}\,|\,\tau_\emptyset=t]}.\]
By the Markov property, since each mark chooses uniformly from amongst the children available,
\[\P[g_{k,T}\ind_{\{\rho^1_t = i\}}\,|\,\tau_\emptyset=t] = \frac{\beta}{\beta+\alpha}\binom{k}{i}\frac{1}{2^k} \P[g_{i,T-t}]\P[g_{k-i,T-t}].\]
Lemma \ref{gproj} tells us that $\P[g_{j,s}] = \P[N^{(j)}_s]$ for any $j$ and $s$, so
\[\P[g_{k,T}\ind_{\{\rho^1_t = i\}}\,|\,\tau_\emptyset=t] = \frac{\beta}{\beta+\alpha}\binom{k}{i}\frac{1}{2^k} \P[N^{(i)}_{T-t}]\P[N^{(k-i)}_{T-t}].\]
If $\alpha\neq\beta$, then applying \eqref{Pmom} gives
\begin{align*}
\P[g_{k,T}\ind_{\{\rho^1_t = i\}}\,|\,\tau_\emptyset=t] &= \frac{\beta}{\beta+\alpha}\binom{k}{i}\frac{1}{2^k} i!(k-i)!\Big(\frac{\beta}{\beta-\alpha}\Big)^{k-2}e^{(\beta-\alpha)(T-t)}(e^{(\beta-\alpha)(T-t)}-1)^{k-2}\\
&= \frac{\beta}{\beta+\alpha}\frac{k!}{2^k} \Big(\frac{\beta}{\beta-\alpha}\Big)^{k-2}e^{(\beta-\alpha)(T-t)}(e^{(\beta-\alpha)(T-t)}-1)^{k-2}.
\end{align*}
Since this does not depend upon $i$, we deduce that the distribution of $\rho^1_t$ under $\Q^{k,T}$ must be uniform. The case $\alpha=\beta$ is the same but using \eqref{critBDratio} in place of \eqref{Pmom}. The result now follows from the symmetry lemma.
\end{proof}

\subsection{Proofs of Theorems \ref{noncritBDthm} and \ref{critBDthm}: explicit distribution functions for unordered split times}

We now have all the ingredients to prove our theorem on the distribution of the split times. We begin with the non-critical case.

\begin{proof}[Proof of Theorem \ref{noncritBDthm}]
By Proposition \ref{BDdensity}, the \emph{ordered} split times are independent of $\mathcal H$ and have density
\begin{multline*}
f_k^T(s_1,\ldots, s_{k-1})\\
= \frac{k!(\beta E_0-\alpha)^k (\beta-\alpha)^{2k-1}}{(E_0-1)^{k-1} E_0} \int_0^1 (1-y)^{k-1} \prod_{j=0}^{k-1} \frac{e^{(\beta-\alpha)(T-s_j)}}{(\beta(1-y)e^{(\beta-\alpha)(T-s_j)} + \beta y - \alpha)^2} \d y
\end{multline*}
for any $0\le s_1\le\ldots\le s_{k-1}\le 1$, where $s_0=0$.
Therefore (see Lemma \ref{unordering}) the \emph{unordered} split times are independent of $\mathcal H$ and have density
\begin{multline*}
\tilde f_k^T(s_1,\ldots, s_{k-1}) \\
= \frac{k (\beta E_0-\alpha)^k (\beta-\alpha)^{2k-1}}{(E_0-1)^{k-1} E_0} \int_0^1 (1-y)^{k-1} \prod_{j=0}^{k-1} \frac{e^{(\beta-\alpha)(T-s_j)}}{(\beta(1-y)e^{(\beta-\alpha)(T-s_j)} + \beta y - \alpha)^2} \d y.
\end{multline*}
Using Lemma \ref{integrateout} to integrate over $s_j$ for each $j=1,\ldots,k-1$, we get
\begin{multline*}
\P(\tilde \S_1 \ge s_1, \ldots, \tilde\S_{k-1} \ge s_{k-1} | N_T\ge k)\\
 = \frac{k(\beta E_0-\alpha)^k(\beta-\alpha)}{(E_0-1)^{k-1}E_0} \hspace{-1.1mm}\int_0^1 \hspace{-1mm}(1-y)^{k-1} \bigg(\hspace{-0.3mm}\prod_{j=1}^{k-1} \frac{E_j-1}{\beta(1-y)E_j + \beta y - \alpha}\bigg) \frac{E_0}{(\beta(1-y)E_0 + \beta y-\alpha)^2} \d y.
 \end{multline*}
Substituting $\theta=1-y$ and simplifying,
\begin{align*}
&\P(\tilde \S_1 \ge s_1, \ldots, \tilde\S_{k-1} \ge s_{k-1} | N_T\ge k)\\ 
&\hspace{5mm}= \frac{\beta k(E_0-\alpha/\beta)^k(\beta-\alpha)}{(E_0-1)^{k-1}E_0} \int_0^1 \bigg(\prod_{j=1}^{k-1} \frac{\theta (E_j-1)}{\theta E_j + 1-\theta - \alpha/\beta}\bigg) \frac{E_0}{(\beta-\alpha+\beta\theta (E_0-1))^2} \d \theta\\
&\hspace{5mm}= \frac{\beta k(E_0-\alpha/\beta)^k}{(E_0-1)^{k-1}(\beta-\alpha)} \int_0^1 \bigg(\prod_{j=1}^{k-1} \Big(1-\frac{1}{1+\theta \frac{\beta}{\beta-\alpha}(E_j-1)}\Big)\bigg) \frac{1}{(1+\theta\frac{\beta}{\beta-\alpha}(E_0-1))^2} \d \theta.
 \end{align*}
We can now apply the second part of Lemma \ref{partialfrac2}, with $e_j = \frac{\beta}{\beta-\alpha} (E_j-1)$ which gives
\begin{multline*}
\P(\tilde \S_1 \ge s_1, \ldots, \tilde\S_{k-1} \ge s_{k-1} | N_T\ge k) \\
= \frac{\beta k(E_0-\alpha/\beta)^k}{(E_0-1)^{k-1}(\beta-\alpha)} \Bigg[\frac{1}{1+e_0}\prod_{i=1}^{k-1}\frac{e_i}{e_i-e_0} + \sum_{j=1}^{k-1} \frac{e_j}{(e_j-e_0)^2}\bigg(\prod_{\substack{i=1\\ i\neq j}}^{k-1} \frac{e_i}{e_i-e_j}\bigg)\log\Big(\frac{1+e_0}{1+e_j}\Big)\Bigg].
\end{multline*}
The result follows.
\end{proof}

We now do the critical case, which is almost identical.

\begin{proof}[Proof of Theorem \ref{critBDthm}]
By Proposition \ref{BDdensity}, the \emph{ordered} split times are independent of $\mathcal H$ and have density
\begin{align*}
f_k^T(s_1,\ldots, s_{k-1}) &= \frac{k!(\beta T+1)^k}{T^{k-1}} \int_0^1 (1-y)^{k-1} \prod_{j=0}^{k-1} \frac{1}{(\beta(1-y)(T-s_j) + 1)^2} \d y.\\
&= \frac{k!(\beta T+1)^k}{T^{k-1}} \int_0^1 \frac{1}{(1+\theta\beta T)^2} \prod_{j=1}^{k-1} \frac{\theta}{(1+\theta\beta(T-s_j))^2} \d \theta.
\end{align*}
for any $0\le s_1\le\ldots\le s_{k-1}\le 1$, where $s_0=0$.
Therefore (see Lemma \ref{unordering}) the \emph{unordered} split times are independent of $\mathcal H$ and have density
\[\tilde f_k^T(s_1,\ldots,s_{k-1}) = \frac{k(\beta T+1)^k}{T^{k-1}} \int_0^1 \frac{1}{(1+\theta\beta T)^2} \prod_{j=1}^{k-1} \frac{\theta}{(1+\theta\beta(T-s_j))^2} \d \theta.\]
Integrating over $s_j$ for each $j=1,\ldots,k-1$, we get
\begin{align*}
&\P(\tilde \S_1 \ge s_1, \ldots, \tilde\S_{k-1} \ge s_{k-1} | N_T\ge k)\\
&\hspace{30mm} = k\beta T \Big(1+\frac{1}{\beta T}\Big)^k \int_0^1 \frac{1}{(1+\theta\beta T)^2} \prod_{j=1}^{k-1} \Big(1-\frac{1}{1+\theta\beta(T-s_j)}\Big) \d \theta\\
&\hspace{30mm} = k T \Big(1+\frac{1}{\beta T}\Big)^k \int_0^1 \frac{1}{(1+\theta T)^2} \prod_{j=1}^{k-1} \Big(1-\frac{1}{1+\theta(T-s_j)}\Big) \d \theta.
 \end{align*}
We can now apply the second part of Lemma \ref{partialfrac2}, with $e_j = (T-s_j)$ and $s_0=0$. This gives
\begin{multline*}
\P(\tilde \S_1 \ge s_1, \ldots, \tilde\S_{k-1} \ge s_{k-1} | N_T\ge k)\\
= k T \Big(1+\frac{1}{\beta T}\Big)^k \Bigg[\frac{1}{1+e_0}\prod_{i=1}^{k-1}\frac{e_i}{e_i-e_0} - \sum_{j=1}^{k-1} \frac{e_j}{(e_j-e_0)^2}\bigg(\prod_{\substack{i=1\\ i\neq j}}^{k-1} \frac{e_i}{e_i-e_j}\bigg)\log\Big(\frac{1+e_0}{1+e_j}\Big)\Bigg].
\end{multline*}
The result now follows from some simple manipulation.
\end{proof}

\section{The near-critical scaling limit}\label{nearcritsec}

We now let our offspring distribution depend on $T$, writing $\P_T$ in place of $\P$. We suppose that $m_T := \P_T[L] = 1+\mu/T + o(1/T)$ for some $\mu\in\R$, and $\P_T[L(L-1)] = \sigma^2 + o(1)$ for some $\sigma>0$. We also assume that $L^2$ is uniformly integrable (that is, for all $\eps>0$ there exists $M$ such that $\sup_T \P_T[L^2\ind_{\{L\ge M\}}]<\eps$). We define $\Q^{k,T}_T$ just as before, except that it is defined relative to $\P^k_T$ instead of $\P^k$.

In order to prove our results we would like some conditions on the higher moments of $L$. The next lemma ensures that we may make some further assumptions without loss of generality.

\begin{lem}\label{momrelax}
Fix $k\ge 1$. Under $\P_T$, there exists a coupling between our Galton-Watson tree with offspring distribution $L$ (and its $k$ chosen particles) and another Galton-Watson tree with offspring distribution $\tilde L$ satisfying
\begin{itemize}
\item $\P_T[\tilde L] = 1+\mu/T + o(1/T)$;
\item $\P_T[\tilde L(\tilde L-1)] = \sigma^2 + o(1)$;
\item there exists a deterministic sequence $J(T)=o(T)$ such that $\P_T(\tilde L = j) = 0$ for all $j\ge J(T)$,
\end{itemize}
such that for each $k$, conditionally on $N_T\ge k$, with probability tending to $1$, the two trees induced by the $k$ chosen particles are equal until time $T$.
\end{lem}

The proof of this lemma is interesting, but not really relevant to the rest of our investigation, so we have included it in the appendix.

In light of Lemma \ref{momrelax}, we further assume without loss of generality that there is a deterministic sequence $J(T)=o(T)$ such that our offspring distribution $L$ satisfies $\P_T(L=j)=0$ for all $j\ge J(T)$; in particular, for any $j\ge 3$,
\begin{equation}\label{momcond}
\P_T[L^{(j)}] = \sum_{i=1}^{J(T)}i^{(j)}p^{(T)}_i \le J(T)^{j-2}\sum_{i=1}^{J(T)}i(i-1)p^{(T)}_i = J(T)^{j-2}(\sigma^2+o(1)) = o(T^{j-2}).
\end{equation}

\subsection{Estimating moments and generating functions under $\P$}

In Section \ref{BDgenfuncts}, we calculated generating functions and moments of the population size under $\P$ precisely for birth-death processes. With more complicated offspring distributions this is no longer possible, but the near-criticality ensures that we can give good approximations.

\begin{lem}\label{genGWmom}
For $k\ge 1$, the $k$th descending moment $M_k(t) = \P[ N_t^{(k)}]$ of any continuous-time Galton-Watson process satisfies
\[M'_k(t)  =  kr(m-1)M_k(t) + r\sum_{j =2 }^k \binom{k}{j}\, \P[L^{(j)}]\, M_{k+1-j}(t).\]
\end{lem}

\begin{proof}
As before let $F(\theta,t) = \P[\theta^{N_t}]$, and let $u(\theta) = \P[\theta^L] - \theta$. Then $F$ and $u$ satisfy the Kolmogorov forward equation \eqref{kolfor}
\begin{equation}\label{nearcritforward}
\frac{\partial F(\theta,t)}{\partial t} = ru(\theta) \frac{\partial F(\theta,t)}{\partial \theta}.
\end{equation}
Note that
\begin{equation}\label{nearcritMtoF}
M_k(t) = \Big[ \frac{\partial^k }{\partial \theta^k} F(\theta,t) \Big]_{\theta=1},
\end{equation}
so, using the fact that $F$ is smooth,
\[\frac{d}{dt} M_k(t) = \frac{d}{dt} \Big[ \frac{\partial^k }{\partial \theta^k} F(\theta,t) \Big]_{\theta=1} = \Big[ \frac{\partial }{\partial t} \frac{\partial^k }{\partial \theta^k} F(\theta,t) \Big]_{\theta=1} = \Big[  \frac{\partial^k }{\partial \theta^k}\frac{\partial }{\partial t} F(\theta,t) \Big]_{\theta=1}.\]
Applying \eqref{nearcritforward},
\[\frac{d}{dt} M_k(t) = \Big[  \frac{\partial^k }{\partial \theta^k} \Big(r u(\theta) \frac{\partial F}{\partial \theta}  \Big) \Big]_{\theta=1}\]
so using \eqref{nearcritMtoF} again,
\[\frac{d}{dt} M_k(t) = r\sum_{j=0}^{k} \binom{k}{j} u^{(j)}(1)\, M_{k+1-j}(t).\]
Finally, $u(1) = 0$, $u'(1) = (m-1)$, and $u^{(j)}(1) = \P[L^{(j)}]$ for $j\ge 2$.
\end{proof}

For real-valued functions $f$ and $g$, we write $f(x) = o(g(x))$ to mean that $f(x)/g(x)\to 0$ as $x\to\infty$.

\begin{lem}\label{scaledPmoments}
If $\mu\neq 0$ then the descending moments at scaled times satisfy
\[\lim_{T\to\infty} \frac{\P_T[N_{sT}^{(k)}]}{T^{k-1}} = \Big( \frac{\sigma^2}{2\mu} \Big)^{k-1} k! e^{r\mu s} ( e^{r\mu s} - 1 )^{k-1}\]
for all $k\ge 1$ and $s\in[0,1]$.
If $\mu = 0$ then instead
\[\lim_{T\to\infty} \frac{\P_T[N_{sT}^{(k)}]}{T^{k-1}} = k!\Big(\frac{r\sigma^2 s}{2}\Big)^{k-1}\]
for all $k\ge 1$ and $s\in[0,1]$.
\end{lem}

\begin{proof}
We proceed by induction. Note that both statements are true for $k=1$. Letting $M_k(t) = \P_T[N_t^{(k)}]$, by Lemma \ref{genGWmom} we have
\begin{align*}
M'_k(t)  =  kr(m_T-1)M_k(t) + r\sum_{j =2 }^k \binom{k}{j} \P_T[L^{(j)}]\,M_{k+1-j}(t).
\end{align*}
So letting $\hat M_k(s) = M_k(sT)$, we have
\begin{align}
\hat M'_k(s) &= T \Big(  kr(m_T-1)\hat M_k(s) + r\sum_{j =2 }^k \binom{k}{j} \P_T[L^{(j)}]\,\hat M_{k+1-j}(s)  \Big)\nonumber\\
&= k r\mu \hat M_k(s) + T r \binom{k}{2} \sigma^2 \hat M_{k-1}(s) + o(T^{k-1})\label{bothmucases}
\end{align}
where we used the induction hypothesis to get the last equality.

We now consider the cases $\mu\neq 0$ and $\mu = 0$ separately. In the case $\mu\neq 0$, using the integrating factor $e^{-kr\mu s}$, and applying the induction hypothesis again, we get
\begin{equation}\label{intfactor}
\frac{\d}{\d s}\big(e^{-kr\mu s} \hat M_k(s)\big) = T^{k-1} k! (k-1)r\mu \Big(\frac{\sigma^2}{2\mu}\Big)^{k-1} e^{-(k-1)r\mu s} (e^{r\mu s}-1)^{k-2} + e^{-kr\mu s} O(T^{k-2}).
\end{equation}
Noting that
\[(k-1)r\mu e^{-(k-1)r\mu s} (e^{r\mu s}-1)^{k-2} = \frac{\d}{\d s}\big(e^{-(k-1)r\mu s}(e^{r\mu s}-1)^{k-1}\big),\]
by integrating \eqref{intfactor} we obtain
\[e^{-kr\mu s} \hat M_k(s) = T^{k-1} k! \Big(\frac{\sigma^2}{2\mu}\Big)^{k-1}e^{-(k-1)r\mu s} (e^{r\mu s}-1)^{k-1} + e^{-kr\mu s} O(T^{k-2}).\]
Multiplying through by $e^{kr\mu s}$ gives the result for $\mu\neq 0$.

If $\mu=0$, then from \eqref{bothmucases} and the induction hypothesis, we have
\[\hat M_k'(s) = T^{k-1} k! \Big(\frac{r\sigma^2}{2}\Big)^{k-1} (k-1) s^{k-2} + o(T^{k-1})\]
and integrating directly gives the result.
\end{proof}

\subsection{Asymptotics for the generating function}

Define
\[F_T(\theta,t) = \P_T[\theta^{N_t}], \hspace{10mm} u_T(\theta) = \P_T[\theta^L] - \theta,\]
and
\[f_T(\phi,s) = T \big( 1 - \P_T[e^{- \frac{ \phi}{T} N_{sT} } ] \big) = T(1-F_T(e^{-\phi/T},sT)).\]
The following result will be important for approximating terms that arise from Campbell's formula.

\begin{lem}\label{fTconv}
For each $\phi\ge 0$,
\[f_T(\phi,s)\to f(\phi,s)\]
and
\[T^2 u_T(F_T(e^{-\phi/T},sT)) \to -\mu f(\phi,s) + \frac{\sigma^2}{2}f(\phi,s)^2\]
as $T\to\infty$, uniformly over $s\in[0,1]$, where
\[f(\phi,s) = \frac{ \phi e^{ \mu r s }}{ 1 + \frac{\sigma^2}{2\mu}  \phi ( e^{ \mu r s } - 1 )} \hspace{4mm} \hbox{ if } \hspace{1mm} \mu\neq 0 \]
and
\[f(\phi,s) = \frac{\phi}{1+r\sigma^2\phi s/2} \hspace{4mm} \hbox{ if } \hspace{1mm} \mu= 0.\]
\end{lem}

\begin{proof}
First we show that for each $\phi$, $f_T$ is bounded in $T>0$ and $s\in[0,1]$. Note that $x \mapsto 1 - e^{- \kappa x}$ is concave and increasing for any $\kappa\ge 0$, so by Jensen's inequality,
\[f_T(\phi,s) = T \big( 1 - \P[e^{- \frac{\phi}{T} N_{sT} } ]\big) \leq T \big( 1 - e^{- \frac{ \phi}{T} \P_T[N_{sT}]}\big) \leq T \big( 1 - e^{- \frac{ \phi}{T}\exp(r\mu + o(1))}\big).\]
Applying the inequality $1-e^{-x}\le x$, we see that
\[f_T(\phi,s) \le \phi e^{r\mu + o(1)}.\]
Now, with $F_T(\theta,t) = \P_T[\theta^{N_t}]$, we have
\begin{equation}\label{ftoF}
\frac{\partial f_T(\phi,s)}{\partial s} = \frac{\partial}{\partial s} \big(T(1-F_T(e^{-\phi/T},sT))\big) = -T^2 \frac{\partial F_T(e^{-\phi/T},t)}{\partial t}\Big|_{t=sT}.
\end{equation}
By the Kolmogorov backwards equation \eqref{kolback},
\begin{equation}\label{FtoUKol}
\frac{\partial}{\partial t}F_T(\theta,t) = ru_T(F_T(\theta,t)) = r\P_T[F_T(\theta,t)^L] - rF_T(\theta,t),
\end{equation}
so
\[\frac{\partial f_T(\phi,s)}{\partial s} = T^2 r \sum_{j=0}^\infty p_j^{(T)} \big( F(e^{-\phi/T},sT) - F(e^{-\phi/T},sT)^j\big) = T^2 r \sum_{j=0}^\infty p_j^{(T)} \Big( 1-\frac{f_T}{T} - \Big(1-\frac{f_T}{T}\Big)^j\Big)\]
where $p^{(T)}_j = \P_T(L=j)$. Expanding $(1-f_T/T)^j$, we get
\begin{align*}
\frac{\partial f_T(\phi,s)}{\partial s} &= T^2 r\sum_{j=0}^\infty p^{(T)}_j \bigg( (j-1) \frac{f_T}{T} - \frac{j(j-1)f_T^2}{2T^2} - \sum_{i=3}^j \binom{j}{i} \Big(-\frac{f_T}{T}\Big)^i \bigg)\\
&= r\mu f_T - \frac{r\sigma^2}{2} f_T^2 + o(1) -T^2 r \sum_{j=0}^\infty p^{(T)}_j \sum_{i=3}^j \binom{j}{i}\Big(-\frac{f_T}{T}\Big)^i.
\end{align*}
Swapping the order of summation, this becomes
\begin{align}
\frac{\partial f_T(\phi,s)}{\partial s} &= r\mu f_T - \frac{r\sigma^2}{2} f_T^2 + o(1) - T^2 r \sum_{i=2}^\infty \frac{1}{i!}\Big(-\frac{f_T}{T}\Big)^i \sum_{j=i}^\infty p_j^{(T)} j(j-1)\ldots (j-i+1)\nonumber\\
&= r\mu f_T - \frac{r\sigma^2}{2} f_T^2 + o(1) - T^2 r \sum_{i=2}^\infty \frac{1}{i!}\Big(-\frac{f_T}{T}\Big)^i \P_T[L^{(i)}]\nonumber\\
&= r\mu f_T - \frac{r\sigma^2}{2} f_T^2 + o(1)\label{partialfeqn}
\end{align}
since $f_T$ is bounded and $\P_T[L^{(i)}]=o(T^{i-2})$ for each $i\ge 3$ (see \eqref{momcond}). Note in particular that the $o(1)$ term is uniform in $s$.

Note that $f$ is the solution to
\[\frac{\partial f}{\partial s} = r\mu f - \frac{r\sigma^2}{2} f^2\]
with $f(\phi,0) = \phi$. Setting $h_T(\phi,s) = f_T(\phi,s)-f(\phi,s)$ we have
\[\frac{\partial h_T}{\partial s} = r\mu (f_T - f) - \frac{r\sigma^2}{2} (f_T^2-f^2) + o(1)\]
where the $o(1)$ term is uniform in $s$. Integrating over $s$ with $\phi$ fixed,
\[h_T(\phi,s) = h_T(\phi,0) + r\mu\int_0^s h_T(\phi,s')\d s' - \frac{r\sigma^2}{2}\int_0^s h_T(\phi,s')(f_T(\phi,s')+f(\phi,s'))\d s + o(1).\]
For fixed $\phi$, both $f_T$ and $f$ are bounded in $s$ and $T$, say by $M_\phi$. Also $|h_T(\phi,0)| = T(1-e^{-\phi/T}) - \phi = o(1)$. Thus
\[|h_T(\phi,s)| \le r\int_0^s |h_T(\phi,s')|(\mu + \sigma^2 M_\phi/2) \d s' + o(1),\]
where again the $o(1)$ term is uniform in $s$. Gronwall's inequality then tells us that $|h_T(\phi,s)| \to 0$ uniformly in $s$. This proves the first part of the lemma.

The second part of the lemma is now implicit in our calculations above: by \eqref{FtoUKol} and then \eqref{ftoF},
\[u_T(F_T(e^{-\phi/T},sT)) = \frac{1}{r}\frac{\partial}{\partial t} F_T(e^{-\phi/T},t) |_{t=sT} = -\frac{1}{rT^2} \frac{\partial f_T(\phi,s)}{\partial s}.\]
Applying \eqref{partialfeqn} tells us that
\[T^2 u_T(F_T(e^{-\phi/T},sT)) = - \mu f_T + \frac{\sigma^2}{2} f_T^2 + o(1),\]
and by the first part of the lemma we get
\[T^2 u_T(F_T(e^{-\phi/T},sT)) \to -\mu f + \frac{\sigma^2}{2} f^2.\qedhere\]
\end{proof}

\begin{lem}\label{scaledPext}
For any $s\in(0,1]$, as $T\to\infty$,
\[T\P_T(N_{sT}>0) \to \frac{2\mu e^{\mu r s}}{\sigma^2(e^{\mu r s}-1)}\hspace{5mm} \hbox{ if } \mu\neq 0\]
and
\[T\P_T(N_{sT}>0) \to \frac{2}{r\sigma^2 s}\hspace{5mm} \hbox{ if } \mu = 0.\]
\end{lem}

\begin{proof}
Note that $\P_T(N_t=0) = F_T(0,t)$, and so satisfies the Kolmogorov backwards equation \eqref{kolback}. Thus the proof of Lemma \ref{fTconv} works exactly the same for
\[T\P_T(N_{sT}>0) = T(1-\P_T(N_{sT}=0)) = T(1-F_T(0,sT)),\]
except for showing that $T\P_T(N_{sT}>0)$ is bounded---we can no longer apply Jensen's inequality.

Instead, we note that in the critical case $m_T=1$ the boundedness is well known (see for example \cite[Chapter III, Section 7, Lemma 2]{athreya_ney:branching_processes}). When $m_T\neq 1$, let $\bar p^{(T)}_0 = p^{(T)}_0$ and for $j\ge 1$,
\[\bar p^{(T)}_j = p^{(T)}_j + (1-m_T) 2^{-j}/j.\]
This gives us a new offspring distribution $\bar L$ that is critical (and has finite variance). We can then easily construct a coupling between $N_t$ and $\bar N_t$, where $\bar N_t$ is the number of particles in a branching process with offspring distribution $\bar L$, such that
\begin{itemize}
\item if $m_T<1$, then $N_t\le \bar N_t$ for all $t\ge 0$;
\item if $m_T>1$, then $N_t\ge \bar N_t$ for all $t\ge 0$.
\end{itemize}
In the case $m_T<1$, we have $T\P(N_{sT}>0) \le T\P(\bar N_{sT}>0)$, which is bounded. In the case $m_T>1$, we have
\[\P_T(N_{sT}>0) = \Q^{1,sT}_T\Big[\frac{\P_T[N_{sT}]}{N_{sT}}\Big] = e^{r(m_T-1)sT} \Q^{1,sT}_T\Big[\frac{1}{N_{sT}}\Big]\]
and similarly for $\bar N_{sT}$ with its equivalent measure $\bar\Q^{1,sT}_T$. Since $T\P(\bar N_{sT}>0)$ is bounded, we get that $T \bar\Q^{1,sT}_T[1/\bar N_{sT}]$ is bounded, but
\[ \Q^{1,sT}_T\Big[\frac{1}{N_{sT}}\Big] \le  \bar\Q^{1,sT}_T\Big[\frac{1}{\bar N_{sT}}\Big], \]
so $T\Q^{1,sT}_T[1/N_{sT}]$ is bounded and therefore $T\P_T(N_{sT}>0)$ is also bounded. This completes the proof.
\end{proof}

\subsection{Spine split times under $\Q^{k,T}_T$}

We now want to feed our calculations for moments and generating functions under $\P$ into understanding the spine split times under $\Q$, as in Lemma \ref{Qsplitdist}. Unfortunately the spine split times in non-binary cases do not have a joint density with respect to Lebesgue measure: for any $j=2,\ldots,k-1$, there is a positive probability that $\psi_j = \psi_{j-1}$. However we show that this probability tends to zero as $T\to\infty$, and therefore will not have an effect on our final answer.

Recall that $n_t$ is the number of distinct spine particles at time $t$, and $\rho^i_t$ is the number of marks carried by spine $i$ at time $t$.

\begin{lem}\label{splitsdistinct}
For any $i=1,\ldots,k-1$ and $t\in(0,1)$,
\[\Q^{k,T}_T\Big(n_{\psi_1} = 2, \, \rho^1_{\psi_1} = i \,\Big|\, \frac{\psi_1}{T} = t\Big) \to \frac{1}{k-1}.\]
\end{lem}

This tells us two things: that with probability tending to $1$ we have exactly $2$ spines at the first spine split time; and that the number of marks following each of those spines is uniformly distributed on $1,\ldots,k-1$.

\begin{proof}
We work in the case $\mu\neq 0$; the case $\mu=0$ proceeds almost identically. From the definition of $\Q$,
\[\Q^{k,T}_T(n_{tT} =2, \, \rho^1_{tT}=i\, |\, \tau_\emptyset = tT,\, n_{tT} \ge 2) = \frac{\P_T[g_{k,T} \ind_{\{n_{tT} =2, \, \rho^1_{tT}=i\}} | \tau_\emptyset = tT]}{\P_T[g_{k,T} \ind_{\{n_{tT}\ge 2\}} | \tau_\emptyset = tT]}.\]
Let $P_T(j;b;a_1,\ldots,a_b)$ be the probability that at time $\tau_\emptyset$, $j$ children are born, $b$ of which are spines, carrying $a_1,\ldots,a_b$ marks. Then
\[\P_T[g_{k,T} \ind_{\{n_{tT}= b,\, \rho^1_{tT}=a_1\}} \,|\, \tau_\emptyset = tT] = \sum_{j=b}^\infty \sum_{a_2,\ldots,a_b} P_T(j;b;a_1,\ldots,a_b) j^k \prod_{i=1}^b \P_T[g_{a_i,T(1-t)}]\]
where the sum over $a_2,\ldots,a_b$ runs over $1,\ldots,k$ such that $a_1+\ldots+a_b=k$.
Now
\[P_T(j;b;a_1,\ldots,a_b) = p_j^{(T)} \binom{j}{b} \frac{k!}{a_1!\ldots a_b!} \frac{1}{j^k}\]
and from Lemma \ref{scaledPmoments}, in the case $\mu\neq 0$,
\[\P_T[N_{T(1-t)}^{(a_i)}] = T^{a_i-1} \Big( \frac{\sigma^2}{2\mu} \Big)^{a_i-1} a_i! e^{r\mu (1-t)} ( e^{r\mu (1-t)} - 1 )^{a_i-1} + o(T^{a_i-1}).\]
This gives us
\begin{multline*}
\P_T[g_{k,T} \ind_{\{n_{tT}=b,\, \rho^1_{tT}=a_1\}}\, |\, \tau_\emptyset = tT] \\
= \sum_{j=b}^\infty  \sum_{a_2,\ldots,a_b} p_j^{(T)} \binom{j}{b} k! T^{k-b}\Big(\frac{\sigma^2}{2\mu}\Big)^{k-b} e^{br\mu(1-t)}(e^{r\mu(1-t)}-1)^{k-b}(1+o(1)).
\end{multline*}
If $b=2$, then fixing $a_1=i$ also fixes $a_2$ since $a_2=k-a_1$, so the second sum disappears and we are left with
\begin{multline}
\P_T[g_{k,T} \ind_{\{n_{tT}= 2,\, \rho^1_{tT}=i\}} | \tau_\emptyset = tT] \\
= \sum_{j=2}^\infty p_j^{(T)} \binom{j}{2} k! T^{k-2}\Big(\frac{\sigma^2}{2\mu}\Big)^{k-2} e^{2r\mu(1-t)}(e^{r\mu(1-t)}-1)^{k-2}(1+o(1))\\
= \frac{\sigma^2}{2} k! T^{k-2}\Big(\frac{\sigma^2}{2\mu}\Big)^{k-2} e^{2r\mu(1-t)}(e^{r\mu(1-t)}-1)^{k-2}(1+o(1)).\label{l2case}
\end{multline}
Notice in particular that this does not depend on the value of $i$.

Next we bound the probability that there are at least three distinct spines at time $\psi_1$ by taking a sum over $a_1$ and then over $b\ge 3$. For each $b$, there are certainly at most $k^b$ possible values of $a_1,\ldots,a_b$ that sum to $k$. Thus we get
\[\P_T[g_{k,T} \ind_{\{n_{tT}\ge 3\}} \,|\, \tau_\emptyset = tT] \le \sum_{b=3}^\infty \P_T[L^{(b)}] \frac{k!}{b!} k^b T^{k-b}\Big(\frac{\sigma^2}{2\mu}\Big)^{k-b} e^{br\mu(1-t)}(e^{r\mu(1-t)}-1)^{k-b}(1+o(1)).\]
Recall that we have assumed \eqref{momcond} that $\P_T[L^{(b)}] = o(T^{b-2})$ for each $b\ge 3$, so
\begin{equation}\label{l3case}
\P_T[g_{k,T} \ind_{\{n_{tT}\ge 3\}} \,|\, \tau_\emptyset = tT] = o(T^{k-2}).
\end{equation}
Dividing \eqref{l3case} by \eqref{l2case}, we see that the probability that there are at least $3$ distinct spines at time $\psi_{1}$ tends to zero as $T\to\infty$; or equivalently, that the probability that there are exactly $2$ distinct spines tends to $1$. Then since the right-hand side of \eqref{l2case} does not depend on $i$, the distribution of $\rho_{\psi_1}$ must be asymptotically uniform.
\end{proof}

Combined with the symmetry lemma, the previous result tells us that with high probability the spine split times are distinct. We want to use this to show that away from $0$, the rescaled split times $\psi_1/T,\ldots,\psi_{k-1}/T$ have an asymptotic density. First we need a preparatory lemma, which will be helpful in describing the topology of our limiting tree as well as calculating the asymptotic density of the split times.

\begin{lem}\label{splitderiv}
For any $s\in(0,1]$ and $t\in(0,s)$,
\[\Q^{k,sT}_T\Big(\frac{\psi_1}{T}>t\Big) \to \Big(\frac{e^{r\mu(s-t)}-1}{e^{r\mu s}-1}\Big)^{k-1}\]
and
\[-\frac{\d}{\d t} \Q^{k,sT}_T\Big(\frac{\psi_1}{T} > t\Big) \to  (k-1)r\mu \frac{(e^{r\mu(s-t)}-1)^{k-2}}{(e^{r\mu s}-1)^{k-1}} e^{r\mu(s-t)}\]
as $T\to\infty$.
\end{lem}

\begin{proof}
The first part of the proof follows easily by combining Lemmas \ref{Qbots} and \ref{scaledPmoments}. The second part is a more involved calculation. As in Lemma \ref{genGWmom}, we write $M_k(t) = \P_T[N_t^{(k)}]$. By Lemma \ref{Qbots},
\[\Q^{k,sT}_T(\psi_1 > tT) = \frac{\P[N^{(k)}_{T(s-t)}]}{\P[N^{(k)}_{sT}]} e^{(m_T-1)rtT} = \frac{M_k(T(s-t))}{M_k(sT)} e^{(m_T-1)rtT},\]
so
\begin{align*}
-\frac{\d}{\d t}\Q^{k,sT}_T(\psi_1 > tT) &= T\frac{M_k'(T(s-t))}{M_k(sT)} e^{(m_T-1)rtT} - T(m_T-1)r \frac{M_k(T(s-t))}{M_k(sT)} e^{(m_T-1)rtT}\\
&= \frac{T}{M_k(sT)}e^{(m_T-1)rtT} \big(M_k'(T(s-t)) - (m_T-1)r M_k(T(s-t))\big).
\end{align*}
Applying Lemma \ref{genGWmom}, this equals
\[\frac{T}{M_k(sT)}e^{(m_T-1)rtT} \bigg( (k-1)r(m_T-1) M_k(T(s-t)) + r\sum_{j=2}^k \binom{k}{j}\P_T[L^{(j)}]M_{k+1-j}(T(s-t))\bigg).\]
We now use Lemma \ref{scaledPmoments}. Since $\P_T[L^{(j)}] = o(T^{j-2})$ for all $j\ge 3$ (see \eqref{momcond}), the terms with $j\ge 3$ in the sum above do not contribute in the limit. We obtain
\begin{multline*}
\frac{Te^{r\mu t}}{(\frac{\sigma^2}{2\mu})^{k-1}k! e^{r\mu s}(e^{r\mu s}-1)^{k-1}T^{k-1}}\bigg[(k-1)r\mu\Big(\frac{\sigma^2}{2\mu}\Big)^{k-1}k!e^{r\mu(s-t)}(e^{r\mu(s-t)}-1)^{k-1}T^{k-2}\\ + r\frac{k(k-1)}{2}\sigma^2 \Big(\frac{\sigma^2}{2\mu}\Big)^{k-2}(k-1)!e^{r\mu(s-t)}(e^{r\mu(s-t)}-1)^{k-2}T^{k-2} + o(T^{k-2})\bigg].
\end{multline*}
Simplifying, this equals
\[\frac{1}{(e^{r\mu s}-1)^{k-1}}\bigg[(k-1)r\mu (e^{r\mu(s-t)}-1)^{k-1} + (k-1)r\mu (e^{r\mu(s-t)}-1)^{k-2} + o(1)\bigg],\]
so simplifying again we get
\[-\frac{\d}{\d t}\Q^{k,sT}_T(\psi_1 > tT) \to (k-1)r\mu \frac{(e^{r\mu(s-t)}-1)^{k-2}}{(e^{r\mu s}-1)^{k-1}} e^{r\mu(s-t)}.\qedhere\]
\end{proof}

Recall that $\mathcal H'$ is the $\sigma$-algebra containing topological information about which marks are following which spines, without information about the spine split times.

\begin{prop}\label{almostdensity}
The spine split times $\psi_1,\ldots,\psi_{k-1}$ are asymptotically independent of $\mathcal H'$ under $\Q^{k,T}_T$, and for any $0<s_1<t_1\le s_2<t_2\le \ldots \le s_{k-1}<t_{k-1}< 1$,
\[\lim_{T\to\infty} \hspace{-1mm} \Q^{k,T}_T\Big(\frac{\psi_1}{T}\in (s_1,t_1],\ldots, \frac{\psi_{k-1}}{T}\in (s_{k-1},t_{k-1}]\Big) \hspace{-0.5mm} = \hspace{-1mm} \int_{s_1}^{t_1}\hspace{-1.5mm}\cdots \int_{s_{k-1}}^{t_{k-1}} f_k(s_1',\ldots,s_{k-1}')\, \d s_{k-1}'\ldots\d s_1',\]
where
\[f_k(s_1,\ldots,s_{k-1}) = (k-1)! \Big(\frac{r\mu}{e^{r\mu}-1}\Big)^{k-1} \prod_{i=1}^{k-1} e^{r\mu(1-s_i)} \hspace{5mm} \hbox{ if } \mu\neq 0\]
and
\[f_k(s_1,\ldots,s_{k-1}) = (k-1)! \hspace{5mm} \hbox{ if } \mu=0.\]
\end{prop}

\begin{proof}
This is a generalization of the proof of Lemma \ref{Qsplitdist}, and the reader may wish to compare the two. The main difference is that now there is a chance that spine splitting events result in more than one new spine particle (since branching events need not be binary), and therefore we need to take care to ensure that the split times $\psi_1,\ldots,\psi_{k-1}$ are distinct.

With this in mind, let $\Upsilon_j$ be the event that the first $j$ spine split times are distinct,
\[\Upsilon_j = \{\psi_i \neq \psi_{i-1} \,\,\forall i=2,\ldots, j\}.\]
We work by induction; fix $j\le k-1$, $T>0$, $0<s_1<\ldots<s_{j-1}<1$. Then for $s\ge s_{j-1}$,
\begin{align*}
&\Q\Big( \frac{\psi_j}{T} > s \,\Big|\, \frac{\psi_{j-1}}{T} = s_{j-1},\ldots, \frac{\psi_1}{T} = s_1\Big)\\
&= \Q\Big( \Upsilon_j,\, \frac{\psi_j}{T} > s \,\Big|\, \frac{\psi_{j-1}}{T} = s_{j-1},\ldots, \frac{\psi_1}{T} = s_1\Big)\\
&= \Q\Big( \frac{\psi_j}{T} > s \,\Big|\,\Upsilon_j,\, \frac{\psi_{j-1}}{T} = s_{j-1},\ldots, \frac{\psi_1}{T} = s_1\Big)\Q\Big( \Upsilon_j \,\Big|\, \frac{\psi_{j-1}}{T} = s_{j-1},\ldots, \frac{\psi_1}{T} = s_1\Big).
\end{align*}
By Lemma \ref{splitsdistinct} and the symmetry lemma,
\[\Q\Big( \Upsilon_j \,\Big|\, \frac{\psi_{j-1}}{T} = s_{j-1},\ldots, \frac{\psi_1}{T} = s_1\Big) \to 1\]
for all $0<s_1<\ldots<s_{j-1}<1$. We also set
\[D(s) = - \frac{\d}{\d s} \Q\Big( \frac{\psi_j}{T} > s \,\Big|\,\Upsilon_j,\, \frac{\psi_{j-1}}{T} = s_{j-1},\ldots, \frac{\psi_1}{T} = s_1\Big)\]
and claim that
\[D(s) = (k-j) r\mu e^{r\mu(1-s_j)} \frac{(e^{r\mu(1-s_j)}-1)^{k-j-1}}{(e^{r\mu(1-s_{j-1})}-1)^{k-j}} + o(1).\]
If this claim holds, then applying induction and taking a product over $j$ gives the result. In particular, since this does not depend on the number of marks following each spine, the split times are asymptotically independent of $\mathcal H'$.

To prove the claim, fix $a_1,\ldots,a_j$ such that $a_i\in\{1,\ldots,k\}$ for each $i$ and $a_1+\ldots+a_j = k$. Let $A_j$ be the event that after time $\psi_{j-1}$, we have $j$ distinct spine particles carrying $a_1,\ldots,a_j$ marks. Then by the symmetry lemma (letting $s_0=0$),
\[\Q^{k,T}_T\Big(\frac{\psi_j}{T} > s_j \,\Big|\,\Upsilon_j,\,A_j,\, \frac{\psi_{j-1}}{T} = s_{j-1}\Big) = \prod_{i=1}^j \Q^{a_i,T(1-s_{j-1})}_T(\psi_1/T > s_j-s_{j-1}).\]
Thus, differentiating, we have
\[D(s)= -\hspace{-1.5mm}\sum_{a_1,\ldots,a_j}\hspace{-1.5mm} P_{a_1,\ldots,a_j} \sum_{l=1}^j \Big(\frac{\d}{\d s} \Q^{a_i,T(1-s_{j-1})}_T(\tfrac{\psi_1}{T} > s-s_{j-1})\Big) \prod_{i\neq l} \Q^{a_i,T(1-s_{j-1})}_T(\tfrac{\psi_1}{T} > s-s_{j-1})\]
where $P_{a_1,\ldots,a_j}$ is the probability that $A_j$ occurs. Applying Lemma \ref{splitderiv} then establishes the claim and completes the proof.
\end{proof}

We recall now the partition $Z_0,Z_1,\ldots$ which contained the information about the marks following each of the distinct spine particles, without the information about the split times.

\begin{lem}\label{Qtopology}
The partition $Z_0,Z_1,\ldots$ has the following distribution under $\Q^{k,T}_T$:
\begin{itemize}
\item If $Z_i$ consists of $i+1$ blocks of sizes $a_1,\ldots,a_{i+1}$, then the $j$th block will split next with probability $\frac{a_j-1}{k-i-1}(1+o(1))$ for each $j=1,\ldots,i+1$.
\item When a block of size $a$ splits, it splits into two new blocks with probability $1+o(1)$, and the probability that these blocks have sizes $l$ and $a-l$ is $\frac{1}{a-1}(1+o(1))$ for each $l=1,\ldots,a-1$.
\end{itemize}
\end{lem}

\begin{proof}
Suppose that we are given $\psi_{i}=sT$. For the first part, by the symmetry lemma, the probability that the $j$th block splits next is
\begin{multline*}
\int_0^{T(1-s)} \Q^{a_j,T(1-s)}_T\Big(\frac{\psi_1}{T}\in\d t\Big) \prod_{l\neq j} \Q^{a_l,T(1-s)}_T\Big(\frac{\psi_1}{T} > t\Big)\\
= \int_0^{T(1-s)}\Big(-\frac{\d}{\d t} \Q^{a_j,T(1-s)}_T\Big(\frac{\psi_1}{T}>t\Big)\Big) \prod_{l\neq j} \Q^{a_l,T(1-s)}_T\Big(\frac{\psi_1}{T} > t\Big)\,\d t.
\end{multline*}
By Lemma \ref{splitderiv}, this converges as $T\to\infty$ to
\[(a_j-1)r\mu\int_0^{T(1-s)} e^{r\mu (1-s-t)} \frac{e^{(r\mu(1-s-t)}-1)^{k-i}}{e^{(r\mu(1-s)}-1)^{k-i-1}} \d t.\]
Since the integrand does not depend on $a_j$, and we know the sum of the above quantity over $j=1,\ldots,i+1$ must converge to $1$ (since one of the blocks must split first), we get
\[r\mu\int_0^{T(1-s)} e^{r\mu (1-s-t)} \frac{e^{(r\mu(1-s-t)}-1)^{k-b-1}}{e^{(r\mu(1-s)}-1)^{k-b}} \d t \to \frac{1}{k-i-1}\]
and therefore the probability that the $j$th block splits next converges to $\frac{a_j-1}{k-i-1}$ as claimed.

The second part follows immediately from Lemma \ref{splitsdistinct}.
\end{proof}

\subsection{Asymptotics for $N_T$ under $\Q^{k,T}_T$}

We now apply our asymptotics for $u_T(F_T(e^{-z},sT))$ to approximate the distribution of $N_T$ when the split times are known.

\begin{lem}\label{lemmgfconv}
For any $\phi\ge 0$ and $0\le s_1\le \ldots\le s_{k-1}\le 1$,
\[\Q^{k,T}_T\Big[e^{-\phi\tilde N_T/T}\,\Big|\,\G^k_T,\,\frac{\psi_1}{T} = s_1, \ldots, \frac{\psi_{k-1}}{T} = s_k\Big] \to \begin{cases} \displaystyle\prod_{i=0}^{k-1} \Big(1+\frac{\sigma^2}{2\mu}\phi(e^{r\mu(1-s_i)}-1)\Big)^{-2} & \hbox{ if } \mu \neq 0\\[4mm]
\displaystyle \prod_{i=0}^{k-1} \Big(1+\frac{r\sigma^2}{2}\phi(1-s_i)\Big)^{-2} & \hbox{ if } \mu = 0\end{cases}\]
almost surely as $T\to\infty$.
\end{lem}

\begin{proof}
From Proposition \ref{Qmgf_gen_prop} we know that
\begin{multline*}
\Q^{k,T}_T\hspace{-0.4mm}\Big[e^{-\phi\tilde N_T/T}\hspace{0.2mm}\Big|\hspace{0.2mm}\G^k_T,\,\frac{\psi_1}{T} = s_1, \ldots, \frac{\psi_{k-1}}{T} = s_k\Big] \\
= \prod_{i=0}^{k-1}\Big(e^{-r(m_T-1)T(1-s_i)}\frac{u_T(F_T(e^{-\phi/T},T(1-s_i)))}{u_T(e^{-\phi/T})}\Big).
\end{multline*}
Of course $(m_T-1)T\to\mu$, and Lemma \ref{fTconv} tells us that
\[T^2 u_T(F_T(e^{-\phi/T},T(1-s_i))) \to -\mu f(\phi,1-s_i) + \frac{\sigma^2}{2}f(\phi,1-s_i)^2\]
where
\[f(\phi,s) = \frac{ \phi e^{ \mu r s }}{ 1 + \frac{\sigma^2}{2\mu}  \phi ( e^{ \mu r s } - 1 )} \hspace{2mm} \hbox{ if }\mu\neq 0 \hspace{8mm} \hbox{ or } \hspace{8mm} f(\phi,s) = \frac{ \phi }{ 1 + \frac{r\sigma^2}{2}  \phi s} \hspace{2mm} \hbox{ if }\mu= 0.\]
Noting that $u_T(e^{-\phi/T}) = u_T(F_T(e^{-\phi/T},0))$, we see that
\[e^{-r(m_T-1)T(1-s_i)}\frac{u_T(F_T(e^{-\phi/T},T(1-s_i)))}{u_T(e^{-\phi/T})} \longrightarrow e^{-r\mu(1-s_i)}\frac{-\mu f(\phi,1-s_i) + \frac{\sigma^2}{2}f(\phi,1-s_i)^2}{-\mu f(\phi,0) + \frac{\sigma^2}{2}f(\phi,0)^2}.\]
Now, in the case $\mu\neq 0$, we simply write out
\begin{align*}
-\mu f(\phi,1-s_i) + \frac{\sigma^2}{2}f(\phi,1-s_i)^2 &= \frac{-\mu\phi e^{r\mu(1-s_i)}(1 + \frac{\sigma^2}{2\mu}  \phi ( e^{ \mu r (1-s_i) } - 1 )) + \frac{\sigma^2}{2}\phi^2 e^{2r\mu(1-s_i)}}{(1 + \frac{\sigma^2}{2\mu}  \phi ( e^{ \mu r (1-s_i) } - 1 ))^2}\\
&= \frac{-\mu\phi e^{r\mu(1-s_i)} + \frac{\sigma^2}{2}\phi^2 e^{r\mu(1-s_i)}}{(1 + \frac{\sigma^2}{2\mu}  \phi ( e^{ \mu r (1-s_i) } - 1 ))^2},
\end{align*}
so since $-\mu f(\phi,0) + \frac{\sigma^2}{2}f(\phi,0)^2 = -\mu \phi + \sigma^2\phi^2/2$, we have
\[e^{-r\mu(1-s_i)}\frac{-\mu f(\phi,1-s_i) + \frac{\sigma^2}{2}f(\phi,1-s_i)^2}{-\mu f(\phi,0) + \frac{\sigma^2}{2}f(\phi,0)^2} = \Big(1 + \frac{\sigma^2}{2\mu}  \phi ( e^{ \mu r (1-s_i) } - 1 )\Big)^{-2}.\]
The result in the case $\mu = 0$ is very similar.
\end{proof}

\begin{lem}\label{noresidue}
For any $\phi\ge 0$,
\[\Q^{k,T}_T[e^{-\phi (N_T-k)/T} | \G^k_T] = \Q^{k,T}_T[e^{-\phi \tilde N_T / T} | \G^k_T](1+o(1))\]
$\Q^{k,T}_T$-almost surely.
\end{lem}

\begin{proof}
Recall that $\tilde N_T$ is the number of ordinary particles alive at time $T$, and there are ($\Q$-almost surely) $k$ spines at time $T$. All other particles are residue particles. Given $\G^k_T$, the number of residue particles is independent of the number of ordinary particles; therefore it suffices to show that
\[\Q^{k,T}[e^{-\phi (N_T - k - \tilde N_T)/T} | \G^k_T] \to 1.\]

Recall that we assumed that there exists a deterministic function $J(T)=o(T)$ such that our offspring distribution satisfies $\P_T(L=j)=0$ for all $j\ge J(T)$. Since $\Q^{k,T}$ is absolutely continuous with respect to $\P_T$, we also have $\Q^{k,T}(L=j)=0$ for all $j\ge J(T)$.

Since non-spine particles behave exactly as under $\P_T$, the number of descendants at time $T$ of any one particle born at time $\psi_i$ is $\P_T[e^{-zN_{T-s}}]|_{s=\psi_i}$. Therefore
\[\Q^{k,T}_T[e^{-\phi(N_T - k - \tilde N_T)/T}|\G^k_T] \ge \prod_{i=1}^{k-1} \P_T[e^{-\phi N_{T-s}/T}]^{J(T)}\Big|_{s=\psi_i}.\]
By Jensen's inequality, for any $t\in[0,T]$,
\[\P_T[e^{-\phi N_t/T}] \ge \exp(-\phi\P_T[N_t]/T) \ge \exp(-\phi e^{r(m_T-1)T}/T),\]
and thus
\[\Q^{k,T}_T[e^{-\phi(N_T - k - \tilde N_T)/T}|\G^k_T] \ge \P_T[\exp(-\phi e^{r(m_T-1)T} J(T) / T)]^{k-1}.\]
Since $J(T)=o(T)$, the right-hand side converges to $1$ as $T\to\infty$, and of course
\[\Q^{k,T}_T[e^{-\phi(N_T - k - \tilde N_T)/T}|\G^k_T]\le 1,\]
so we are done.
\end{proof}

Recall that $\Upsilon_{k-1}$ is the event that all the split times are distinct, and $\mathcal H'$ is the $\sigma$-algebra that contains topological information about which marks follow which spines without information about the spine split times. Let $(\tilde \psi_1,\ldots,\tilde \psi_{k-1})$ be a uniform random permutation of $(\psi_1,\ldots,\psi_{k-1})$. We combine several of our results to prove the following.

\begin{lem}\label{combineasymp}
Fix $s_1,\ldots,s_{k-1}\in(0,1)$. Let
\[f(\xi_T) = \ind_{\{\tilde \psi_1/T > s_1,\ldots, \tilde\psi_{k-1}/T > s_{k-1},\Upsilon_{k-1}\}\cap H}\]
where $H\in\mathcal H'$. There exists a constant $h$ such that $\Q^{k,T}_T(H)\to h$ as $T\to\infty$. For any $\phi\ge0$, if $\mu\neq 0$ then
\[\lim_{T\to\infty} \Q^{k,T}_T[e^{-\phi(N_T-k)/T}f(\xi_T)] = \Big(\frac{1}{e^{r\mu}-1}\Big)^{k-1} \frac{h}{(1+\frac{\sigma^2}{2\mu}\phi(e^{r\mu}-1))^2} \prod_{i=1}^{k-1} \frac{e^{r\mu(1-s_i)}-1}{1+\frac{\sigma^2}{2\mu}\phi (e^{r\mu(1-s_i)}-1)}\]
and if $\mu=0$ then
\[\lim_{T\to\infty} \Q^{k,T}_T[e^{-\phi(N_T-k)/T}f(\xi_T)] = \frac{h}{(1+r\sigma^2\phi/2)^2} \prod_{i=1}^{k-1} \frac{1-s_i}{1+r\sigma^2\phi (1-s_i)/2}.\]
\end{lem}

\begin{proof}
The fact that $\Q^{k,T}_T(H)$ converges follows from Lemma \ref{Qtopology}. Now, by Proposition \ref{almostdensity} and Lemma \ref{unordering}, in the case $\mu\neq 0$,
\begin{multline*}
\Q^{k,T}_T[e^{-\phi(N_T-k)/T}f(\xi_T)]\\
\hspace{-30mm}= (1+o(1))\int_{s_1}^1\cdots \int_{s_{k-1}}^1 \Big(\frac{r\mu}{e^{r\mu}-1}\Big)^{k-1} \bigg(\prod_{i=1}^{k-1} e^{r\mu(1-s_i')}\bigg)\\
\hspace{20mm}\cdot\Q^{k,T}_T\Big[\ind_H \Q^{k,T}_T\Big[e^{-\phi(N_T-k)/T}\,\Big|\,\G^k_T,\, \frac{\tilde\psi_1}{T}=s_1',\ldots,\frac{\tilde\psi_1}{T}=s_{k-1}'\Big]\Big].
\end{multline*}
By Lemma \ref{noresidue}, we may replace $N_T-k$ with $\tilde N_T$; and then by Lemma \ref{lemmgfconv}, the above equals
\begin{multline*}
(1+o(1)) \int_{s_1}^1\cdots \int_{s_{k-1}}^1 \Big(\frac{r\mu}{e^{r\mu}-1}\Big)^{k-1}\bigg(\prod_{i=1}^{k-1} e^{r\mu(1-s_i')}\bigg)\\
\cdot \Q^{k,T}_T(H)\prod_{j=0}^{k-1}\Big(1+\frac{\sigma^2}{2\mu}\phi(e^{r\mu(1-s_j')}-1)\Big)^{-2}\,\d s_{k-1}'\ldots\d s_1'
\end{multline*}
almost surely. After some small rearrangements this becomes
\[(1+o(1)) \Big(\frac{r\mu}{e^{r\mu}-1}\Big)^{k-1} \frac{h}{(1+\frac{\sigma^2}{2\mu}\phi(e^{r\mu}-1))^2}\prod_{i=1}^{k-1} \int_{s_i}^1 \frac{e^{r\mu(1-s_i')}}{(1+\frac{\sigma^2}{2\mu}\phi(e^{r\mu(1-s_i')}-1))^2}\,\d s_i',\]
and then applying the second part of Lemma \ref{integrateout} gives the result. The case $\mu=0$ is similar.
\end{proof}

\subsection{The final steps in the proof of Theorem \ref{nearcritthm}}

\begin{proof}[Proof of Theorem \ref{nearcritthm}]
By Proposition \ref{firstprop}, for any measurable $f$,
\[\P_T\bigg[\frac{1}{N_T^{(k)}} \sum_{u\in \Nc_T^{(k)}} f(u) \, \bigg| \, N_T \ge k\bigg] = \frac{\P_T[N_T^{(k)}]}{\P_T(N_T\ge k)(k-1)!} \int_0^\infty (e^z-1)^{k-1} \Q^{k,T}_T\Big[e^{-zN_T} f(\xi_T)\Big]\,\d z.\]
Substituting $z = \phi/T$ and rearranging, we get
\[\frac{1}{(k-1)!}\frac{\P_T[N_T^{(k)}]}{T^{k-1}}\frac{1}{T\P_T(N_T\ge k)}\int_0^\infty (T(1-e^{-\phi/T}))^{k-1}\Q^{k,T}_T[e^{-\phi(N_T-k)/T} f(\xi_T)]\, \d\phi.\]
By Lemma \ref{scaledPmoments},
\[\frac{\P_T[N_T^{(k)}]}{T^{k-1}} \to \Big(\frac{\sigma^2}{2\mu}\Big)^{k-1} k! e^{r\mu}(e^{r\mu}-1)^{k-1} \hspace{2mm} \hbox{ if } \mu\neq 0 \hspace{4mm} \hbox{ and } \hspace{4mm}  \frac{\P_T[N_T^{(k)}]}{T^{k-1}} \to \Big(\frac{r\sigma^2}{2}\Big)^{k-1} k! \hspace{2mm} \hbox{ if } \mu= 0,\]
and by Lemma \ref{scaledPext},
\[T\P_T(N_T\ge k) \to \frac{2\mu e^{r\mu}}{\sigma^2(e^{r\mu}-1)} \hspace{2mm} \hbox{ if } \mu\neq 0 \hspace{4mm} \hbox{ and } \hspace{4mm} T\P_T(N_T\ge k) \to \frac{2}{r\sigma^2} \hspace{2mm} \hbox{ if } \mu\neq 0.\]
Therefore
\[\frac{1}{(k-1)!}\frac{\P_T[N_T^{(k)}]}{T^{k-1}}\frac{1}{T\P_T(N_T\ge k)} \to k\Big(\frac{\sigma^2}{2\mu}\Big)^k(e^{r\mu}-1)^k\hspace{5mm} \hbox{ if } \mu\neq 0\]
and
\[\frac{1}{(k-1)!}\frac{\P_T[N_T^{(k)}]}{T^{k-1}}\frac{1}{T\P_T(N_T\ge k)} \to k\Big(\frac{r\sigma^2}{2}\Big)^k\hspace{5mm} \hbox{ if } \mu= 0.\]
We deduce that
\begin{multline}\label{keyfreln}
\P_T\bigg[\frac{1}{N_T^{(k)}} \sum_{u\in \Nc_T^{(k)}} f(u) \, \bigg| \, N_T \ge k\bigg]\\
= (1+o(1))k\Big(\frac{\sigma^2}{2\mu}\Big)^k(e^{r\mu}-1)^k \int_0^\infty (T(1-e^{-\phi/T}))^{k-1}\Q^{k,T}_T[e^{-\phi(N_T-k)/T} f(\xi_T)]\, \d\phi
\end{multline}
when $\mu\neq 0$, and when $\mu=0$
\begin{multline*}
\P_T\bigg[\frac{1}{N_T^{(k)}} \sum_{u\in \Nc_T^{(k)}} f(u) \, \bigg| \, N_T \ge k\bigg]\\
= (1+o(1))k\Big(\frac{r\sigma^2}{2}\Big)^k \int_0^\infty (T(1-e^{-\phi/T}))^{k-1}\Q^{k,T}_T[e^{-\phi(N_T-k)/T} f(\xi_T)]\, \d\phi.
\end{multline*}

Our aim now is to choose $f$ as in Lemma \ref{combineasymp}, and apply dominated convergence and Lemma \ref{combineasymp} to complete the proof. We do this only in the case $\mu\neq 0$; the case $\mu=0$ is very similar. Let
\[A(\phi,T) = (T(1-e^{-\phi/T}))^{k-1} \Q^{k,T}_T[e^{-\phi(N_T-k)/T}f(\xi_T)]\]
and
\[B(\phi,T) = (T(1-e^{-\phi/T}))^{k-1} \Q^{k,T}_T[e^{-\phi(N_T-k)/T}].\]
Then $0\le A(\phi,T)\le B(\phi,T)$ for all $\phi$, $T$. By letting $s_1,\ldots,s_{k-1}\downarrow 0$ in Lemma \ref{combineasymp}, we get that
\begin{align*}
\lim_{T\to\infty} \Q^{k,T}_T[e^{-\phi(N_T-k)/T}\ind_{\Upsilon_{k-1}}] &= \Big(\frac{1}{e^{r\mu}-1}\Big)^{k-1} \frac{1}{(1+\frac{\sigma^2}{2\mu}\phi(e^{r\mu}-1))^2} \Big(\frac{e^{r\mu}-1}{1+\frac{\sigma^2}{2\mu}\phi(e^{r\mu}-1)}\Big)^{k-1}\\
&= \frac{1}{(1+\frac{\sigma^2}{2\mu}\phi(e^{r\mu}-1))^{k+1}}.
\end{align*}
Also, by Lemma \ref{splitsdistinct},
\[\Q^{k,T}_T[e^{-\phi(N_T-k)/T}\ind_{\Upsilon_{k-1}^c}] \le \Q^{k,T}_T(\Upsilon_{k-1}^c) \to 0,\]
so
\[\lim_{T\to\infty} B(\phi,T) = \phi^{k-1} \frac{1}{(1+\frac{\sigma^2}{2\mu}\phi(e^{r\mu}-1))^{k+1}}.\]
On the other hand, by \eqref{keyfreln} with $f\equiv 1$,
\[1 = \P_T\bigg[\frac{1}{N_T^{(k)}} \sum_{u\in \Nc_T^{(k)}} 1 \, \bigg| \, N_T \ge k\bigg] = (1+o(1))k\Big(\frac{\sigma^2}{2\mu}\Big)^k(e^{r\mu}-1)^k \int_0^\infty B(\phi,T)\, \d\phi,\]
so
\[\lim_{T\to\infty}\int_0^\infty B(\phi,T)\, \d\phi = \frac{1}{k}\Big(\frac{2\mu}{\sigma^2(e^{r\mu}-1)}\Big)^k;\]
and as a result we see that
\[\lim_{T\to\infty}\int_0^\infty B(\phi,T)\, \d\phi = \int_0^\infty \lim_{T\to\infty} B(\phi,T)\,\d\phi.\]
Therefore, by dominated convergence,
\begin{equation}\label{domcon}
\lim_{T\to\infty}\int_0^\infty A(\phi,T)\, \d\phi = \int_0^\infty \lim_{T\to\infty} A(\phi,T)\,\d\phi.
\end{equation}

Lemma \ref{combineasymp} tells us that
\[A(\phi,T)\to \Big(\frac{\phi}{e^{r\mu}-1}\Big)^{k-1}\frac{h}{(1+\frac{\sigma^2}{2\mu}\phi(e^{r\mu}-1))^2} \prod_{i=1}^{k-1} \frac{e^{r\mu(1-s_i)}-1}{1+\frac{\sigma^2}{2\mu}\phi(e^{r\mu(1-s_i)}-1)}\]
where $h=\lim_{T\to\infty} \Q^{k,T}_T(H)$, so by \eqref{keyfreln} and \eqref{domcon},
\begin{align*}
&\lim_{T\to\infty} \P_T\bigg[\frac{1}{N_T^{(k)}} \sum_{u\in \Nc_T^{(k)}} f(u) \, \bigg| \, N_T \ge k\bigg]\\
&\hspace{15mm} = k\Big(\frac{\sigma^2}{2\mu}\Big)^k(e^{r\mu}-1) \int_0^\infty \phi^{k-1} \frac{h}{(1+\frac{\sigma^2}{2\mu}\phi(e^{r\mu}-1))^2} \prod_{i=1}^{k-1} \frac{e^{r\mu(1-s_i)}-1}{1+\frac{\sigma^2}{2\mu}\phi(e^{r\mu(1-s_i)}-1)}\, \d\phi\\
&\hspace{15mm} = \frac{k\sigma^2}{2\mu}(e^{r\mu}-1) \int_0^\infty \frac{h}{(1+\frac{\sigma^2}{2\mu}\phi(e^{r\mu}-1))^2} \prod_{i=1}^{k-1} \Big(1-\frac{1}{1+\frac{\sigma^2}{2\mu}\phi(e^{r\mu(1-s_i)}-1)}\Big)\, \d\phi.
\end{align*}
Note that, for any $\mu\neq 0$, we have $\frac{\sigma^2}{2\mu}(e^{r\mu(1-s_i)}-1)>0$ for all $i$, so we can apply the first part of Lemma \ref{partialfrac2} to get
\begin{multline*}
\lim_{T\to\infty} \P_T\bigg[\frac{1}{N_T^{(k)}} \sum_{u\in \Nc_T^{(k)}} f(u) \, \bigg| \, N_T \ge k\bigg]\\
= hk\bigg(\prod_{i=1}^{k-1} \frac{e_i}{e_i-e_0}\bigg) + hke_0\sum_{j=1}^{k-1} \frac{e_j}{(e_j-e_0)^2}\bigg(\prod_{\substack{i=1\\ i\neq j}}^{k-1} \frac{e_i}{e_i-e_j}\bigg)\log \frac{e_0}{e_j}
\end{multline*}
where $e_j = \frac{\sigma^2}{2\mu}(e^{r\mu(1-s_j)}-1)$ for each $j$ (including $j=0$, where $s_0=0$).
\end{proof}

\subsection{Proof of construction of the scaling limit}

In this section we prove the results of Section \ref{limitconstructionsec}.

\begin{proof}[Proof of Theorem \ref{maxconcrit}]
Of course $\P(M_k\le \theta) = \P(X_1\le \theta)^k$, so $\P(M_k\in\d \theta) = k\P(X_1\in\d \theta) \P(X_1\le \theta)^{k-1}$. Thus
\begin{align*}
&\P(T_1\in\d s_1,\ldots, T_{k-1}\in\d s_{k-1})\\
&= \int_0^\infty \P(M_k\in \d \theta) \P(T_1\in\d s_1,\ldots, T_{k-1}\in\d s_{k-1} | M_k=\theta)\\
&= \int_0^\infty \hspace{-1mm}k\P(X_1\in\d \theta) \P(X_1\le \theta)^{k-1} \P(1\hspace{-0.5mm}-\hspace{-0.5mm}\tfrac{X_1}{\theta}\hspace{-0.5mm} \in\hspace{-0.5mm} \d s_1,\ldots,1\hspace{-0.5mm}-\hspace{-0.5mm}\tfrac{X_{k-1}}{\theta} \hspace{-0.5mm}\in\hspace{-0.5mm} \d s_{k-1} | X_1\le \theta,\ldots, X_{k-1}\le \theta)\\
&= \int_0^\infty \frac{k}{(1+\theta)^2} \P(X_1\le \theta)^{k-1} \prod_{i=1}^{k-1} \P\Big(1-\frac{X_i}{\theta}\in\d s_i \Big| X_i\le \theta\Big) \,\d \theta\\
&= \int_0^\infty \frac{k}{(1+\theta)^2} \prod_{i=1}^{k-1} \P\Big(1-\frac{X_i}{\theta}\in\d s_i \Big) \,\d \theta\\
&= \int_0^\infty \frac{k}{(1+\theta)^2} \bigg(\prod_{i=1}^{k-1} \frac{\theta}{(1+\theta(1-s_i))^2}\, \d s_i\bigg) \,\d \theta.
\end{align*}
This is exactly the density that we saw for $(\tilde \S^k_1,\ldots,\tilde\S^k_{k-1})$ at the start of Section \ref{chatsec}.

To see that our tree has the same topology as claimed, start by assigning $k$ marks to the top of the tallest line, i.e.~at the point $(U_I,1-T_I)$. Colour this line green. Next consider the second tallest line, which we colour blue; let its index be $J$. Since it is positioned uniformly on the horizontal axis, the number $L$ of shorter lines to its left is uniformly distributed on $\{0,\ldots,k-2\}$, and so is the number $k-2-L$ to its right. Suppose without loss of generality that the blue line is to the left of the green line, and assign $L+1$ marks to the top of the blue line, i.e.~at point $(U_J,1-T_J)$, and $k-(L+1)$ marks to the point $(U_I,1-T_J)$. (If the blue line were to the right of the green line, we would assign $k-(L+1)$ marks to $(U_J,1-T_J)$ and $L+1$ marks to $(U_I,1-T_J)$.) Thus the number of marks assigned to the top of the blue line is uniform on $\{1,\ldots,k-1\}$.

Moving downwards through our picture, the next horizontal line to appear will correspond to the third-tallest vertical line. We ask which of the two coloured lines this next horizontal line will join to, which corresponds to which of the branches in the tree will split next. By our construction, the event that the third tallest line joins to the blue line (given that the blue line is to the left of the green line) is exactly the event that the third tallest line is to the left of the blue line. Since the lengths of the branches are independent and identically distributed, this has probability $L/(k-2)$. Furthermore, observe that the position of the third tallest line, conditionally on it falling to the left of the blue line (respectively to the right), is uniformly distributed on $(0,U_J)$ (respectively $(U_J,1)$).

More generally, once we have seen the $n$ tallest vertical lines, and assigned $a_i$ marks to line $i$ for each line $i$ that we have seen, the $(n+1)$st tallest vertical line has probability $(a_i-1)/(k-n)$ of joining line $i$; and the number of marks this new line gets is uniformly distributed on $\{1,\ldots,a_i-1\}$. This corresponds exactly to the topology outlined in Theorem \ref{nearcritthm}.
\end{proof}

\begin{figure}[h!]
  \centering
   \includegraphics[height=11cm]{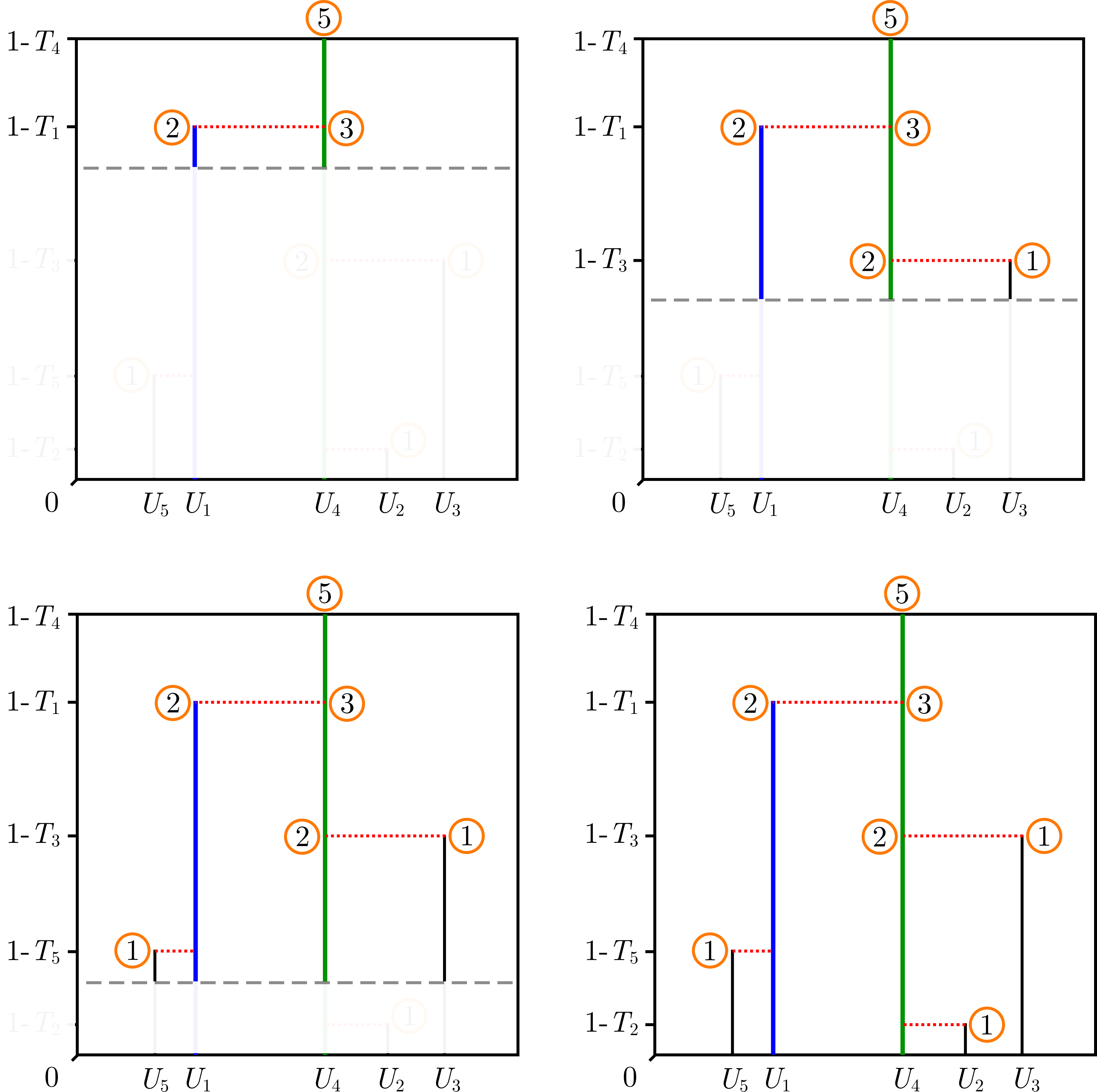}
   \vspace{0mm}
  \caption{\small{Constructing the tree by moving downwards through our picture. The number of marks are shown in circles. The as yet ``unseen'' parts of the tree are left blank. Here $k=5$, $I=4$ and $J=1$.}}\label{aldProoffig}
  \end{figure}

\begin{proof}[Proof of Theorem \ref{maxconnoncrit}]
Rather than doing the calculation directly, this follows from Theorem \ref{maxconcrit} by noting that making the substitution
\[t_i = \frac{e^{r\mu}-e^{r\mu(1-s_i)}}{e^{r\mu}-1}\]
in the density $f_k$ recovers the critical case from the non-critical.
\end{proof}

\appendix

\section{}

Here we gather some results that are easy but still require proofs. We begin with the calculation of some integrals.

\begin{lem}\label{partialfrac2}
Suppose that $k\ge 2$ and $e_0,\ldots,e_{k-1}\in(0,\infty)$ with $e_i\neq e_j$ for any $i\neq j$. Then
\begin{multline*}
\int_0^\infty \frac{1}{(1+\theta e_0)^2} \prod_{j=1}^{k-1} \Big(1-\frac{1}{1+\theta e_j}\Big)\d \theta\\
=\frac{1}{e_0}\bigg(\prod_{i=1}^{k-1} \frac{e_i}{e_i-e_0}\bigg) + \sum_{j=1}^{k-1} \frac{e_j}{(e_j-e_0)^2}\bigg(\prod_{\substack{i=1\\ i\neq j}}^{k-1} \frac{e_i}{e_i-e_j}\bigg)\log\Big(\frac{e_0}{e_j}\Big)
\end{multline*}
and
\begin{multline*}
\int_0^1 \frac{1}{(1+\theta e_0)^2} \prod_{j=1}^{k-1} \Big(1-\frac{1}{1+\theta e_j}\Big)\d \theta\\
= \frac{1}{1+e_0}\prod_{i=1}^{k-1}\frac{e_i}{e_i-e_0} + \sum_{j=1}^{k-1} \frac{e_j}{(e_j-e_0)^2}\bigg(\prod_{\substack{i=1\\ i\neq j}}^{k-1} \frac{e_i}{e_i-e_j}\bigg)\log\Big(\frac{1+e_0}{1+e_j}\Big).
\end{multline*}
\end{lem}

\begin{proof}
First note that since $e_j\in(0,\infty)$ for each $j$,
\[0\le \int_0^\infty \frac{1}{(1+\theta e_0)^2} \prod_{j=1}^{k-1} \Big(1-\frac{1}{1+\theta e_j}\Big)\d \theta \le \int_0^\infty \frac{1}{(1+\theta e_0)^2} \d \theta < \infty.\]

Expanding the product, we have 
\[\prod_{j=1}^{k-1} \Big(1-\frac{1}{1+\theta e_j}\Big) = 1 - \sum_{j_1=1}^{k-1}\frac{1}{1+\theta e_{j_1}} + \sum_{j_1<j_2}\prod_{i=1}^2\frac{1}{1+\theta e_{j_i}} + \ldots + \sum_{j_1<\ldots<j_{k-1}}\prod_{i=1}^{k-1}\frac{1}{1+\theta e_{j_i}}.\]
We view this as one sum in which all terms are products of factors of the form $\frac{1}{1+\theta e_i}$ for some $i$; therefore, using partial fractions, the whole thing can be written as a sum of terms of the form $\frac{c_i}{1+\theta e_i}$ for some coefficients $c_i$ which do not depend on $\theta$. As a result, our entire integrand may be written in the form
\begin{equation}\label{infpfrac}
\frac{1}{(1+\theta e_0)^2} \prod_{j=1}^{k-1} \Big(1-\frac{1}{1+\theta e_j}\Big) = \frac{a_1}{1+\theta e_0} + \frac{a_2}{(1+\theta e_0)^2} + \sum_{j=1}^{k-1} \frac{b_j}{1+\theta e_j}
\end{equation}
for some coefficients $a_1$, $a_2$ and $b_1,\ldots,b_{k-1}$ that do not depend on $\theta$.

Setting $\theta=-1/e_0$, we see that necessarily
\[a_2 = \prod_{i=1}^{k-1} \frac{e_i}{e_i-e_0}.\]
Setting $\theta = -1/e_j$ for $j=1,\ldots,k-1$, some elementary calculations reveal that
\[b_j = -\frac{e_j^2}{(e_j-e_0)^2}\prod_{\substack{i=1\\ i\neq j}}^{k-1} \frac{e_i}{e_i-e_j}.\]
For $a_1$, we observe that
\[\frac{\d}{\d\theta} \prod_{j=1}^{k-1} \Big(1-\frac{1}{1+\theta e_j}\Big) \to a_1 e_0 \hspace{3mm} \hbox{ as } \theta \to -1/e_0,\]
and then after some simple calculations we get
\[a_1 = e_0 \bigg(\prod_{i=1}^{k-1} \frac{e_i}{e_i-e_0}\bigg)\sum_{j=1}^{k-1} \frac{1}{e_j-e_0}.\]

Integrating \eqref{infpfrac},
\begin{multline*}
\int_0^\infty \frac{1}{(1+\theta e_0)^2} \prod_{j=1}^{k-1} \Big(1-\frac{1}{1+\theta e_j}\Big)\d \theta = \lim_{M\to\infty}\Big( \frac{a_1}{e_0}\log(1+ M e_0) + \frac{a_2}{e_0} + \sum_{j=1}^{k-1} \frac{b_j}{e_j}\log(1+M e_j)\Big)\\
= \frac{a_1}{e_0}\log e_0 + \frac{a_2}{e_0} + \sum_{j=1}^{k-1} \frac{b_j}{e_j}\log e_j + \lim_{M\to\infty}\Big(\frac{a_1}{e_0} + \sum_{j=1}^{k-1} \frac{b_j}{e_j}\Big)\log M.
\end{multline*}
Since we have already checked that the integral is finite, the last term must be zero; this leaves
\begin{align*}
\int_0^\infty \frac{1}{(1+\theta e_0)^2} \prod_{j=1}^{k-1} \Big(1-\frac{1}{1+\theta e_j}\Big)\d \theta &= \frac{a_1}{e_0}\log e_0 + \frac{a_2}{e_0} + \sum_{j=1}^{k-1} \frac{b_j}{e_j}\log e_j\\
&= -\sum_{j=1}^{k-1} \frac{b_j}{e_j}\log e_0 + \frac{a_2}{e_0} + \sum_{j=1}^{k-1} \frac{b_j}{e_j}\log e_j
\end{align*}
which is the first part of the result. The second part follows similarly by integrating \eqref{infpfrac} over $(0,1)$ instead of $(0,\infty)$.
\end{proof}

\begin{lem}\label{integrateout}
For any $0\le s_j \le T$, $\beta\neq \alpha$ and $y\in[0,1]$,
\[\int_{s_j}^{T} \frac{e^{(\beta-\alpha)(T-s)}}{(\beta(1-y)e^{(\beta-\alpha)(T-s)}+\beta y-\alpha)^2} \d s = \frac{e^{(\beta-\alpha)(T-s_j)}-1}{(\beta-\alpha)^2(\beta(1-y)e^{(\beta-\alpha)(T-s_j)}+ \beta y - \alpha)}.\]
Also, for any $0\le s_i \le 1$, $r,\sigma >0$ and $\mu\neq 0$,
\[\int_{s_i}^1 \frac{e^{r\mu (1-s)}}{(1+\frac{\sigma^2}{2\mu}\phi(e^{r\mu(1-s)}-1))^2} \d s = \frac{1}{r\mu}\Big(\frac{e^{r\mu(1-s_i)}-1}{1 + \frac{\sigma^2}{2\mu}\phi (e^{r\mu(1-s_i)}-1)}\Big).\]
\end{lem}

\begin{proof}
By substituting $t = e^{(\beta-\alpha)(T-s)}$, we see that
\begin{multline*}
\int_{s_j}^{T} \frac{e^{(\beta-\alpha)(T-s)}}{(\beta(1-y)e^{(\beta-\alpha)(T-s)}+\beta y-\alpha)^2} \d s = \frac{1}{\beta-\alpha}\int_1^{e^{(\beta-\alpha)(T-s_j)}} \frac{1}{(\beta(1-y)t + \beta y-\alpha)^2} \d t\\
= \frac{1}{(\beta-\alpha)\beta(1-y)}\Big(\frac{1}{\beta - \alpha} - \frac{1}{\beta(1-y)e^{(\beta-\alpha)(T-s_j)} + \beta y - \alpha}\Big).
\end{multline*}
Furthermore,
\begin{align*}
\frac{1}{\beta - \alpha} - \frac{1}{\beta(1-y)e^{(\beta-\alpha)(T-s_j)} + \beta y - \alpha} &= \frac{\beta(1-y)e^{(\beta-\alpha)(T-s_j)}+\beta y - \beta}{(\beta-\alpha)(\beta(1-y)e^{(\beta-\alpha)(T-s_j)} + \beta y - \alpha)}\\
&= \frac{\beta(1-y)(e^{(\beta-\alpha)(T-s_j)}-1)}{(\beta-\alpha)(\beta(1-y)e^{(\beta-\alpha)(T-s_j)} + \beta y - \alpha)}.
\end{align*}
Combining these two calculations gives the first part of the result. The second is very similar.
\end{proof}

The following lemma is elementary, but we do not know a suitable reference.

\begin{lem}\label{unordering}
Suppose that $X_1\le X_2\le \ldots \le X_n$ are ordered random variables satisfying
\[\P(X_1\in(a_1,b_1],\ldots,X_n\in(a_n,b_n]) = \int_{a_1}^{b_1}\cdots\int_{a_n}^{b_n} f(x_1,\ldots,x_n) \d x_n\ldots\d x_1\]
for some symmetric function $f$ and any $a_1<b_1\le a_2<b_2\le\ldots\le a_n < b_n$. Let $Y_1,\ldots,Y_n$ be a uniformly random permutation of $X_1,\ldots,X_n$. Then
\[\P(Y_1 > y_1,\ldots,Y_n>y_n,\, Y_i\neq Y_j \,\forall i\neq j) = \frac{1}{n!}\int_{y_1}^1\cdots\int_{y_n}^1 f(x_1,\ldots,x_n) \d x_n\ldots\d x_1.\]
\end{lem}

\begin{proof}
First note that, via a standard limiting procedure, for any $c_1,\ldots,c_n\in[0,1]$,
\[\P(X_1>c_1,\ldots, X_n>c_n, \, X_i\neq X_j\,\forall i\neq j) = \int_{c_1}^1\cdots\int_{c_n}^1 f(x_1,\ldots,x_n)\ind_{\{x_1<\ldots<x_n\}} \d x_n\ldots\d x_1.\]
We now deviate from our usual notation by temporarily letting $S_n$ be the symmetric group on $n$ objects. Then
\begin{align*}
&\P(Y_1 > y_1,\ldots,Y_n>y_n,\, Y_i\neq Y_j \,\forall i\neq j)\\
&\hspace{20mm}= \sum_{\sigma\in S_n} \frac{1}{n!} \P(\sigma(X_1)>y_1,\ldots, \sigma(X_n)>y_n, \, X_i\neq X_j\,\forall i\neq j)\\
&\hspace{20mm}= \frac{1}{n!} \sum_{\sigma\in S_n} \P(X_1>\sigma^{-1}(y_1),\ldots, X_n>\sigma^{-1}(y_n), \, X_i\neq X_j\,\forall i\neq j)\\
&\hspace{20mm}= \frac{1}{n!} \sum_{\sigma\in S_n} \int_{\sigma^{-1}(y_1)}^1 \cdots \int_{\sigma^{-1}(y_n)}^1 f(x_1,\ldots,x_n) \ind_{\{x_1<\ldots<x_n\}} \d x_n\ldots\d x_1\\
&\hspace{20mm}= \frac{1}{n!} \sum_{\sigma\in S_n} \int_{y_1}^1 \cdots \int_{y_n}^1 f(\sigma(x_1),\ldots,\sigma(x_n)) \ind_{\{\sigma(x_1)<\ldots<\sigma(x_n)\}} \d x_n\ldots\d x_1.
\end{align*}
Since $f$ is symmetric, this equals
\[\frac{1}{n!} \sum_{\sigma\in S_n} \int_{y_1}^1 \cdots \int_{y_n}^1 f(x_1,\ldots,x_n) \ind_{\{\sigma(x_1)<\ldots<\sigma(x_n)\}} \d x_n\ldots\d x_1,\]
and since for any $x_1,\ldots,x_n$, exactly one of the permutations in $S_n$ satisfies $\sigma(x_1)<\ldots<\sigma(x_n)$, we get the result.
\end{proof}

Finally, we prove Lemma \ref{momrelax}. This roughly said that we can assume without loss of generality that $\P_T[L^{(j)}] = o(T^{j-2})$ for each $j\ge 3$. More precisely, for each $k\ge 1$, under $\P_T$, there exists a coupling between our Galton-Watson tree with offspring distribution $L$ (and its $k$ chosen particles) and another Galton-Watson tree with offspring distribution $\tilde L$ satisfying
\begin{itemize}
\item $\P_T[\tilde L] = 1+\mu/T + o(1/T)$;
\item $\P_T[\tilde L(\tilde L-1)] = \sigma^2 + o(1)$;
\item there exists a deterministic sequence $J(T)=o(T)$ such that $\P_T(\tilde L = j) = 0$ for all $j\ge J(T)$,
\end{itemize}
such that conditionally on $N_T\ge k$, with probability tending to $1$, the two trees induced by the $k$ chosen particles are equal until time $T$.

\begin{proof}[Proof of Lemma \ref{momrelax}]
We claim that we can choose integers $J(T)$ such that $J(T)=o(T)$ and $\sum_{j=J(T)}^\infty j p_j^{(T)} = o(1/T)$. To see this, first note that for any $\eps>0$, $\sum_{\eps T}^\infty jp_j^{(T)} = o(1/T)$, otherwise $\sum_{\eps T}^\infty j^2 p_j^{(T)}$ is larger than a constant infinitely often, contradicting the uniform integrability of $L^2$.

Choose any sequence $\eps_i\to 0$; by the above, we may choose $t_i\ge i$ such that
\begin{equation}\label{epsieq}
\sum_{\eps_i T}^\infty j p_j^{(T)} < \frac{\eps_i}{T} \hspace{4mm} \forall T\ge t_i.
\end{equation}
Then for any $T$, let $I(T) = \max\{i : t_i\le T\}$ and $J(T) = \lceil\eps_{I(T)}T\rceil$.

Since $I(T)\to\infty$ (because $t_i\ge i$) we have $J(T) = \lceil\eps_{I(T)}T\rceil =o(T)$. But also, by \eqref{epsieq},
\[\sum_{j=J(T)}^\infty j p_j^{(T)} < \frac{\eps_{I(T)}}{T}\]
since $T\ge t_{I(T)}$ by definition of $I(T)$. Therefore $J(T)$ satisfies the claim.

We now choose our distribution $\tilde L$. If $1\le j < J(T)$ then let $\tilde p^{(T)}_j = p^{(T)}_j$. If $j\ge J(T)$ then let $\tilde p^{(T)}_j = 0$. Then choose $\tilde p^{(T)}_0$ so that $\sum_j \tilde p^{(T)}_j = 1$. Let $\tilde L$ satisfy 
\[\P_T(\tilde L = j) = \tilde p^{(T)}_j \hspace{4mm} \forall j\ge 0.\]
We then have
\[\P_T[\tilde L] = \sum_{j=1}^{J(T)-1} j p^{(T)}_j = \P_T[L] - \sum_{j= J(T)}^\infty jp^{(T)}_j = 1+\frac{\mu}{T} + o(1/T)\]
by the claim that we have just proved about $J(T)$, and
\[\P_T[\tilde L(\tilde L-1)] = \sum_{j=2}^{J(T)-1} j(j-1) p^{(T)}_j = \P_T[L(L-1)] - \sum_{j=J(T)}^\infty j(j-1)p^{(T)}_j = \sigma^2 + o(1)\]
by the fact that $L^2$ is uniformly integrable. Therefore $\tilde L$ satisfies the three properties required in the statement of the lemma.

Couple two Galton-Watson trees GW$(L)$ and GW$(\tilde L)$ in the obvious way: if a particle in GW$(L)$ has $j$ children for some $j<J(T)$, then it also has $j$ children in GW$(\tilde L)$. On the other hand, if a particle in GW$(L)$ has $j$ children for some $j\ge J(T)$, then it has no children in GW$(\tilde L)$. The set of particles in GW$(\tilde L)$ is then a subset of those in GW$(L)$ and any particle that exists in GW$(\tilde L)$ has lifetime equal to its counterpart in GW$(L)$. Choose $k$ particles uniformly at random without replacement at time $T$ in GW$(L)$. If they exist in GW$(\tilde L)$ then they are also our chosen particles in GW$(\tilde L)$; if not, then pick $k$ particles uniformly and independently from GW$(\tilde L)$.

The two trees induced by the chosen particles are equal if and only if none of the ancestors of the $k$ chosen particles in GW$(L)$ gave birth to more than $J(T)$ children. By a union bound, it suffices to show that the probability that the first of the $k$ particles has an ancestor that gave birth to more than $J(T)$ particles, conditional on $N_T\ge k$, tends to $0$. From now on we may assume without loss of generality that $\P_T[L]\ge 1$.

For a particle $u\in \mathcal N_T$, let $\Phi_T(u)$ be the event that at least one of the ancestors of $u$ had more than $J(T)$ children. By \eqref{subfirstprop},
\[\P_T\Big[\frac{1}{N_T}\sum_{u\in \mathcal N_T} \ind_{\Phi_T(u)}\,\Big|\, N_T\ge 1 \Big] = \frac{\P_T[N_T]}{\P_T(N_T\ge 1)} \Q^{1,T}_T\Big[\frac{1}{N_T}\ind_{\Phi_T(\xi^1_T)}\Big].\]
By the FKG inequality,
\[\Q^{1,T}_T\Big[\frac{1}{N_T}\ind_{\Phi_T(\xi^1_T)}\Big] \le \Q_T^{1,T}\Big[\frac{1}{N_T}\Big]\Q_T^{1,T}(\Phi_T(\xi^1_T)),\]
so applying \eqref{subfirstprop} again with $f\equiv 1$,
\[\P_T\Big[\frac{1}{N_T}\sum_{u\in \mathcal N_T} \ind_{\Phi_T(u)}\,\Big|\, N_T\ge 1 \Big] \le \frac{\P_T[N_T]}{\P_T(N_T\ge 1)} \Q^{1,T}_T\Big[\frac{1}{N_T}\Big]\Q^{1,T}_T(\Phi_T(\xi^1_T)) = \Q^{1,T}_T(\Phi_T(\xi^1_T)).\]
By Markov's inequality, this is at most the expected number of births of size larger than $J(T)$ along the spine by time $T$ under $\Q^{1,T}_T$; by Lemma \ref{Qbots} (note that since we have only one spine, $\psi_1=\infty$) the births occur as a Poisson point process of rate $r m_T$, and by Lemma \ref{sizebiased} the sizes of the births are size-biased. Thus
\[\Q^{1,T}_T(\Phi_T(\xi^1_T)) \le r m_T T \sum_{j=J(T)}^\infty \frac{j p_j^{(T)}}{m_T}.\]
But $m_T\to 1$ and we chose $J(T)$ such that the sum above is $o(1/T)$; so $\Q^{1,T}_T(\Phi_T(\xi^1_T))\to 0$ and therefore
\[\P_T\Big[\frac{1}{N_T}\sum_{u\in \mathcal N_T} \ind_{\Phi_T(u)}\,\Big|\, N_T\ge 1 \Big] \to 0.\]

We wanted to show that the probability that the first chosen particle has an ancestor that gave birth to more than $J(T)$ particles, conditional on $N_T\ge k$, tends to $0$. We have shown the same statement conditional on $N_T\ge 1$, so it now suffices to show that
\[\P_T(N_T\in\{1,\ldots,k-1\}|N_T\ge 1)\to 0.\]
This follows from the near-critical version of Yaglom's theorem: see \cite[Theorem 2]{fahady_et_al:heavy_traffic_GW}.
\end{proof}

\section*{Acknowledgements}
All three authors are extremely grateful to Amaury Lambert for an in-depth discussion of our paper and related work, including pointing out several helpful references. We also thank an anonymous referee for many very detailed and helpful comments that we feel have significantly improved this article; in particular, for pushing us to give a clearer intuition and probabilistic explanation of our results. 

MR was supported during the early stages of this work by EPSRC fellowship EP/K007440/1, and during the latter stages by a Royal Society University Research Fellowship.
SGGJ was supported during for the major part of this work by University of Bath URS funding.

\bibliographystyle{plain}

\begin{thebibliography}{10}

\bibitem{aldous:CRTI}
David Aldous.
\newblock The continuum random tree. {I}.
\newblock {\em The Annals of Probability}, 19(1):1--28, 1991.

\bibitem{aldous_popovic:critical_bp_biodiversity}
David Aldous and Lea Popovic.
\newblock A critical branching process model for biodiversity.
\newblock {\em Advances in Applied Probability}, 37(04):1094--1115, 2005.

\bibitem{aldous:coalescence_review}
D.J. Aldous.
\newblock Deterministic and stochastic models for coalescence (aggregation and
  coagulation): a review of the mean-field theory for probabilists.
\newblock {\em Bernoulli}, 5(1):3--48, 1999.

\bibitem{athreya:coalescence}
K.B. Athreya.
\newblock Coalescence in critical and subcritical {G}alton-{W}atson branching
  processes.
\newblock {\em Journal of Applied Probability}, 49(3):627--638, 2012.

\bibitem{athreya_ney:branching_processes}
K.B. Athreya and P.E. Ney.
\newblock {\em Branching Processes}.
\newblock Springer-Verlag, New York, 1972.

\bibitem{donnelly_kurtz:particle_measure_popn_models}
P.~Donnelly and T.G. Kurtz.
\newblock Particle representations for measure-valued population models.
\newblock {\em The Annals of Probability}, 27(1):166--205, 1999.

\bibitem{durrett:genealogy}
R~Durrett.
\newblock The genealogy of critical branching processes.
\newblock {\em Stochastic Processes and their Applications}, 8(1):101--116,
  1978.

\bibitem{fahady_et_al:heavy_traffic_GW}
K.S. Fahady, M.P. Quine, and D.~Vere-Jones.
\newblock Heavy traffic approximations for the galton-watson process.
\newblock {\em Advances in Applied Probability}, 3(2):282--300, 1971.

\bibitem{fleischmann_ss:reduced_GW}
K.~Fleischmann and R.~Siegmund-Schultze.
\newblock The structure of reduced critical galton-watson processes.
\newblock {\em Mathematische Nachrichten}, 79(1):233--241, 1977.

\bibitem{gernhard:conditioned_reconstructed}
Tanja Gernhard.
\newblock The conditioned reconstructed process.
\newblock {\em Journal of Theoretical Biology}, 253(4):769--778, 2008.

\bibitem{grosjean_huillet:genealogy_coalescence}
Nicolas Grosjean and Thierry Huillet.
\newblock On the genealogy and coalescence times of
  {B}ienaym{\'{e}}-{G}alton-{W}atson branching processes.
\newblock 2017.
\newblock To appear in \emph{Stochastic Models}. Preprint:
  \texttt{http://arxiv.org/abs/1709.07630}.

\bibitem{harris_roberts:many_to_few}
S.~C. Harris and M.~I. Roberts.
\newblock The many-to-few lemma and multiple spines.
\newblock {\em Annales de l'Institut Henri Poincar{\'e}, Probabilit{\'e}s et
  Statistiques}, 53(1):226--242, 2017.

\bibitem{harris_hesse_kyprianou:bbm_strip}
Simon~C. Harris, Marion Hesse, and Andreas~E. Kyprianou.
\newblock Branching {B}rownian motion in a strip: survival near criticality.
\newblock {\em The Annals of Probability}, 44(1):235--275, 2016.

\bibitem{johnston:coalescence_subcrit_supercrit}
Samuel~G.G. Johnston.
\newblock Coalescence in supercritical and subcritical continuous-time
  {G}alton-{W}atson trees.
\newblock 2017.
\newblock Preprint: \texttt{http://arxiv.org/abs/1709.08500}.

\bibitem{kolmogorov:solution_biological_problem}
Andrei~Nikolaevitch Kolmogorov.
\newblock On the solution of a biological problem.
\newblock {\em Proceedings of Tomsk University}, 2:7--12, 1938.

\bibitem{lambert:coalescence_GW}
Amaury Lambert.
\newblock Coalescence times for the branching process.
\newblock {\em Advances in Applied Probability}, 35(04):1071--1089, 2003.

\bibitem{lambert:contour_splitting_trees}
Amaury Lambert.
\newblock The contour of splitting trees is a {L}{\'e}vy process.
\newblock {\em The Annals of Probability}, 38(1):348--395, 2010.

\bibitem{lambert:genealogy_binary}
Amaury Lambert.
\newblock The genealogy of a sample from a binary branching process.
\newblock {\em arXiv preprint arXiv:1710.02220}, 2017.

\bibitem{lambert:coalescent_branching_trees}
Amaury Lambert and Lea Popovic.
\newblock The coalescent point process of branching trees.
\newblock {\em The Annals of Applied Probability}, 23(1):99--144, 2013.

\bibitem{lambert_stadler:birth_death_cpp}
Amaury Lambert and Tanja Stadler.
\newblock Birth--death models and coalescent point processes: The shape and
  probability of reconstructed phylogenies.
\newblock {\em Theoretical Population Biology}, 90:113--128, 2013.

\bibitem{le:coalescence_GW}
Vi~Le.
\newblock Coalescence times for the {B}ienaym{\'e}-{G}alton-{W}atson process.
\newblock {\em Journal of Applied Probability}, 51(01):209--218, 2014.

\bibitem{lyons_et_al:conceptual_llogl_mean_behaviour_bps}
R.~Lyons, R.~Pemantle, and Y.~Peres.
\newblock Conceptual proofs of {$L\log L$} criteria for mean behavior of
  branching processes.
\newblock {\em Ann. Probab.}, 23(3):1125--1138, 1995.

\bibitem{lyons_peres:probability_on_trees}
Russell Lyons and Yuval Peres.
\newblock {\em Probability on trees and networks}, volume~42.
\newblock Cambridge University Press, 2017.

\bibitem{oconnell:genealogy_mrca}
Neil O'Connell.
\newblock The genealogy of branching processes and the age of our most recent
  common ancestor.
\newblock {\em Advances in {A}pplied {P}robability}, 27(02):418--442, 1995.

\bibitem{popovic:asymptotic_genealogy_critical_bp}
Lea Popovic.
\newblock Asymptotic genealogy of a critical branching process.
\newblock {\em The Annals of Applied Probability}, 14(4):2120--2148, 2004.

\bibitem{ren_song_sun:2spinesuperprocess}
Yan-Xia Ren, Renming Song, and Zhenyao Sun.
\newblock Spine decompositions and limit theorems for a class of critical
  superprocesses.
\newblock 2017.
\newblock Preprint: \texttt{http://arxiv.org/abs/1711.09188}.

\bibitem{ren_song_sun:2spineGW}
Yan-Xia Ren, Renming Song, and Zhenyao Sun.
\newblock A 2-spine decomposition of the critical galton-watson tree and a
  probabilistic proof of yaglom's theorem.
\newblock {\em Electron. Commun. Probab.}, 23:12 pp., 2018.

\bibitem{rogers:guided_tour_excursions}
L.~C.~G. Rogers.
\newblock A guided tour through excursions.
\newblock {\em Bull. London Math. Soc.}, 21:305--341, 1989.

\bibitem{schweinsberg2003coalescent}
Jason Schweinsberg.
\newblock Coalescent processes obtained from supercritical {G}alton--{W}atson
  processes.
\newblock {\em Stochastic Processes and their Applications}, 106(1):107--139,
  2003.

\bibitem{yaglom:certain_limit_thms}
A.~M. Yaglom.
\newblock Certain limit theorems of the theory of branching random processes.
\newblock {\em Doklady Akad. Nauk SSSR (N.S.)}, 56:795--798, 1947.

\end{thebibliography}
\def\cprime{$'$}

\end{document}